\documentclass[a4paper, 10pt]{article}
\usepackage{fancyhdr}
\usepackage[english]{babel}
\usepackage[utf8x]{inputenc}
\usepackage[T1]{fontenc}
\usepackage[a4paper,top=3cm,bottom=2cm,left=3cm,right=3cm,marginparwidth=1.75cm]{geometry}
\usepackage{lipsum}
\usepackage{adjustbox}
\usepackage{tikz-cd}
\usepackage{amsmath}
\usepackage{amscd}
\usepackage{amssymb}
\usepackage{amsfonts}
\usepackage{graphicx}
\usepackage[english]{babel}
\usepackage[utf8x]{inputenc}
\usepackage[T1]{fontenc}
\usepackage{amsthm}
\usepackage{mathtools}
\usepackage{commath}
\usepackage{tikz}
\usetikzlibrary{matrix,arrows}
\usepackage{graphicx}
\usepackage{extpfeil}
\usepackage{titlesec}
\usepackage{accents}
\usepackage[colorinlistoftodos]{todonotes}
\usepackage[colorlinks=true, allcolors=black]{hyperref}
\usepackage{pgfplots,caption}
\pgfplotsset{compat=1.9}

\DeclareMathOperator{\arctanh}{arctanh}

\usepackage{tocloft}

\newcommand{\subsubsubsection}[1]{\paragraph{#1}\mbox{}\\}
\usepackage{titling}
\newcommand*{\dt}[1]{%
   \accentset{\mbox{\large\bfseries .}}{#1}}

\newcommand{\defeq}{\stackrel{\text{def}}{=}}

\pgfkeys{/pgf/declare function={arctanh(\x) = (1/2)*ln((1+\x)/(1-\x));}}
\pgfkeys{/pgf/declare function={Threshold(\x) = (1/(\x*arctanh(sqrt(3/(2*\x^2+3)))) ;}}

\title{Spin(7) Instantons and Hermitian Yang-Mills Connections for the Stenzel Metric}
\author{Vasileios Ektor Papoulias}
\theoremstyle{plain}
\newtheorem{theorem}{Theorem}[section]
\newtheorem{proposition}[theorem]{Proposition}

\newtheorem{corollary}[theorem]{Corollary}

\theoremstyle{definition}

\numberwithin{equation}{section}

\begin{document}
\predate{}
\postdate{}
\date{}
\maketitle
\begin{abstract}
    We use the large isometry group of the Stenzel asymptotically conical Calabi-Yau metric on $T^{\star}S^{4}$ to study the relationship between the Spin(7) instanton and Hermitian-Yang Mills (HYM) equations. We reduce both problems to tractable ODEs and look for invariant solutions. In the abelian case, we establish local equivalence and prove a global nonexistence result. We analyze the nonabelian equations with structure group SO(3) and construct the moduli space of invariant Spin(7) instantons in this setting. This includes an explicit one parameter family of irreducible Spin(7) instantons only one of which is HYM. We thus negatively resolve the question regarding the equivalence of the two gauge theoretic PDEs. The HYM connections play a role in the compactification of this moduli space, exhibiting a phenomenon that we aim to further look into in future work.
\end{abstract}
\tableofcontents
\section{Introduction}
\subsection{Summary}
The inclusion of SU(4) in Spin(7) demonstrates that a Calabi-Yau 4-fold is -in a natural way- a Spin(7) manifold. The SU(4) and Spin(7) structures give rise to associated generalised instanton equations: the \textit{Hermitian Yang-Mills} (HYM) \textit{equations} and the Spin(7) \textit{instanton equations} respectively. It is natural to inquire about the relationship of these two gauge theoretic problems. One immediately observes that HYM is a stronger condition. In the compact case, it is known  that as long as an HYM connection exists, the two types of instantons coincide (Lewis \cite{Lewis}). Consequently, if one hopes to display a compact counterexample to equivalence, there must not be any HYM connections at all. Furthermore, we have a general existence theorem for HYM connections over stable holomorphic bundles (Uhlenbeck, Yau \cite{Uh}). This restricts the choices of bundles one could look at. Finally, compact, irreducible special holonomy manifolds admit no continuous symmetries (Joyce, \cite{Joyce}). This precludes the use of symmetry techniques. We are thus motivated to look for a non-compact counterexample. Since Lewis's argument is essentially an energy estimate, it does not apply to the noncompact setting.

$\newline$
We study a non-compact cohomogeneity one CY $4$-fold: the cotangent bundle of the $4$-sphere equipped with the Stenzel metric (Stenzel \cite{Stenz}). We use the natural cohomogeneity one SO(5)-action to reduce the instanton equations to tractable ODEs and proceed to study the SO(5)-invariant solutions. In section 2, we study the abelian equations. We establish the local equivalence of the two problems and prove a global nonexistence result. In section 3, we study the nonabelian equations corresponding to the structure group SO(3). We classify the relevant cohomogeneity one bundles and their invariant connections and we adjust the extension criterion of Eschenburg and Wang (Eschenburg, Wang\cite{Esch}) to this setting. We construct the full moduli space of SO(5)-invariant Spin(7) instantons with structure group SO(3). This contains solutions living on two distinct bundles which agree outside of a codimension $4$ Cayley submanifold. Each bundle carries a $1$-parameter family of instantons. Each of these families contains precisely one HYM connection in its interior. This negatively resolves the question regarding the equivalence of the two equations. One of the families is a closed interval. The other is a half-open half-closed interval. Its missing endpoint is the (unique) HYM connection on the other bundle. Our example suggests that the HYM connections might play a role in the compactification of Spin(7) instanton moduli spaces (over noncompact CY 4-folds), a phenomenon we intend to further look into in future work.

\subsection{The Stenzel Manifold}
In this section we provide a brief introduction to the Stenzel CY $4$-fold $X^{8}$. For details, we refer to the articles (Stenzel \cite{Stenz}), (Oliveira \cite{Oli}). The latter carries out the corresponding calculations in complex dimension $3$. The overall technique for studying invariant objects in cohomogeneity one is essentially the same as in the article (Lotay-Oliveira \cite{Lot}).
\subsubsection{The Underlying Manifold and the SO(5)-Action}
The underlying space of $X^{8}$ is the cotangent bundle of the $4$-sphere. The natural SO(5) action on $S^{4}$ extends to $T^{\star}S^{4}$ by pullback. The singular orbit is the zero section $S^{4}$. Its stabiliser is the group SO(4). The principal orbits are the positive radius sphere bundles in the metric inherited by $\mathbb{R}^{10}$. Their stabiliser is the group SO(3). They are $7$-dimensional Stiefel manifolds.

$\newline$
The underlying space of $X^{8}$ can be equivalently realised as a complex quadric in $\mathbb{C}^{5}$. It thus inherits a natural complex structure. Consider the degree $2$ homogeneous polynomial:
\begin{equation}\nonumber
F\defeq z_{1}^{2}+...+z_{5}^{2}.
\end{equation}
We split the complex coordinates of $\mathbb{C}^{5}$ into their real and imaginary parts $z_{j}=x_{j}+iy_{j}$ and introduce the functions:
\begin{equation}\nonumber
r^{2}\defeq|z_{1}|^{2}+|z_{2}|^{2}+|z_{3}|^{2}+|z_{4}|^{2}+|z_{5}|^{2},
\end{equation}
\begin{equation}\nonumber
R_{+}^{2}\defeq x_{1}^{2}+x_{2}^{2}+x_{3}^{2}+x_{4}^{2}+x_{5}^{2},\;\;\;R_{-}^{2}\defeq y_{1}^{2}+y_{2}^{2}+y_{3}^{2}+y_{4}^{2}+y_{5}^{2}.
\end{equation}
The following relations follow:
\begin{equation}\nonumber
R_{+}^{2}=\frac{r^{2}+1}{2},\;\; R_{-}^{2}=\frac{r^{2}-1}{2},\;\;r^{2}=R_{+}^{2}+R_{-}^{2}.
\end{equation}
Define the map:
$$\Psi:\mathbb{C}^{5}\to\mathbb{R}^{10},$$
\begin{equation}\label{Diffeo}
(z_{1},...,z_{5})\mapsto \left(\frac{x}{R_{+}},y^{\intercal}\right).
\end{equation}
It may be easily seen that this cuts down to a diffeomorphism:
\begin{equation}
\Psi:F^{-1}(1)\xrightarrow{\sim}T^{\star}S^{4}.
\end{equation}
The minimum value of $r$ on $X^{8}$ is $r=1$ and the associated level set corresponds to the singular orbit. The latter sits inside $X^{8}$ as an embedded totally real submanifold (Patrizio \cite{Pat}). We will denote the principal orbit at radius $r>1$ as $\mathcal{O}_{r}$. We define the reference points:
\begin{equation}\label{SpecialPoint}
p_{r}\defeq\left(R_{+},iR_{-},0,0,0\right)^{\intercal}\in \mathcal{O}_{r},\;\;\;\;p_{1}\defeq\left(1,0,0,0,0\right)^{\intercal}\in S^{4}.
\end{equation}
They form a ray from $p_{1}\in S^{4}$ to infinity. This choice fixes the embeddings of the principal and singular stabiliser groups in SO(5). They are the lower right copies of SO(3) and SO(4) respectively. Furthermore, the choice (\ref{SpecialPoint}) induces a projection map exhibiting $\mathcal{O}_{r}$ as a coset manifold:
\begin{align}
\pi:\text{SO}(5)&\to \mathcal{O}_{r},\nonumber\\
g&\mapsto gp_{r}.\label{OrbitBundleProjection}
\end{align}
The complement of the singular orbit splits as:
\begin{equation}\label{SphereBundle}
T^{\star}S^{4}-S_{4}\cong(0,\infty)\times\frac{\text{SO}(5)}{\text{SO(3)}}.
\end{equation}
We study the adjoint action of SO(5) to obtain a natural frame on $T_{p_{r}}\mathcal{O}_{r}$. The Lie algebra $\mathfrak{so}(5)$ consists of all $5\times 5$ antisymmetric matrices under the commutator bracket. It is given by:
\begin{equation}\nonumber
\mathfrak{so}(5)=\text{Span}\left\{C_{ij}\;|\;1\le i<j\le 5\right\},
\end{equation}
where $C_{ij}=e_{ij}-e_{ji}$ and $e_{ij}$ is the matrix with $ij$ entry equal to $1$ and all other entries vanishing. The bracket is characterized by the relations:
\begin{align}
\left[C_{ij},C_{ik}\right]&=-C_{jk},\label{Bracket1}\\
\left[C_{ij},C_{kl}\right]&=0\text{ for }i\neq j\neq k\neq l.\label{Bracket2}
\end{align}
We introduce the following notation:
\begin{align}
X_{1}&\defeq C_{12},\; X_{2}\defeq C_{13},\; X_{3}\defeq C_{14},\; X_{4}\defeq C_{15},\; X_{5}\defeq C_{23}\nonumber\\
X_{6}&\defeq C_{24},\; X_{7}\defeq C_{25},\; X_{8}\defeq C_{34},\; X_{9}\defeq C_{35},\; X_{10}\defeq C_{45}.\nonumber
\end{align}
Let $\rho$ be the restriction of $\text{Ad}_{\text{SO(5)}}$ to Stab($p_{r}$). An element $g\in \text{SO(3)}$ acts on $A\in\mathfrak{so}(5)$ by conjugation. We split this representation into irreducibles:
\begin{equation}\label{irreps}
\mathfrak{so}(5)=\langle X_{1} \rangle\oplus \langle X_{2},X_{3},X_{4} \rangle\oplus \langle X_{5},X_{6},X_{7} \rangle\oplus \langle X_{10},-X_{9},X_{8} \rangle.
\end{equation}
The first summand is trivial and the other three summands are isomorphic to the vector representation of SO(3) (the order in which the $X_{i}^{'s}$ appear corresponds to the standard basis of $\mathbb{R}^{3}$). The Lie algebra of the stabiliser is given by the third summand. We define the natural reductive complement:
\begin{equation}\label{Reductive}
\mathfrak{m}=\langle X_{1} \rangle\oplus \langle X_{2},X_{3},X_{4} \rangle\oplus \langle X_{5},X_{6},X_{7} \rangle.
\end{equation}
Owing to (\ref{irreps}), it is stable under $\rho$ yielding the isotropy representation. Using (\ref{Diffeo}) and (\ref{OrbitBundleProjection}) we find that:
\begin{align}\label{projection1}
d\pi_{|_{\text{Id}}}:\mathfrak{m}&\xrightarrow{\sim}T_{p_{r}}\mathcal{O}_{r},\nonumber\\
A&\mapsto\left(R_{+}c_{1}(A),-R_{-}r_{2}(A)\right).
\end{align}
Here $c_{1}(\cdot)$ denotes the operation of taking the first column and $r_{2}(\cdot)$ denotes the operation of taking the second row. Using \ref{projection1} we obtain the equations:
\begin{align}
& d\pi_{|_{\text{Id}}}X_{1}=-R_{+}\partial_{x^{2}_{|_{p_{r}}}}+R_{-}\partial_{y^{1}_{|_{p_{r}}}},\label{projection2}\\
& d\pi_{|_{\text{Id}}}X_{2}=-R_{+}\partial_{x^{3}_{|_{p_{r}}}},\label{projection3}\\
& d\pi_{|_{\text{Id}}}X_{3}=-R_{+}\partial_{x^{4}_{|_{p_{r}}}},\label{projection4}\\
& d\pi_{|_{\text{Id}}}X_{4}=-R_{+}\partial_{x^{5}_{|_{p_{r}}}},\label{projection5}\\
& d\pi_{|_{\text{Id}}}X_{5}=-R_{-}\partial_{y^{3}_{|_{p_{r}}}},\label{projection6}\\
& d\pi_{|_{\text{Id}}}X_{6}=-R_{-}\partial_{y^{4}_{|_{p_{r}}}},\label{projection7}\\
& d\pi_{|_{\text{Id}}}X_{7}=-R_{-}\partial_{y^{5}_{|_{p_{r}}}}.\label{projection8}
\end{align}
It is evident from (\ref{projection1})-(\ref{projection4}) that (for r=1) $X_{1},X_{2},X_{3},X_{4}$ correspond to infinitesimal motions in the horizontal direction along the base $S^{4}$. Similarly, (\ref{projection5})-(\ref{projection8}) demonstrate that (for r>1) $X_{5},X_{6},X_{7}$ correspond to infinitesimal vertical motions along the fiber $S^{3}$ of the sphere bundle.

$\newline$
To obtain a basis of $T_{p_{r}}X^{8}$ we need to combine $X_{1},...,X_{7}$ with a radial vector:
\begin{proposition}
There exists a unique smooth vector field $\partial_{r}$ on $X^{8}-S^{4}$ characterised by the following properties:
\begin{enumerate}
    \item{The vector field $\partial_{r}$ is tangent to $(0,\infty)$ in the splitting \ref{SphereBundle}}.
    \item{$dr\left(\partial _{r}\right)=1$}.
\end{enumerate}
Let $(x,y)\in X^{8}\subset\mathbb{C}^{5}$. The vector field $\partial_{r}$ can be expressed as follows in terms of the standard coordinate vector fields on $\mathbb{C}^{5}$:
\begin{equation}\label{radial}
    \partial_{r_{|_{(x,y)}}}=\frac{r}{2R_{+}^{2}}\left(\sum_{j=1}^{5}x^{j}\partial_{x^{j}_{|_{(x,y)}}}\right)+\frac{r}{2R_{-}^{2}}\left(\sum_{j=1}^{5}y^{j}\partial_{y^{j}_{|_{(x,y)}}}\right).
\end{equation}
\end{proposition}
Evaluating the expression (\ref{radial}) at $p_{r}$ we obtain:
\begin{equation}\label{radialp}
\partial_{r}=\frac{r}{2R_{+}}\partial_{x^{1}}+\frac{r}{2R_{-}}\partial_{y^{2}}.
\end{equation}
Over $p_{r}$, tensors can be written as linear combinations of tensor products of $X_{i}$, $\partial_{r}$ and the dual coframe $\theta^{i},\; dr$. The tensor in question is invariant if and only if it is stabilised by the isotropy action. In that case, the basis expansion at $p_{r}$ is well defined over $X^{8}-S^{4}$. Of all the vectors in our frame, only $\partial_{r}$ and $X_{1}$ satisfy this condition. From here on, we will suppress application of $d\pi$ and evaluation at $p_{r}$. Using (\ref{projection2}-\ref{radialp}) we conclude that:
\begin{align}
& dx^{1}=\frac{r}{2R_{+}}dr,\;\;\; dy^{1}=R_{-}\theta^{1},\label{dxdyatp1}\\
& dx^{2}=-R_{+}\theta^{1},\;\;\; dy^{2}=\frac{r}{2R_{-}}dr,\label{dxdyatp2}\\
& dx^{3}=-R_{+}\theta^{2},\;\;\; dy^{3}=-R_{-}\theta^{5},\label{dxdyatp3}\\
& dx^{4}=-R_{+}\theta^{3},\;\;\; dy^{4}=-R_{-}\theta^{6},\label{dxdyatp4}\\
& dx^{5}=-R_{+}\theta^{4},\;\;\; dy^{5}=-R_{-}\theta^{7}.\label{dxdyatp5}
\end{align}
Using equations (\ref{dxdyatp1}-\ref{dxdyatp5}) we obtain:
\begin{align}
&dz^{1}=\frac{r}{2R_{+}}dr+iR_{-}\theta^{1},\label{dzatp1}\\
&dz^{2}=-R_{+}\theta^{1}+i\frac{r}{2R_{-}}dr,\label{dzatp2}\\
&dz^{3}=-R_{+}\theta^{2}-iR_{-}\theta^{5},\label{dzatp3}\\
&dz^{4}=-R_{+}\theta^{3}-iR_{-}\theta^{6},\label{dzatp4}\\
&dz^{5}=-R_{+}\theta^{4}-iR_{-}\theta^{7}.\label{dzatp5}
\end{align}
\subsubsection{The Stenzel Metric}
Since the second cohomology group vanishes, any K\"{a}hler structure comes from a global K\"{a}hler potential $\mathcal{F}(r^{2})$. The associated K\"{a}hler form will then be:
\begin{equation}\nonumber
\omega=\frac{i}{2}\partial\overline{\partial}\mathcal{F}(r^{2}).
\end{equation}
We introduce the functions:
\begin{equation}\label{P}
P(r)\defeq\frac{r}{2}\left(\frac{R_{+}}{R_{-}}+\frac{R_{-}}{R_{+}}\right)\mathcal{F}'(r^{2})+2rR_{+}R_{-}\mathcal{F}''(r^{2}),
\end{equation}
\begin{equation}\label{Q}
Q(r)\defeq R_{+}R_{-}\mathcal{F}'(r^{2}).
\end{equation}
A short calculation demonstrates that:
\begin{equation}\label{InvKahlerForm}
\omega=P(r)dr\wedge\theta^{1}+Q(r)\left(\theta^{25}+\theta^{36}+\theta^{47}\right).
\end{equation}
The volume form of the resulting metric looks like:
\begin{equation}\label{omegaVol}
\text{Vol}_{\omega}=\frac{\omega^{4}}{4!}=-PQ^{3}dr\wedge\theta^{1234567}.
\end{equation}
The complex structure $J$ can be written in terms of invariant forms:
\begin{align}
JX_{1}&=-\frac{2R_{+}R_{-}}{r}\partial_{r}, \;\; J\partial_{r}=\frac{r}{2R_{+}R_{-}}X_{1},\label{CplxX1}\\
&JX_{2}=\frac{R_{+}}{R_{-}}X_{5}, \;\;\; JX_{5}=-\frac{R_{-}}{R_{+}}X_{2},\label{CplxX2}\\
&JX_{3}=\frac{R_{+}}{R_{-}}X_{6}, \;\;\; JX_{6}=-\frac{R_{-}}{R_{+}}X_{3},\label{CplxX3}\\
&JX_{4}=\frac{R_{+}}{R_{-}}X_{7}, \;\;\; JX_{7}=-\frac{R_{-}}{R_{+}}X_{4}.\label{CplxX4}
\end{align}
Formulae (\ref{InvKahlerForm}), (\ref{CplxX1})-(\ref{CplxX4}) can be used to obtain the associated Riemannian metric:
\begin{align}
g&=\frac{rP}{2R_{+}R_{-}}dr\otimes dr+\frac{2R_{+}R_{-}P}{r}\theta^{1}\otimes\theta^{1}\label{InvMetric}\\
&+\frac{R_{+}Q}{R_{-}}\left(\theta^{2}\otimes\theta^{2}+\theta^{3}\otimes\theta^{3}+\theta^{4}\otimes\theta^{4}\right)\nonumber\\
&+\frac{R_{-}Q}{R_{+}}\left(\theta^{5}\otimes\theta^{5}+\theta^{6}\otimes\theta^{6}+\theta^{7}\otimes\theta^{7}\right).\nonumber
\end{align}
Among the invariant K\"{a}hler structures discussed so far, precisely one is Calabi-Yau: the Stenzel metric. We begin with a simple proposition characterizing the canonical bundle of $X^{8}$:
\begin{proposition}
The bundle $K_{X^{8}}$ is holomorphically trivial.
\end{proposition}

\begin{proof}
Let $S_{i}\subset\mathbb{C}^{5}$ be the open subset where $z_{i}\neq 0$. Introduce the following $(4,0)$-form on $S_{i}$:
\begin{equation}\label{HolVolForm}
    {\Omega_{i}}\defeq\frac{1}{z^{i}}dz^{i+1}\wedge dz^{i+2}\wedge...\wedge dz^{i-1}.
\end{equation}
Here the indices are reduced mod $5$. It is easily checked that the forms $\iota_{X^{8}}^{\star}\Omega_{i}$ glue to a global holomorphic volume form on $X^{8}$.
\end{proof}
The trivialization $\Omega$ is easily written in terms of invariant forms:
\begin{align}
\mathfrak{Re}(\Omega) &=R_{+}^{3}\theta^{1234}-R_{+}R_{-}^{2}\left(\theta^{1267}+\theta^{1537}+\theta^{1564}\right)\nonumber\\
&+\frac{r}{2}dr\wedge\left(R_{+}\left(\theta^{237}+\theta^{264}+\theta^{534}\right)-\frac{R_{-}^{2}}{R_{+}}\theta^{567}\right),\label{InvRealOmega}
\end{align}

\begin{align}
\mathfrak{Im}(\Omega)=&-R_{-}^{3}\theta^{1567}+R_{+}^{2}R_{-}\left(\theta^{1237}+\theta^{1264}+\theta^{1534}\right)\nonumber\\
&+\frac{r}{2}dr\wedge\left(R_{-}\left(\theta^{267}+\theta^{537}+\theta^{564}\right)-\frac{R_{+}^{2}}{R_{-}}\theta^{234}\right).\label{InvImaginaryOmega}
\end{align}
The associated volume form is then given by:
\begin{equation}\label{OmegaVol}
\text{Vol}_{\Omega}=(-1)^{\frac{n(n-1)}{2}}\left(\frac{i}{2}\right)^{n}\Omega\wedge\overline{\Omega}=-\frac{r}{2}R_{+}^{2}R_{-}^{2}dr\wedge\theta^{1234567}. 
\end{equation}
The Calabi-Yau equation is equivalent to volume compatibility:
\begin{equation}\nonumber
\text{Vol}_{\omega}=\text{Vol}_{\Omega}.
\end{equation}
Using (\ref{omegaVol}) and (\ref{OmegaVol}), we discover that this amounts to an ODE for $\mathcal{F}(r^{2})$:
\begin{equation}\label{CalabiYauEquation}
PQ^{3}=\frac{r}{2}R_{+}^{2}R_{-}^{2}.
\end{equation}
Equation (\ref{CalabiYauEquation}) can be solved explicitly. We obtain:
\begin{equation}\label{PCY}
P(r)=\left(\frac{3}{4}\right)^{\frac{3}{4}}\frac{r(r^{2}+1)}{(r^{2}+2)^{\frac{3}{4}}(r+1)^{\frac{1}{2}}(r-1)^{\frac{1}{2}}},
\end{equation}
 \begin{equation}\label{QCY}
Q(r)=\frac{1}{2}\left(\frac{4}{3}\right)^{\frac{1}{4}}(r^{2}+2)^{\frac{1}{4}}(r+1)^{\frac{1}{2}}(r-1)^{\frac{1}{2}}.
 \end{equation}
 We compute the pointwise norms of the vectors in the standard framing:
 \begin{equation}\label{X1lengthCY}
|X_{1}|^{2}=\left(\frac{3}{4}\right)^{\frac{3}{4}}\frac{(r^{2}+1)^{\frac{3}{2}}}{(r^{2}+2)^{\frac{3}{4}}},
\end{equation}

\begin{equation}\label{RadiallengthCY}
|\partial_{r}|^{2}=\left(\frac{3}{4}\right)^{\frac{3}{4}}\frac{r^{2}(r^{2}+1)^{\frac{1}{2}}}{(r^{2}+2)^{\frac{3}{4}}(r+1)(r-1)},
\end{equation}

\begin{align}
&|X_{2}|^{2}=|X_{3}|^{2}=|X_{4}|^{2}=\frac{1}{2}\left(\frac{4}{3}\right)^{\frac{1}{4}}(r^{2}+1)^{\frac{1}{2}}(r^{2}+2)^{\frac{1}{4}},\label{X234lengthCY}\\
&|X_{5}|^{2}=|X_{6}|^{2}=|X_{7}|^{2}=\frac{1}{2}\left(\frac{4}{3}\right)^{\frac{1}{4}}\frac{(r^{2}+2)^{\frac{1}{4}}(r+1)(r-1)}{(r^{2}+1)^{\frac{1}{2}}}.\label{X567lengthCY}
\end{align}
As $r\to 1$, $|\partial_{r}|^{2}$ blows up monotonically, $|X_{1}|^{2}$, $|X_{2}|^{2}$, $|X_{3}|^{2},|X_{4}|^{2}$ approach $1$ and $|X_{5}|^{2},|X_{6}|^{2},|X_{7}|^{2}$ tend to $0$. Over $S^{4}$, the kernel of the projection map (\ref{projection1}) extends to $\mathfrak{so}(4)$ so that $X_{5},X_{6},X_{7}$ project to $0$. Consequently, the decay of their norms holds for any smooth metric. We can pull $g$ back to the singular orbit $S^{4}$ to find that it is round of unit radius.

$\newline$
The Cayley Calibration of a CY 4-fold is given by (Salamon-Walpuski \cite{Sal} p.81):
\begin{equation}\label{CayleyCalib}
\Phi=\frac{\omega^{2}}{2}+\mathfrak{Re}(\Omega).
\end{equation}
Using (\ref{CayleyCalib}), we obtain:
\begin{align}\label{InvCayley}
\Phi&=dr\wedge\left[PQ\left(\theta^{125}+\theta^{136}+\theta^{147}\right)+\frac{rR_{+}}{2}\left(\theta^{237}+\theta^{264}+\theta^{534}\right)-\frac{rR_{-}^{2}}{2R_{+}}\theta^{567}\right]\nonumber\\
&+R_{+}^{3}\theta^{1234}-R_{+}R_{-}^{2}\left(\theta^{1267}+\theta^{1537}+\theta^{1564}\right)+Q^{2}\left(\theta^{2536}+\theta^{2547}+\theta^{3647}\right).
\end{align}
When we pull $\Phi$ back to the singular orbit $S^{4}$, only the $\theta^{1234}$ term survives. On $S^{4}$ we have $r=1$. We therefore find that:
\begin{equation}\label{Calibrated}
\iota_{S^{4}}^{\star}\Phi=\theta^{1234}.
\end{equation}
We conclude that the singluar orbit is calibrated for $\Phi$ and is therefore a Cayley submanifold of the Spin(7) manifold $(X^{8},\Phi)$. As such, it is volume minimizing in its homology class (Joyce \cite{Joyce}).

\newpage
\section{SO(5)-Invariant Instantons with Structure Group U(1)}
\subsection{Cohomogeneity One Bundles with Structure Group U(1)}\label{AbelianBundles}
Let $r>1$. The homogeneous U(1) bundles over the orbit $\mathcal{O}_{r}$ correspond to element-conjugacy (i.e. conjugation by a fixed element in the target) classes of Lie group homomorphisms:
\begin{equation}\label{HomogU(1)Bundles}
\lambda:\text{SO}(3)\to\text{U}(1).
\end{equation}
For the classification of homogeneous bundles and invariant connections see (Wang \cite{wang}, Oliveira \cite{Oli} Section 3.1, Lotay-Oliveria \cite{Lot} Section 2.4). Our notation and conventions agree with the latter.

$\newline$
Since the target is abelian, the element-conjugacy relation is trivial: the classes are singletons. The only map of type (\ref{HomogU(1)Bundles}) is $\phi=1$. Consequently, the only homogeneous U(1) bundle over $\mathcal{O}_{r}$ -up to equivariant princpal bundle isomorphism- is the trivial one:
\begin{equation}
P_{1}=\mathcal{O}_{r}\times \text{U}(1)=\frac{\text{SO}(5)}{\text{SO}(3)}\times\text{U}(1).
\end{equation}
SO(5)- invariant U(1)-connections on $P_{1}$ are parameterised by representation morphisms:
\begin{equation}
\Lambda:\left(\mathfrak{m},\text{Ad}_{\text{SO}(5)_{|_{\text{SO}(3)}}}\right)\to\left(\mathfrak{u}(1),\text{Ad}_{\text{U}(1)}\circ\lambda\right)=\left(i\mathbb{R},1\right).
\end{equation}
Recalling the decomposition (\ref{Reductive}) and applying Schur's lemma, we obtain that:
\begin{equation}
\text{Hom}_{\text{SO}(3)}\left(\mathfrak{m},\mathfrak{u}(1)\right)=i\mathbb{R}.
\end{equation}
Here, the imaginary number $i\alpha$ corresponds to:
\begin{equation}
\Lambda_{\alpha}\defeq i\alpha\theta^{1}.
\end{equation}
The cohomogeneity one bundle over $X^{8}-S^{4}$ associated to $P_{1}$ is obtained by pulling back along the map:
\begin{equation}\nonumber
X^{8}-S^{4}\xrightarrow{\sim}(1,\infty)\times\frac{\text{SO(5)}}{\text{SO(3)}}\twoheadrightarrow\frac{\text{SO(5)}}{\text{SO(3)}}.
\end{equation}
We slightly abuse notation by suppressing the pullback symbol and denoting the resulting bundle by $P_{1}$. It is trivial and it admits a unique extension across the singular orbit given by $X^{8}\times\text{U}(1)$.

$\newline$
Connections over $X^{8}-S^{4}$ can be put in temporal gauge through an equivariant gauge transformation (Lotay-Oliveira \cite{Lot} p. 21, Remark 5). Consequently, each invariant connection on $P_{1}$ is equivariantly gauge equivalent to one lying in the space:
\begin{equation}\label{U(1)InvConnection}
    \mathcal{A}_{\text{inv}}\left(P_{1}\right)=\left\{i \alpha(r)\theta^{1} \;|\;a\in C^{\infty}\left(0,\infty\right)\right\}\subset\mathcal{A}\left(P_{1}\right).
\end{equation}
These connections can only be related by an $r$-independent gauge transformation. If such a gauge transformation is equivariant, it is given by a fixed element of U(1) and it stabilises all connections. It follows that no two distinct elements of $\mathcal{A}_{\text{inv}}\left(P_{1}\right)$ are equivariantly gauge equivalent.

$\newline$
We compute the curvature of $A\in\mathcal{A}\left(P_{1}\right)$:
\begin{align}\label{curvature}
F_{A}&=dA\nonumber\\
&=i\frac{d\alpha}{dr}dr\wedge\theta^{1}+i\alpha(r)d\theta^{1}.
\end{align}
To simplify the second term we use the Maurer-Cartan relations (Kobayashi, Nomizu \cite{Kob} p. 41). For this calculation we require the structure constants of $\mathfrak{so}(5)$. They can be computed using (\ref{Bracket1}) and (\ref{Bracket2}). Carrying out the calculation gives:
\begin{equation}\nonumber
d\theta^{1}=\theta^{25}+\theta^{36}+\theta^{47}.
\end{equation}
Incorporating this into (\ref{curvature}), we obtain:
\begin{equation}\label{U(1)InvCurvature}
F_{A}=i\frac{d\alpha}{dr}dr\wedge\theta^{1}+i\alpha(r)\left(\theta^{25}+\theta^{36}+\theta^{47}\right).
\end{equation}
The Ambrose-Singer holonomy theorem implies that any non-flat U(1) connection is irreducible. Consequently, all elements of $\mathcal{A}_{\text{inv}}\left(P_{1}\right)$ -excluding the trivial connection- are irreducible.
\subsection{The SO(5)-Invariant ODEs}
The Spin(7) instanton equation reads:
\begin{equation}\label{Spin(7)InstantonEquation}
    \star_{g}F_{A}=-\Phi\wedge F_{A}.
\end{equation}
Since the metric diagonalises we have:
\begin{equation}\label{HodgeStar}
\star_{g}\theta^{i_{1}}\wedge...\wedge\theta^{i_{k}}=\frac{\sqrt{\det(g)}}{g_{i_{1}i_{1}}...g_{i_{k}i_{k}}}\theta^{i_{k+1}}\wedge...\wedge\theta^{i_{n}}.
\end{equation}
Here, $i_{1},...,i_{n}$ is an even permutation of $1,...,n$. Using (\ref{HodgeStar}) and (\ref{InvMetric}) we compute:
\begin{align}
&\star_{g}dr\wedge\theta^{1}=-\frac{Q^{3}}{P}\theta^{234567},\nonumber\\
&\star_{g}\theta^{25}=-PQ dr\wedge\theta^{13467},\nonumber\\
&\star_{g}\theta^{36}=-PQ dr\wedge\theta^{12356}.\nonumber
\end{align}
Incorporating these in (\ref{U(1)InvCurvature}), we obtain:
\begin{equation}\label{HodgeStarCurvature}
\star_{g}F_{A}=-i\frac{Q^{3}}{P}\frac{d\alpha}{dr}\theta^{234567}-iPQ\alpha dr\wedge\left(\theta^{13467}+\theta^{12457}+\theta^{12356}\right).
\end{equation}
We now use (\ref{InvCayley}) and (\ref{U(1)InvCurvature}) to compute:
\begin{equation}\label{PhiWedgeCurvature}
\Phi\wedge F_{A}=-3iQ^{2}\alpha(r)\theta^{234567}-i\left(Q^{2}\frac{d\alpha}{dr}+2PQ\alpha(r)\right)dr\wedge\left(\theta^{13467}+\theta^{12457}+\theta^{12356}\right).
\end{equation}
Imposing (\ref{Spin(7)InstantonEquation}) and comparing coefficients gives two equations. These are the same and read:
\begin{equation}\label{U(1)Spin(7)ODE}
\frac{d\alpha}{dr}=-3\frac{P}{Q}\alpha.
\end{equation}

$\newline$
The Hermitian Yang-Mills equations are:
\begin{align}
 F_{A}\wedge\star\omega&=0,\label{HYMEquation1}\\
 F_{A}\wedge\Omega&=0\label{HYMEquation2}.
\end{align}
We find that (\ref{HYMEquation2}) holds identically. This can be seen by direct computation using (\ref{InvRealOmega}), (\ref{InvImaginaryOmega}) and (\ref{U(1)InvCurvature}). It follows that an SO(5)-invariant U(1)-connection $A$ is HYM if and only if (\ref{HYMEquation1}) holds.

$\newline$
Over a Hermitian manifold of complex dimension $n$, we have:
\begin{equation}\label{U(1)StarOmega}
\star_{g}\omega=\frac{\omega^{n-1}}{(n-1)!}.
\end{equation}
Using (\ref{InvKahlerForm}), we compute:
\begin{equation}\label{CubeInvKahlerForm}
\omega^{3}=6PQ^{2}dr\wedge\left(\theta^{12536}+\theta^{12547}+\theta^{13647}\right)+6Q^{3}\theta^{253647}.
\end{equation}
Using (\ref{U(1)StarOmega}), (\ref{CubeInvKahlerForm}) and (\ref{U(1)InvCurvature}) we calculate:
\begin{align}
F_{A}\wedge\star\omega&=F_{A}\wedge\frac{\omega^{3}}{3!}\nonumber\\
&=-i\left(Q^{3}\frac{d\alpha}{dr}+3PQ^{2}\alpha(r)\right)dr\wedge\theta^{1234567}.\nonumber
\end{align}
It follows that an SO(5)-invariant U(1)-connection is HYM if and only if:
\begin{equation}\label{U(1)HYMODE}
\frac{d\alpha}{dr}=-3\frac{P}{Q}\alpha.
\end{equation}
We observe that this equation is the same as (\ref{U(1)Spin(7)ODE}).

$\newline$
Using the uniqueness part of the standard Picard theorem, we obtain:
\begin{theorem}
An \emph{SO(5)}-invariant \emph{U(1)}-connection over $X^{8}-S^{4}$ equipped with the Stenzel Calabi-Yau structure is a \emph{Spin(7)} instanton if and only if it is Hermitian-Yang-Mills.
\end{theorem}
\subsection{Explicit Solution and Breakdown Near the Singular Orbit}
We study the ODE (\ref{U(1)HYMODE}). Using (\ref{PCY}) and (\ref{QCY}) we write it as:
\begin{equation}\label{U(1)ODE}
\frac{da}{dr}=-\frac{9}{2}\frac{r(r^{2}+1)}{(r^{2}+2)(r+1)(r-1)}\alpha(r).
\end{equation}
We integrate (\ref{U(1)ODE}) directly to see that the solution takes the following form for some $K\in\mathbb{R}$:
\begin{equation}\label{U(1)sol}
\alpha(r)=\frac{K}{(r^{2}+2)^{\frac{3}{4}}(r+1)^{\frac{3}{2}}(r-1)^{\frac{3}{2}}}.
\end{equation}
An elementary calculation yields:
\begin{equation}\label{U(1)solDer}
\frac{da}{dr}=-\frac{9K}{2}\frac{r(r^{2}+1)}{(r^{2}+2)^{\frac{7}{4}}(r+1)^{\frac{5}{2}}(r-1)^{\frac{5}{2}}}.
\end{equation}
Recalling the formulae (\ref{U(1)InvConnection}) and (\ref{U(1)InvCurvature}) and incorporating (\ref{U(1)sol}) and (\ref{U(1)solDer}), we obtain:
\begin{theorem}\label{U(1)P1InstantonClassificationTheorem}
Let $M=X^{8}-S^{4}$ be equipped with the Stenzel Calabi-Yau structure. Let $P_{1}$ be the unique cohomogeneity one \emph{U(1)}-bundle over $M$ (i.e. the trivial bundle). There exists a one-parameter family of smooth \emph{SO(5)}-invariant \emph{Spin(7)} instantons $A_{K}\in\mathcal{A}_{\text{inv}}\left(P_{1}\right)$:
\begin{equation}
A_{K}=\frac{iK}{(r^{2}+2)^{\frac{3}{4}}(r+1)^{\frac{3}{2}}(r-1)^{\frac{3}{2}}}\theta^{1}\;\text{ where }K\in\mathbb{R}.
\end{equation}
All elements of this family are HYM. They are all irreducible -apart from the trivial connection- and no two of them are gauge equivalent. Furthermore, these are all the invariant \emph{Spin(7)} instantons on $P_{1}$.

$\newline$
The curvature of $A_{K}$ is given by:
\begin{equation}\label{U(1)InstantonCurv}
F_{A_{K}}=iK\left(-\frac{9}{2}\frac{r(r^{2}+1)}{(r^{2}+2)^{\frac{7}{4}}(r+1)^{\frac{5}{2}}(r-1)^{\frac{5}{2}}}dr\wedge\theta^{1}+\frac{\theta^{25}+\theta^{36}+\theta^{47}}{(r^{2}+2)^{\frac{3}{4}}(r+1)^{\frac{3}{2}}(r-1)^{\frac{3}{2}}}\right).
\end{equation}
\end{theorem}
We need to justify our claim that these solutions lie in distinct gauge equivalence classes. We have seen that this holds when we quotient by the group of equivariant gauge transformations. The claim follows from the fact that gauge equivalent, irreducible, invariant connections are automatically equivariantly gauge equivalent.

$\newline$
Using (\ref{X1lengthCY}), (\ref{RadiallengthCY}), (\ref{X234lengthCY}) and (\ref{X567lengthCY}), we find that:
\begin{equation}\label{AbelianBlowupRate}
    \norm{F_{A}}_{g}^{2}=O\left(|r-1|^{-4}\right)\text{ as }r\to 1.
\end{equation}
In particular:
\begin{equation}\nonumber
    \lim_{r\to 1}\norm{F_{A}}_{g}^{2}=+\infty.
\end{equation}
Since the metric extends smoothly to the singular orbit, this behaviour is precluded for connections that are smooth over the whole space. We therefore obtain the following global nonexistence result:
\begin{theorem}
There exist no global, abelian, \emph{SO(5)}-invariant \emph{Spin(7)} instantons/ HYM connections on the Stenzel manifold $X^{8}$ apart from the trivial connection $A=0$ (corresponding to $K=0$).
\end{theorem}
As a closing remark, we note that breakdown around Cayley $4$-folds is an interesting feature of the Spin(7)-instanton equation. It is related to the non-compactness of the moduli space. In Donaldson theory, noncompactness occurs in the form of a sequence of ASD instantons failing to have a limit due to finitely many point singularities. In the $8$-dimensional Spin(7) setting, points are typically replaced by four-dimensional Cayley submanifolds.

$\newline$
The non-existence result we encountered has to do with abelian gauge theory being too coarse to capture the behaviour we would like to see. In the following section we study the nonabelian equations associated to the structure group SO(3). The nonlinearity induced by the noncommutativity of the group \textit{smoothes} the equations and we are able to obtain solutions that extend over the singular orbit $S^4$.

\newpage
\section{SO(5)-Invariant Instantons with Structure Group SO(3): Preliminaries and Analysis on the Trivial Bundle}
\subsection{Cohomogeneity One Bundles with Structure Group SO(3)}
\subsubsection{Bundles and Bundle Extensions}
Let $r>1$. The homogeneous SO(3) bundles over the orbit $\mathcal{O}_{r}$ correspond to element-conjugacy classes of Lie group homomorphisms:
\begin{equation}
\lambda:\text{SO}(3)\to\text{SO}(3).
\end{equation}
There are two such classes. They are represented by the trivial map and the identity respectively. Consequently, there are precisely two homogeneous principal SO(3) bundles over $\mathcal{O}_{r}$ -up to equivariant principal bundle isomorphism. We denote these by $P_{1}$ and $P_{\text{Id}}$. Slightly abusing notation, we also denote by $P_{1}$ and $P_{\text{Id}}$ the pullbacks of the respective bundles along the map:
\begin{equation}\nonumber
X^{8}-S^{4}\xrightarrow{\sim}(1,\infty)\times\frac{\text{SO(5)}}{\text{SO(3)}}\twoheadrightarrow\frac{\text{SO(5)}}{\text{SO(3)}}.
\end{equation}
We now classify smooth homogeneous extensions of $P_{1}$ and $P_{\text{Id}}$ across the singular orbit $S^{4}$. These correspond to element-conjugacy classes of Lie group homomorphisms:
\begin{equation}\label{homomorphSO(3)ext}
    \mu:\text{SO}(4)\to\text{SO(3)}.
\end{equation}
Once such a map is chosen, one uses it to form the associated homogeneous bundle $P_{\mu}$ over $S^{4}$. The extension is then determined by pulling $P_{\mu}$ back over $X^{8}$ through the natural projection:
\begin{equation}\nonumber
X^{8}\cong T^{\star}S^{4}\twoheadrightarrow S^{4}.
\end{equation}
The element-conjugacy class of the restriction of $\mu$ to the lower diagonal copy of SO(3) determines which bundle is being extended.

$\newline$
We are therefore required to classify element-conjugacy classes of homomorphisms of type (\ref{homomorphSO(3)ext}). Natural representatives are described by passing through the respective universal covers. We have the two-sheeted covering maps:
\begin{align}
    &\pi_{\text{Spin}(4)}:\text{Sp}(1)^{2}\twoheadrightarrow \text{SO}(4),\nonumber\\
    &\pi_{\text{Spin}(3)}:\text{Sp}(1)\twoheadrightarrow \text{SO}(3);\nonumber
\end{align}
where:
\begin{align}
    &\pi_{\text{Spin}(4)}(x,y):\mathbb{H}\to\mathbb{H},\nonumber\\
    &\;\;\;\;\;\;\;\;\;\;\;\;\;\;\;\;\;\;\;\;\;\;\;\;q\mapsto xqy^{-1},\label{Spin(4)}\\
    &\pi_{\text{Spin}(3)}(x)=\pi_{\text{Spin}(4)}(x,x)_{|_{\mathfrak{Im}(\mathbb{H})}}.\label{Spin(3)}
\end{align}
It is clear that the upper left copy of $\text{SO}(3)$ inside $\text{SO}(4)$ corresponds to the anti-diagonal copy of Sp(1) in $\text{Sp}(1)^{2}$, while the lower right copy of SO(3) corresponds to the diagonal one. Considering (\ref{Spin(4)}) and (\ref{Spin(3)}), we obtain:
\begin{equation}\nonumber
    \text{SO}(4)=\frac{\text{Sp(1)}^{2}}{\left\{(1,1),(-1,-1)\right\}},\;\;\;\text{SO}(3)=\frac{\text{Sp(1)}}{\pm 1}.
\end{equation}
There are precisely three element-conjugacy classes of homomorphisms of type (\ref{homomorphSO(3)ext}), one of them being that of the trivial map. The two nontrivial classes are represented by the two projections
\begin{equation}\nonumber
    \pi_{1},\pi_{2}:\text{SO}(4)=\frac{\text{Sp(1)}^{2}}{\left\{(1,1),(-1,-1)\right\}}\twoheadrightarrow\frac{\text{Sp(1)}}{\pm 1}\times \frac{\text{Sp(1)}}{\pm 1}\twoheadrightarrow\frac{\text{Sp(1)}}{\pm 1}=\text{SO}(3).
\end{equation}
Hence, there are precisely three principal SO(3)-bundles of cohomogeneity-one over $X^{8}$. We denote these as $P_{1}$, $P_{\pi_{1}}$ and $P_{\pi_{2}}$. The first is the trivial bundle. It extends the trivial bundle on $X^{8}-S^{4}$. The other two bundles are nontrivial (see section \ref{BundNotTriv}). They provide distinct extensions of $P_{\text{Id}}$.

\subsubsection{Invariant Connections on the Complement of the Singular Orbit}
We now classify the invariant connections on the bundles $P_{1}$ and $P_{\text{Id}}$ over $X^{8}-S^{4}$. For each of these connections we compute the associated curvature tensor in terms of the standard framing.

$\newline$
In general, an invariant connection $A\in\mathcal{A}_{\text{inv}}\left(P_{\lambda}\right)$ corresponds to a map of representations:
\begin{equation}\label{InvConSO(3)}
    \Lambda:\left(\mathfrak{m},\text{Ad}_{\text{SO}(5)_{|_{\text{SO}(3)}}}\right)\to \left(\mathfrak{so}(3),\text{Ad}_{\text{SO}(3)}\circ\lambda\right).
\end{equation}
Given such a map, we use the canonical invariant connection $d\lambda$ as a reference and write:
\begin{equation}
A=d\lambda+\Lambda.
\end{equation}
We first deal with $P_{1}$. In this case $\lambda=1$ and the target representation is trivial. Recalling the splitting (\ref{Reductive}) and applying Schur's lemma, we see that $\Lambda$ must take the form:
\begin{equation}\label{LambdSO(3)triv}
\Lambda = \theta^{1}\otimes\left(a^{1}\;e_{1}+a^{2}\;e_{2}+a^{3}\;e_{3}\right).
\end{equation}
The canonical invariant connection $A^{can}_{1}$ is represented by $d\;1=0$. Evidently, it is flat.

$\newline$
As in section \ref{AbelianBundles}, any connection over $X^{8}-S^{4}$ can be brought to temporal gauge by an equivariant gauge transformation. It follows that any invariant connection on $P_{1}$ is equivariantly gauge equivalent to one lying in the space:
\begin{equation}\label{InvConSO(3)triv}
    \mathcal{A}_{\text{inv}}\left(P_{1}\right)=\left\{\theta^{1}\otimes\left(a^{1}(r)\;e_{1}+a^{2}(r)\;e_{2}+a^{3}(r)\;e_{3}\right)\;|\;a^{1},a^{2},a^{3}\;\in C^{\infty}\left(0,\infty\right)\right\}.
\end{equation}
A gauge transformation relating two elements of $\mathcal{A}_{\text{inv}}\left(P_{1}\right)$ must be $r$-independent. If it is equivariant, it is given by a fixed element of SO(3) acting on $\mathcal{A}_{\text{inv}}\left(P_{1}\right)$ by conjugation. It follows that the elements of $\mathcal{A}_{\text{inv}}\left(P_{1}\right)$ need not lie in distinct equivariant gauge equivalence classes.

$\newline$
A calculation analogous to the one in section $2.1$ yields:
\begin{equation}\label{CurvSO(3)Triv}
 F_{A}=\left(\frac{da^{i}}{dr}dr\wedge\theta^{1}+a^{i}\left(\theta^{25}+\theta^{36}+\theta^{47}\right)\right)\otimes e_{i}.  
\end{equation}
The Ambrose-Singer holonomy theorem implies that the elements of $\mathcal{A}\left(P_{1}\right)$ need not be irreducible. This happens -for instance- if one of the components $a^{i}$ vanishes identically.

$\newline$
We now work on $P_{\text{Id}}$. In this case, the target representation is the adjoint representation of SO(3). Recalling the decomposition (\ref{Reductive}) and applying Schur's lemma, we see that equivariant maps of type (\ref{InvConSO(3)}) always vanish on the first summand and either restrict to isomorphisms or the zero map on the second and third summands. The automorphisms of $\text{Ad}_{\text{SO}(3)}$ are given by multiplication by fixed scalars. We concude that for $\lambda=\text{Id}$, maps of type (\ref{InvConSO(3)}) look like:
\begin{equation}\nonumber
    \Lambda=a\left(\theta^{2}\otimes e_{3}-\theta^{3}\otimes e_{2}+\theta^{4}\otimes e_{1}\right)+b\left(\theta^{5}\otimes e_{3}-\theta^{6}\otimes e_{2}+\theta^{7}\otimes e_{1}\right),\;\;a,b\in\mathbb{R}.
\end{equation}
The canonical invariant connection $A^{\text{can}}_{\text{Id}}=d\;\text{Id}_{\text{SO}(3)}$ on $P_{\text{id}}$ takes the form:
\begin{equation}\nonumber
    A^{\text{can}}_{\text{Id}}=\theta^{8}\otimes e_{1}+\theta^{9}\otimes e_{2}+\theta^{10}\otimes e_{3}.
\end{equation}
It is not flat. Its curvature is given by:
\begin{align}
    F_{A^{\text{can}}_{\text{Id}}}&=dA^{\text{can}}_{\text{Id}}+\frac{1}{2}\left[A^{\text{can}}_{\text{Id}}\wedge A^{\text{can}}_{\text{Id}}\right]\nonumber\\
    &=\left(\theta^{23}+\theta^{56}\right)\otimes e_{1} + \left(\theta^{24}+\theta^{57}\right)\otimes e_{2}+\left(\theta^{34}+\theta^{67}\right)\otimes e_{3}.\nonumber
\end{align}
The radial component of an invariant tensorial $1$-form is an invariant section of the adjoint bundle. In this context, these objects correspond to fixed points of $\text{Ad}_{\text{SO(3)}}$. This representation has no fixed points, implying that all invariant connections are already in temporal gauge. Consequently, the space of invariant connections is given by:
\begin{align}
    &\mathcal{A}_{\text{inv}}\left(P_{1}\right)=\nonumber\\
    &\left\{A^{\text{can}}_{\text{Id}}+a(r)\left(\theta^{2}e_{3}-\theta^{3} e_{2}+\theta^{4}e_{1}\right)+b(r)\left(\theta^{5}e_{3}-\theta^{6}e_{2}+\theta^{7} e_{1}\right)|a,b\in C^{\infty}\left(0,\infty\right)\right\}.\label{InvConSO(3)Id}
\end{align}
Equivariant gauge transformations correspond to central elements of SO(3). Since SO(3) is centerless, the only possibility is the identity. Consequently, each invariant connection constitutes its own equivariant gauge equivalence class.

$\newline$
To compute the curvature of a general element $A=A^{\text{can}}_{\text{Id}}+\Lambda\in\mathcal{A}_{inv}(P_{\text{Id}})$,
we use the formula:
\begin{equation}\nonumber
F_{A}=F_{A^{\text{can}}_{\text{Id}}}+d_{A^{\text{can}}_{\text{Id}}}\Lambda+\frac{1}{2}\left[\Lambda\wedge\Lambda\right].
\end{equation}
Routine calculation yields:
\begin{align}
    d_{A^{\text{can}}_{\text{Id}}}\Lambda&=d\Lambda+\left[A^{\text{can}}_{\text{Id}}\wedge\Lambda\right]=\nonumber\\
    &\left(b\;\theta^{14}-a\;\theta^{17}\right)\otimes e_{1}+\left(a\;\theta^{16}-b\;\theta^{13}\right)\otimes e_{2}+\left(b\;\theta^{12}-a\;\theta^{15}\right)\otimes e_{3}\nonumber\\
    &+\frac{da}{dr}dr\wedge\left(\theta^{2}\otimes e_{3}-\theta^{3}\otimes e_{2}+\theta^{4}\otimes e_{1}\right)+\frac{db}{dr}dr\wedge\left(\theta^{5}\otimes e_{3}-\theta^{6}\otimes e_{2}+\theta^{7}\otimes e_{1}\right).
\end{align}
The final summand is also easily seen to take the form:
\begin{align}
    \frac{1}{2}\left[\Lambda\wedge\Lambda\right]=&\left(-a^{2}\;\theta^{23}-ab\;\theta^{26}+ab\; \theta^{35}-b^{2}\;\theta^{56}\right)\otimes e_{1}\nonumber\\
    &\left(-a^{2}\;\theta^{24}-ab\;\theta^{27}+ab\; \theta^{45}-b^{2}\;\theta^{57}\right)\otimes e_{2}\nonumber\\
    &\left(-a^{2}\;\theta^{34}-ab\;\theta^{37}+ab\; \theta^{46}-b^{2}\;\theta^{67}\right)\otimes e_{3}.
\end{align}
Overall, we obtain the following expression for the curvature:
\begin{align}\label{InvCurvSO(3)Id}
    F_{A}&=\\
    &\;\;\;\; \left((1-a^{2})\;\theta^{23}+(1-b^{2})\;\theta^{56}-ab\; \theta^{26}+ab\; \theta^{35}+b\;\theta^{14}-a\;\theta^{17}+\frac{da}{dr}dr\wedge\theta^{4}+\frac{db}{dr}dr\wedge\theta^{7}\right)\otimes e_{1}\nonumber\\
     &+\left((1-a^{2})\;\theta^{24}+(1-b^{2})\;\theta^{57}-ab\; \theta^{27}+ab\; \theta^{45}-b\;\theta^{13}+a\;\theta^{16}-\frac{da}{dr}dr\wedge\theta^{3}-\frac{db}{dr}dr\wedge\theta^{6}\right)\otimes e_{2}\nonumber\\
      &+\left((1-a^{2})\;\theta^{34}+(1-b^{2})\;\theta^{67}-ab\; \theta^{37}+ab\; \theta^{46}+b\;\theta^{12}-a\;\theta^{15}+\frac{da}{dr}dr\wedge\theta^{2}+\frac{db}{dr}dr\wedge\theta^{5}\right)\otimes e_{3}.\nonumber
\end{align}
The Ambrose-Singer holonomy theorem implies that all elements of $\mathcal{A}_{\text{inv}}\left(P_{\text{Id}}\right)$ are irreducible. Since gauge equivalent, irreducible, invariant connections are equivariantly gauge equivalent, the elements of $\mathcal{A}_{\text{inv}}\left(P_{\text{Id}}\right)$ all lie in distinct gauge equivalence classes.

\subsubsection{Invariant Connections on the Extended Bundles}
It remains to understand how to describe invariant connections on the extensions of $P_{1}$ and $P_{\text{Id}}$ over ${S^{4}}$. For $P_{1}$ this is easy. The unique extension is given by the trivial bundle. The canonical invariant connection is still equal to zero. It follows that the $\text{ad}(P_{1})$-valued forms \ref{InvConSO(3)triv} are still meaningful over the extended bundle and describe the relevant invariant connections with this choice of reference.

$\newline$
The situation is slightly more subtle for $P_{\text{Id}}$. The canonical invariant connection of $P_{\text{Id}}$ disagrees with those of $P_{\pi_{1}}$ and $P_{\pi_{2}}$. In fact, $A^{\text{can}}_{\text{Id}}$ does not even extend to a connection on either of these bundles. To see this, we compute the canonical invariant connections of $P_{\pi_{1}}$, $P_{\pi_{2}}$. These are given by:
\begin{align}
    A_{\pi_{1}}^{\text{can}}&=d\pi_{1}\nonumber\\
    &=\left(\theta^{8}+\theta^{7}\right)\otimes e_{1}+\left(\theta^{9}-\theta^{6}\right)\otimes e_{2}+\left(\theta^{10}+\theta^{5}\right)\otimes e_{3}\nonumber\\
    &=A_{\text{Id}}^{\text{can}}+\left(\theta^{5}\otimes e_{3}-\theta^{6}\otimes e_{2}+\theta^{7}\otimes e_{1}\right),\label{CanonConPi1}
\end{align}
\begin{align}
    A_{\pi_{2}}^{\text{can}}&=d\pi_{2}\nonumber\\
    &=\left(\theta^{8}-\theta^{7}\right)\otimes e_{1}+\left(\theta^{9}+\theta^{6}\right)\otimes e_{2}+\left(\theta^{10}-\theta^{5}\right)\otimes e_{3}\nonumber\\
    &=A_{\text{Id}}^{\text{can}}-\left(\theta^{5}\otimes e_{3}-\theta^{6}\otimes e_{2}+\theta^{7}\otimes e_{1}\right).\label{CanonConPi2}
\end{align}
Let $A\in \mathcal{A}_{\text{inv}}(P_{\text{Id}})$ be an invariant connection over $X^{8}-S^{4}$. Then:
\begin{equation}
    A=A^{\text{can}}_{\text{Id}}+a(r)\left(\theta^{2}\otimes e_{3}-\theta^{3}\otimes  e_{2}+\theta^{4}\otimes e_{1}\right)+b(r)\left(\theta^{5}\otimes e_{3}-\theta^{6}\otimes e_{2}+\theta^{7}\otimes  e_{1}\right).
\end{equation}
We rewrite it using $A_{\pi_{1}}^{\text{can}}$ and $A_{\pi_{2}}^{\text{can}}$ as the reference. We obtain:
\begin{equation}\label{Pi1InvCon}
    A=A^{\text{can}}_{\pi_{1}}+a(r)\left(\theta^{2}\otimes e_{3}-\theta^{3}\otimes  e_{2}+\theta^{4}\otimes e_{1}\right)+\left(b(r)-1\right)\left(\theta^{5}\otimes e_{3}-\theta^{6}\otimes e_{2}+\theta^{7}\otimes  e_{1}\right),
\end{equation}
\begin{equation}\label{Pi2InvCon}
    A=A^{\text{can}}_{\pi_{2}}+a(r)\left(\theta^{2}\otimes e_{3}-\theta^{3}\otimes  e_{2}+\theta^{4}\otimes e_{1}\right)+\left(b(r)+1\right)\left(\theta^{5}\otimes e_{3}-\theta^{6}\otimes e_{2}+\theta^{7}\otimes  e_{1}\right).
\end{equation}
The forms $\theta^{5}$, $\theta^{6}$ and $\theta^{7}$ blow up as $r\to 1$. We conclude that a necessary condition for $A$ to extend to $P_{\pi_{1}}$ is:
\begin{equation}\label{ContPi1}
    \lim_{r\to 1}b(r)=1.
\end{equation}
Similarly, if $A$ extends to $P_{\pi_{2}}$ we have:
\begin{equation}\label{ContPi2}
    \lim_{r\to 1}b(r)=-1.
\end{equation}
The connection $A^{\text{can}}_{\text{Id}}$ corresponds to $a=b=0$. Both conditions (\ref{ContPi1}) and (\ref{ContPi2}) fail. Consequently $A^{\text{can}}_{\text{Id}}$ does not extend to either $P_{\pi_{1}}$ or $P_{\pi_{2}}$.
\subsection{The SO(5)-Invariant ODEs on \texorpdfstring{$P_{1}$}{text}: Derivation and Explicit Solution}
We proceed to study the Spin(7) instanton and Hermitian-Yang Mills equations on the bundle $P_{1}$. A general invariant connection $A$ is equivariantly gauge equivalent to one of the form (\ref{InvConSO(3)triv}). The associated curvature tensor $F_{A}$ is given by (\ref{CurvSO(3)Triv}). These expressions are manifestly similar to (\ref{U(1)InvConnection}) and (\ref{U(1)InvCurvature}). An identical computation to the one carried out in the abelian case gives:
\begin{equation}\nonumber
F_{A}\wedge\Omega=0,\nonumber
\end{equation}
\begin{equation}\nonumber
F_{A}\wedge\star\omega=-\left(Q^{3}\frac{da^{i}}{dr}+3PQ^{2}a^{i}(r)\right)dr\wedge\theta^{1234567}\otimes e_{i}\nonumber.
\end{equation}
Consequently, the invariant Hermitian Yang-Mils equations take the form:
\begin{equation}\label{HYMODEP1}
\frac{da^{i}}{dr}=-3\frac{P}{Q}a^{i}.
\end{equation}
Using (\ref{CayleyCalib}) and computing as in the abelian case we obtain:
\begin{equation}\nonumber
\Phi\wedge F_{A}=-\left[3Q^{2}a^{i}(r)\theta^{234567}+\left(Q^{2}\frac{da^{i}}{dr}+2PQ\alpha(r)\right)dr\wedge\left(\theta^{13467}+\theta^{12457}+\theta^{12356}\right)\right]\otimes e_{i},
\end{equation}
\begin{equation}\nonumber
\star_{g}F_{A}=-\left[\frac{Q^{3}}{P}\frac{da^{i}}{dr}\theta^{234567}+PQa^{i} dr\wedge\left(\theta^{13467}+\theta^{12457}+\theta^{12356}\right)\right]\otimes e_{i}.
\end{equation}
Equating these expressions yields the invariant Spin(7) instanton ODEs. They are identical to (\ref{HYMODEP1}).  We thus obtain the following local equivalence result:
\begin{theorem}
An \emph{SO}(5)-invariant \emph{SO}(3)-connection $A\in\mathcal{A_{\text{inv}}}\left(P_{1}\right)$ over $T^{\star}S^{4}-S^{4}$ equipped with the Stenzel Calabi-Yau structure is a \emph{Spin(7)} instanton if and only if it is Hermitian-Yang-Mills.
\end{theorem}
The ODE (\ref{HYMODEP1}) has already been studied in the context of the abelian equations. We immediately obtain an analogous existence/classification result:
\begin{theorem}\label{P1InstantonClassificationTheorem}
Let $M=X^{8}-S^{4}$ be equipped with the Stenzel Calabi-Yau structure.  There is a $3$-parameter family of invariant \emph{Spin(7)} instantons $A_{K^{1},K^{2},K^{3}}\in\mathcal{A}_{\text{inv}}\left(P_{1}\right)$:
\begin{equation}
A_{K^{1},K^{2},K^{3}}=\frac{\theta^{1}}{(r^{2}+2)^{\frac{3}{4}}(r+1)^{\frac{3}{2}}(r-1)^{\frac{3}{2}}}\otimes K^{i}e_{i}\text{ where }K^{i}\in\mathbb{R}.\nonumber
\end{equation}
They are all HYM. Any invariant \emph{Spin(7)} instanton on $P_{1}$ is equivariantly gauge equivalent to some element of this family.

$\newline$
The curvature of $A_{K^{1},K^{2},K^{3}}$ is given by:
\begin{equation}
F_{A_{K^{1},K^{2},K^{3}}}=\left(-\frac{9}{2}\frac{r(r^{2}+1)}{(r^{2}+2)^{\frac{7}{4}}(r+1)^{\frac{5}{2}}(r-1)^{\frac{5}{2}}}dr\wedge\theta^{1}+\frac{\theta^{25}+\theta^{36}+\theta^{47}}{(r^{2}+2)^{\frac{3}{4}}(r+1)^{\frac{3}{2}}(r-1)^{\frac{3}{2}}}\right)\otimes K^{i}e_{i}.\nonumber
\end{equation}
\end{theorem}
The Ambrose-Singer holonomy theorem implies that all the (non-zero) instantons of theorem \ref{P1InstantonClassificationTheorem} are reducible. In fact, they have holonomy U(1). Let $A\neq 0$ be one of them. Let $Q\subset P_{1}$ denote the trivial subbundle with fiber U(1) obtained by exponentiating a nonzero vector in the holonomy algebra. $A$ restricts to an irreducible connection on $Q$. The resulting instanton is one of those promised by theorem \ref{U(1)P1InstantonClassificationTheorem}. This makes rigorous the apparent similarities with the situation in section $2$.

$\newline$
Unless $K^{1}=K^{2}=K^{3}=0$, the curvature norm of $A_{K^{1},K^{2},K^{3}}$ is unbounded as $r\to 1$. We thus obtain the following global non-existence result:
\begin{theorem}\label{NonexistenceTrivialBundle}
There are no global \emph{SO(5)}-invariant \emph{Spin(7)} instantons/ HYM connections on the trivial \emph{SO(3)}- bundle $P_{1}$ over $X^{8}$ -apart from the trivial connection A=0.
\end{theorem}

\subsection{Extendibility of Connections Across the Singular Orbit}
It remains to study the equations on $P_{\text{Id}}$. This is the content of section \ref{FinalSection}. As we shall see, $P_{\text{Id}}$ admits solutions that are well behaved near $S^{4}$. As a preliminary step, we need to understand when a general $A\in\mathcal{A}_{\text{inv}}\left(P_{\text{Id}}\right)$ arises as the restriction of a global connection (either on $P_{\pi_{1}}$ or $P_{\pi_{2}}$). Our task is to formulate the criterion of Eschenburg and Wang (Eschenburg-Wang \cite{Esch}) in the context of gauge theory on $X^{8}$.
\subsubsection{Extendibility of Tensorial Forms: the Criterion of Eschenburg and Wang}
Let $S$ be a Lie group and let $\mu$ be a Lie group homomorphism:
\begin{equation}\nonumber
\mu:\text{SO}(4)\to S.
\end{equation}
Denote by $\lambda$ the restriction of $\mu$ to the bottom right copy of SO(3). Let $P$ be the cohomogeneity one principal $S$-bundle over $X^{8}$ whose restrictions over each orbit $\mathcal{O}_{r}$ ($r\ge 1$) are given by:
\[ 
P_{|_{\mathcal{O}_{r}}}= \left\{
\begin{array}{ll}
      \text{SO}(5)\times_{\left(\text{SO}(4),\mu\right)}S\text{ if }r=1\\
      \text{SO}(5)\times_{\left(\text{SO}(3),\lambda\right)}S\text{ if }r>1.
\end{array} 
\right. 
\]
Let $(V,\rho)$ be a representation of the group $S$. We can form the associated vector bundle over $X^{8}$:
\begin{equation}\nonumber
    \rho\left(P\right)\defeq P\times_{\rho}V.
\end{equation}
We consider the problem of extending SO(5)-invariant $\rho\left(P\right)$-valued k-forms across the singular orbit $S^{4}$.

$\newline$
Eschenburg and Wang give necessary and sufficient conditions for extending invariant linear tensors across the singular orbit of a cohomogeneity one space. Since we are interested in bundle-valued forms, their technique does not apply directly. We resolve this issue by passing to the total space $P$ and working with $V$-valued forms instead. In this section we set up the requisite framework to implement this idea.

$\newline$
The manifold $P$ is a cohomogeneity one space for the group $\text{SO(5)}\times S$. Its principal orbits are isomorphic to $P_{\lambda}$ and its singular orbit is the bundle $P_{\mu}$.

$\newline$
Define the reference points:
\[ 
x_{r}\defeq \left\{
\begin{array}{ll}
      [1,1]\in P_{\mu}\text{ if }r=1\\
      \left[1,1\right]\in P_{\lambda}\text{ if }r>1.
\end{array} 
\right. 
\]
With this definition, the point $x_{r}$ lies in the fiber above $p_{r}\in \mathcal{O}_{r}$ for all $r\ge 1$.

$\newline$
Using these reference points, the isotropy subgroups corresponding to the principal and singular orbits are respectively given by:
\begin{equation}\label{PStabPrincipal}
\text{Stab}\left(x_{r}\right)=\left\{(h,\lambda(h))\in\text{SO(5)}\times S \text{ such that }h\in \text{SO}(3)\right\}\cong\text{SO}(3),
\end{equation}
\begin{equation}\label{PStabSingular}
\text{Stab}\left(x_{1}\right)=\left\{(h,\mu(h))\in \text{SO(5)}\times S \text{ such that }h\in \text{SO}(4)\right\}\cong \text{SO}(4).
\end{equation}
In formulae (\ref{PStabPrincipal}) and (\ref{PStabSingular}), SO(4) and SO(3) denote the bottom right inclusions of these groups in SO(5). In what follows, when we consider the action of SO(4) on $P$, it will be through its embedding in $\text{SO}(5)\times S$ as the singular isotropy group (\ref{PStabSingular}).

$\newline$
Let $\omega$ be an invariant, tensorial form of type $\rho$. Its extendibility can be decided by studying the restriction $\omega_{|_{W}}$ along a particularly simple embedded submanifold $W\subset P$. This will make the problem tractable. Let $W$ be the union of the SO(4)-orbits of all points $x_{r}$ in $P$:
\begin{equation}\label{Wdfn2}
W\defeq\bigcup_{r\ge 1}\text{Stab}\left(x_{1}\right)\;x_{r}.
\end{equation}
This is a 4-dimensional linear SO(4)-representation. The SO(4)-action is obvious. The linear structure is inherited from $T^{\star}S^{4}_{p_{1}}$ through the projection map:
\begin{equation}\nonumber
\pi:P\twoheadrightarrow X^{8}.
\end{equation}
In particular, $\pi$ restricts to a diffeomorphism:
\begin{equation}\label{PiWdiffeo} 
  \pi:W\xrightarrow{\sim}T^{\star}_{p_{1}}S^{4}\subset X^{8}.
\end{equation}
The latter is a smoothly embedded submanifold of $X^{8}$ stable under the action of SO(4). The equivariance of $\pi$, implies that W and $T^{\star}_{p_{1}}S^{4}$ are isomorphic SO(4)-representations. Since $T^{\star}_{p_{1}}S^{4}$ is a vector space, it can be naturally identified with the tangent space at its origin (e.g. by the exponential map of the underlying additive group). Endowing the latter with the isotropy action, this identification becomes equivariant. These considerations allow us to view W as the vector representation of SO(4):
 \begin{equation}\label{WIdentified}
 W\cong\left<\partial_{y^{2}},\partial_{y^{3}},\partial_{y^{4}},\partial_{y^{5}}\right>.
\end{equation}
The extendibility problem for invariant tensors is addressed by examining their restrictions along $W$. We are thus interested in finding a useful way to describe these restrictions. Pull the bundle
$$\Lambda^{k}T^{\star}P\otimes \underline{V}$$
back to $W$ using the inclusion map. Since $W$ is linear, the pullback is trivial. We will now give a particular trivialization that elucidates the action of SO(4). Using the canonical invariant connection to decompose $TP$ into vertical and horizontal distributions, we obtain an equivariant identification:
\begin{equation}\label{PrincipalBundleTangentDecomp}
TP\cong \pi^{\star}TX^{8}\oplus\underline{\mathfrak{s}}.
\end{equation}
Furthermore, there is an obvious SO(4)-equivariant trivialization:
\begin{equation}\label{VectorBundleTangentDecomp}
    TX^{8}_{|_{\pi(W)}}\cong \pi(W)\times\left(\left<X_{1},...,X_{4}\right>\oplus \left<\partial_{y^{2}},\partial_{y^{3}},\partial_{y^{4}},\partial_{y^{5}}\right>\right).
\end{equation}
Putting these together we have:
\begin{equation}
\left(\Lambda^{k}T^{\star}P\otimes \underline{V}\right)_{|_{W}}\cong W\times\left(\Lambda^{k}\left<X_{1},...,X_{4}\right>^{\star}\otimes V\oplus \Lambda^{k}\left<\partial_{y^{2}},\partial_{y^{3}},\partial_{y^{4}},\partial_{y^{5}}\right>^{\star}\otimes V \oplus\Lambda^{k}\mathfrak{s}^{\star}\otimes V\right).\nonumber
\end{equation}
Here, the action of SO(4) is as follows: The action on $\mathfrak{s}$ is trivial. The action on $V$ is obtained by composing $\mu$ and $\rho$. Finally, the brackets $\left<X_{1},...,X_{4}\right>^{\star}$ and $\left<\partial_{y^{2}},\partial_{y^{3}},\partial_{y^{4}},\partial_{y^{5}}\right>^{\star}$ are vector representations.

$\newline$
We study the restriction of $\omega$ along $W_{0}$: the vector space $W$ punctured at its origin
\begin{equation}\nonumber
W_{0}\defeq W-\left\{x_{1}\right\}.
\end{equation}
Since tensorial forms vanish on vertical vectors, $\omega_{|_{W_{0}}}$ is a section of the trivial bundle with fiber equal to:
\begin{equation}\nonumber
    E\defeq\Lambda^{k}\left<X_{1},...,X_{4}\right>^{\star}\otimes V\oplus \Lambda^{k}\left<\partial_{y^{2}},\partial_{y^{3}},\partial_{y^{4}},\partial_{y^{5}}\right>^{\star}\otimes V.
\end{equation}

Due to the triviality of the bundle, the form $\omega|_{W_{0}}$ amounts to an SO(4)-equivariant function:
\begin{equation}\nonumber
    f:W_{0} \to E.
\end{equation}
The invariance of $\omega$ implies that no information is lost in passing from $\omega$ to $f$. In fact, $\omega$ is determined by the values of $f$ on the reference points $x_{r}$ forming a ray from the origin of $W$ to infinity. This recovers our usual description of invariant forms as curves in a group representation:
\begin{equation}\label{UsualFormforExtension1}
    \omega_{r}:(1,\infty)\to E.
\end{equation}
Eschenburg and Wang prove that the extendibility of $\omega$ is contingent to a representation-theoretic condition on the formal Taylor series expansion of an appropriate reparameterization of $\omega_{r}$. This condition reflects the behaviour of $f$ near $x_{1}$.

$\newline$
The requisite reparameterization is obtained as follows. Using (\ref{WIdentified}), the Euclidean metric on $\mathbb{R}^{10}$ induces an inner product on $W$. We consider the radial function of the associated norm. Concretely, we set:
\begin{equation}\nonumber
    t\defeq R_{-}=\left(\frac{r^{2}-1}{2}\right)^{\frac{1}{2}},\;\;\;r(t)=\left(2t^{2}+1\right)^{\frac{1}{2}}.
\end{equation}
We thus obtain a curve:
\begin{equation}\nonumber
    \gamma(t)\defeq \omega_{r(t)}.
\end{equation}
It is clear that $f$ determines $\gamma$ and vice-versa:
\begin{equation}\nonumber
\gamma(t)=f\left(x_{r(t)}\right).
\end{equation}
The result of Eschenburg and Wang (Eschenburg-Wang \cite{Esch}, Lemma 1.1, p.113) asserts that $\omega$ extends smoothly over the singular orbit if and only if the following hold:
\begin{itemize}
  \item The curve $\gamma$ is smooth from the right at $t=0$
  \item The formal Taylor series of $\gamma$ at $t=0$ can be written as:
  \begin{equation}\nonumber
      \gamma\sim\sum_{k\ge 0}u_{k}\left(x_{|_{t=1}}\right)t^{k}
  \end{equation}
where:
\begin{equation}\nonumber
u_{k}:W\to E
\end{equation}
is a homogeneous equivariant polynomial of degree $k$.
\end{itemize}
Note that we have provided explicit descriptions of the SO(4)-actions on $W$ and $E$. These descriptions facilitate the computations required for applications.
\subsubsection{Application: Extendibility of Connections}
We are interested in studying the extendibility of tensorial forms $\omega$ describing connections on $P$ (relative to the canonical invariant connection). Therefore -in the context of our application- we have:
\begin{equation}\nonumber
    S=\text{SO}(3),\;V=\mathfrak{so}(3)\;,\rho=\text{Ad}_{\text{SO}(3)},\; k=1.
\end{equation}
Given our setup, $\omega$ will usually be available in the form (\ref{UsualFormforExtension1}). Given this data, we need to pass to the associated curve $\gamma(t)$ and express it in a basis of $E$ coming from evaluation of homogeneous equivariant polynomials at $x_{|_{t=1}}\in W$. To achieve this, we need to be able to find appropriate equivariant polynomials. This task can be simplified if we understand the relevant representations in terms of quaternions. To this end, we identify the spaces $W$ and $\left<X_{1},...,X_{4}\right>$ with $\mathbb{H}$ by:
\begin{equation}\nonumber
\left<X^{1},X^{2},X^{3},X^{4}\right>\cong\left<1,i,j,k\right>\cong\left<\partial_{y^{2}},\partial_{y^{3}},\partial_{y^{4}},\partial_{y^{5}}\right>.
\end{equation}
Furthermore, we lift the action of SO(4) to $\text{Sp}(1)^{2}$ using the covering map $\pi_{\text{Spin}(4)}$. Under these identifications, the SO(4)-action is captured by the usual spinor representation of $\text{Sp}(1)^{2}$ on $\mathbb{H}$.

$\newline$
General points $p\in W$ and $q\in \left<X_{1},...,X_{4}\right>$ can be written as:
\begin{align}
    p&=p^{0}X_{1}+p^{1}X_{2}+p^{2}X_{3}+p^{3}X_{4},\;\;\;\;\;\;\;\;q=q^{0}\partial_{y^{2}}+q^{1}\partial_{y^{3}}+q^{2}\partial_{y^{4}}+q^{3}\partial_{y^{5}}\nonumber\\
    &=p^{0}+p^{1}i+p^{2}j+p^{3}k\;\;\;\;\;\;\;\;\;\;\;\;\;\;\;\;\;\;\;\;\;\;\;\;\;=q^{0}+q^{1}i+q^{2}j+q^{3}k.\nonumber
\end{align}
With this choice of coordinates we have:
\begin{equation}\nonumber
    x_{|_{t=1}}=1\in\text{Sp}(1)\subset\mathbb{H}.
\end{equation}
The Lie algebra $\mathfrak{so}(3)$ can be naturally identified with $\mathfrak{sp}(1)=\mathfrak{Im}\left(\mathbb{H}\right)$ using the differential of the covering map $\pi_{\text{Spin}(3)}$. This identification is Ad-equivariant. Explicitly, it takes the following form:
\begin{equation}\nonumber
    \left<e_{1},e_{2},e_{3}\right>\cong\left<-\frac{k}{2},\frac{j}{2},-\frac{i}{2}\right>.
\end{equation}
These considerations demonstrate that we require homogeneous $\text{Sp}(1)^{2}$-equivariant polynomials:
\begin{equation}\nonumber
    u:\mathbb{H}\to\mathbb{H}^{\star}\otimes \mathfrak{Im}\left(\mathbb{H}\right)\oplus\mathbb{H}^{\star}\otimes \mathfrak{Im}\left(\mathbb{H}\right)
\end{equation}
with prescribed value at $x=1$. Separating the components in the target, such maps take the form:
\begin{equation}\nonumber
    u\left(x\right)\left(p,q\right)=u_{1}\left(x\right)\left(p\right)+u_{2}\left(x\right)\left(q\right).
\end{equation}
The $\text{Sp}(1)^{2}$-equivariance condition for $u:W\to E$ translates to the following:
\begin{equation}\label{EquiCon1}
u_{1}\left(ax\overline{b}\right)\left(p\right)=\text{Ad}_{\mu\circ\pi_{\text{Spin}(4)}\left(a,b\right)}u_{1}\left(x\right)\left(\overline{a}pb\right)\text{ for all }\left(a,b\right)\in\text{Sp}(1)^{2},
\end{equation}
\begin{equation}\label{EquiCon2}
u_{2}\left(ax\overline{b}\right)\left(q\right)=\text{Ad}_{\mu\circ\pi_{\text{Spin}(4)}\left(a,b\right)}u_{2}\left(x\right)\left(\overline{a}qb\right)\text{ for all }\left(a,b\right)\in\text{Sp}(1)^{2}.
\end{equation}
\subsubsection*{The Case of $P_{\pi_{1}}$}
In this case $\mu=\pi_{1}$. The action of $\text{Sp}(1)^{2}$ on $\mathfrak{Im}\left(\mathbb{H}\right)$ is given by projecting the group element to the first factor and conjugating by the result. Conditions (\ref{EquiCon1}), (\ref{EquiCon2}) become:
\begin{equation}\nonumber
u_{1}\left(ax\overline{b}\right)\left(p\right)=au_{1}\left(x\right)\left(\overline{a}pb\right)\overline{a}\text{ for all }\left(a,b\right)\in\text{Sp}(1)^{2},
\end{equation}
\begin{equation}\nonumber
u_{2}\left(ax\overline{b}\right)\left(q\right)=au_{2}\left(x\right)\left(\overline{a}qb\right)\overline{a}\text{ for all }\left(a,b\right)\in\text{Sp}(1)^{2}.
\end{equation}

$\newline$
Using (\ref{Pi1InvCon}), we write a general invariant connection over $X^{8}-S^{4}$ as $A^{\text{can}}_{\pi_{1}}+\omega$, where:
\begin{equation}\nonumber
    \omega=a(r(t))\left(\theta^{2}\otimes e_{3}-\theta^{3}\otimes  e_{2}+\theta^{4}\otimes e_{1}\right)+\left(b(r(t))-1\right)\left(\theta^{5}\otimes e_{3}-\theta^{6}\otimes e_{2}+\theta^{7}\otimes  e_{1}\right).
\end{equation}
The form $\theta^{2}\otimes e_{3}-\theta^{3}\otimes  e_{2}+\theta^{4}\otimes e_{1}$ corresponds to:
\begin{align}
\left(p,q\right)&\mapsto -\frac{1}{2}\left(p^{1}i+p^{2}j+p^{3}k\right)\nonumber\\
&=-\frac{1}{2}\mathfrak{Im}(p)\nonumber\\
&=\frac{\left<p,1\right>-p}{2}.\nonumber
\end{align}
Similarly, the form $\theta^{5}\otimes e_{3}-\theta^{6}\otimes e_{2}+\theta^{7}\otimes  e_{1}$ corresponds to:
\begin{align}
\left(p,q\right)&\mapsto-\frac{1}{2t}\left(q^{1}i+q^{2}j+q^{3}k\right)\nonumber\\
&=-\frac{1}{2t}\mathfrak{Im}(q)\nonumber\\
&=\frac{\left<q,1\right>-q}{2t}.\nonumber
\end{align}
Any equivariant polynomial $u$ satisfying:
\begin{equation}\label{Pi1EquivPolyAt1p}
u\left(1\right)\left(p,q\right)=\frac{\left<p,1\right>-p}{2}
\end{equation}
has the following restriction on $S^{3}\subset\mathbb{H}$:
\begin{equation}\label{Pi1EquivPolyp}
u\left(x\right)\left(p\right)=\frac{\left<x,p\right>-p\overline{x}}{2}.
\end{equation}
Similarly, any equivariant polynomial $v$ satisfying:
\begin{equation}\label{Pi1EquivPolyAt1q}
v\left(1\right)\left(p,q\right)=\frac{\left<q,1\right>-q}{2}
\end{equation}
has the following restriction on $S^{3}\subset\mathbb{H}$:
\begin{equation}\label{Pi1EquivPolyq}
v\left(x\right)\left(q\right)=\frac{\left<x,q\right>-q\overline{x}}{2}.
\end{equation}
As soon as $u$ and $v$ are specified on the unit sphere, they are extended on $\mathbb{H}$ by homogeneity. The extensions depend on the degree $d$, which is yet unspecified. Given $d$, we define:
\[ 
u_{d}\left(x\right)\left(p\right)\defeq \left\{
\begin{array}{ll}
      |x|^{d}u\left(\frac{x}{|x|}\right)\left(p\right)\text{ if }x\neq 0\\
      0 \text{ if }x=0,
\end{array} 
\right. 
\]
\[ 
v_{d}\left(x\right)\left(p\right)\defeq \left\{
\begin{array}{ll}
      |x|^{d}v\left(\frac{x}{|x|}\right)\left(p\right)\text{ if }x\neq 0\\
      0 \text{ if }x=0.
\end{array} 
\right. 
\]
The admissible choices of the value of $d$ are constrained. Clearly, we could take $d$ to be $1$. This would correspond to defining $u$ and $v$ by the formulae (\ref{Pi1EquivPolyp}) and (\ref{Pi1EquivPolyq}) on the whole of $\mathbb{H}$. Other choices can only arise by multiplying $u_{1}$ and $v_{1}$ by powers of the homogeneous degree $2$ polynomials $\norm{p}^{2}$ and $\norm{q}^{2}$ respectively. We conclude that for each positive integer:
\begin{equation}
    d=1+2k
\end{equation}
we have precisely one homogeneous equivariant polynomial $u_{d}$ of degree $d$ satisfying (\ref{Pi1EquivPolyAt1p}) and precisely one homogeneous equivariant polynomial $v_{d}$ of degree $d$ satisfying (\ref{Pi1EquivPolyAt1q}).

$\newline$
We rewrite the form $\omega$ as:
\begin{equation}
\omega= a(r(t))\;u_{d}\left(1\right)+\frac{b(r(t))-1}{t}\;v_{d}\left(1\right).
\end{equation}
Applying the criterion of Eschenburg and Wang we obtain:
\begin{proposition}\label{PPi1ExtendCondition}
Let $A\in\mathcal{A}_{\text{inv}}\left(P_{\text{Id}}\right)$ be an invariant connection defined over $X^{8}-S^{4}$. Let
\begin{equation}\nonumber
\omega=a(r(t))\left(\theta^{2}\otimes e_{3}-\theta^{3}\otimes  e_{2}+\theta^{4}\otimes e_{1}\right)+\left(b(r(t))-1\right)\left(\theta^{5}\otimes e_{3}-\theta^{6}\otimes e_{2}+\theta^{7}\otimes  e_{1}\right)
\end{equation}
be the tensorial form expressing $A$ with respect to the canonical invariant connection of $P_{\pi_{1}}$. Then $A$ extends over the singular orbit on $P_{\pi_{1}}$ if and only if the following hold:
\begin{itemize}
    \item The function $a(r(t))$ is smooth from the right at $t=0$, odd and $O\left(t\right)$.
    \item The function $b(r(t))-1$ is smooth from the right at $t=0$, even and $O\left(t^{2}\right)$.
\end{itemize}
\end{proposition}

\subsubsection*{The Case of $P_{\pi_{2}}$}
In this case $\mu=\pi_{2}$. The action of $\text{Sp}(1)^{2}$ on $\mathfrak{Im}\left(\mathbb{H}\right)$ is given by projecting the group element to the second factor and conjugating by the result. Conditions (\ref{EquiCon1}), (\ref{EquiCon2}) become:
\begin{equation}\nonumber
u_{1}\left(ax\overline{b}\right)\left(p\right)=bu_{1}\left(x\right)\left(\overline{a}pb\right)\overline{b}\text{ for all }\left(a,b\right)\in\text{Sp}(1)^{2},
\end{equation}
\begin{equation}\nonumber
u_{2}\left(ax\overline{b}\right)\left(q\right)=bu_{2}\left(x\right)\left(\overline{a}qb\right)\overline{b}\text{ for all }\left(a,b\right)\in\text{Sp}(1)^{2}.
\end{equation}
Using (\ref{Pi2InvCon}), we write a general invariant connection over $X^{8}-S^{4}$ as $A^{\text{can}}_{\pi_{2}}+\omega$, where:
\begin{equation}\nonumber
    \omega=a(r(t))\left(\theta^{2}\otimes e_{3}-\theta^{3}\otimes  e_{2}+\theta^{4}\otimes e_{1}\right)+\left(b(r(t))+1\right)\left(\theta^{5}\otimes e_{3}-\theta^{6}\otimes e_{2}+\theta^{7}\otimes  e_{1}\right).
\end{equation}
Any equivariant polynomial $u$ satisfying:
\begin{equation}\label{Pi2EquivPolyAt1p}
u\left(1\right)\left(p,q\right)=\frac{\left<p,1\right>-p}{2}
\end{equation}
has the following restriction on $S^{3}\subset\mathbb{H}$:
\begin{equation}
u\left(x\right)\left(p\right)=\frac{\left<x,p\right>-\overline{x}p}{2}.
\end{equation}
Similarly, any equivariant polynomial $v$ satisfying:
\begin{equation}\label{Pi2EquivPolyAt1q}
v\left(1\right)\left(p,q\right)=\frac{\left<q,1\right>-q}{2}
\end{equation}
has the following restriction on $S^{3}\subset\mathbb{H}$:
\begin{equation}
v\left(x\right)\left(q\right)=\frac{\left<x,q\right>-\overline{x}q}{2}.
\end{equation}
We thus have precisely one homogeneous equivariant polynomial $u_{d}$ satisfying (\ref{Pi2EquivPolyAt1p}) and precisely one homogeneous equivariant polynomial $v_{d}$ satisfying (\ref{Pi2EquivPolyAt1q}) in each degree $d=1+2k$.

$\newline$
We rewrite the form $\omega$ as:
\begin{equation}
\omega= a(r(t))\;u_{d}\left(1\right)+\frac{b(r(t))+1}{t}\;v_{d}\left(1\right).
\end{equation}
Applying the criterion of Eschenburg and Wang we obtain:
\begin{proposition}\label{PPi2ExtendCondition}
Let $A\in\mathcal{A}_{\text{inv}}\left(P_{\text{Id}}\right)$ be an invariant connection defined over $X^{8}-S^{4}$. Let
$$\omega=a(r(t))\left(\theta^{2}\otimes e_{3}-\theta^{3}\otimes  e_{2}+\theta^{4}\otimes e_{1}\right)+\left(b(r(t))+1\right)\left(\theta^{5}\otimes e_{3}-\theta^{6}\otimes e_{2}+\theta^{7}\otimes  e_{1}\right)$$
be the tensorial form expressing $A$ with respect to the canonical invariant connection of $P_{\pi_{2}}$. Then $A$ extends over the singular orbit on $P_{\pi_{2}}$ if and only if the following hold:
\begin{itemize}
    \item The function $a(r(t))$ is smooth from the right at $t=0$, odd and $O\left(t\right)$.
    \item The function $b(r(t))+1$ is smooth from the right at $t=0$, even and $O\left(t^{2}\right)$.
\end{itemize}
\end{proposition}
\section{SO(5)-Invariant Instantons with Structure Group SO(3): Analysis on \texorpdfstring{$P_{\text{Id}}$}{text} }\label{FinalSection}
In the remainder, we will study the invariant Spin(7) instanton and HYM equations on the bundle $P_{\text{Id}}$. The relevant analysis constitutes the heart of the present article.
\subsection{The SO(5)-Invariant HYM ODEs on \texorpdfstring{$P_{\text{Id}}$}{text}}
We begin with the HYM equations for which the situation is significantly simpler. 
\subsubsection{Derivation}
A general invariant connection $A\in\mathcal{A}_{\text{inv}}\left(P_{\text{Id}}\right)$ takes the form (\ref{InvConSO(3)Id}). The associated curvature tensor $F_{A}=F_{A}^{j}\otimes e_{j}$ is given in (\ref{InvCurvSO(3)Id}). The Hodge dual of the K\"{a}hler form has been computed in (\ref{CubeInvKahlerForm}). We observe that:
\begin{equation}
F_{A}\wedge\star_{g}\omega=0.
\end{equation}
Consequently, the HYM equations reduce to (\ref{HYMEquation2}). We write:
\begin{equation}\nonumber
  F_{A}\wedge\Omega=\left[F_{A}^{j}\wedge\mathfrak{Re}\left(\Omega\right)+iF_{A}^{j}\wedge\mathfrak{Im}\left(\Omega\right)\right]\otimes e_{j}.
\end{equation}
Using (\ref{InvRealOmega}) and (\ref{InvImaginaryOmega}) we obtain the results:
\begin{align}\label{SO(3)HYMRe1}
    F^{1}_{A}\wedge\mathfrak{Re}\left(\Omega\right)&=\left(R_{+}^{3}(1-b^{2})-R_{+}R_{-}^{2}(1-a^{2})\right)\theta^{123456}-2R_{+}R_{-}^{2}ab\;\theta^{123567}\\
    &+\frac{r}{2}\left(R_{+}(1-b^{2})-\frac{R_{-}^{2}}{R_{+}}(1-a^{2})\right)dr\wedge\theta^{23567}+rR_{+}ab\;dr\wedge\theta^{23456}\nonumber\\
    &+\left(\frac{rR_{+}}{2}b+R_{+}^{3}\frac{db}{dr}\right)dr\wedge\theta^{12347}-\left(\frac{rR_{-}^{2}}{2R_{+}}b+R_{+}R_{-}^{2}\frac{db}{dr}\right)dr\wedge\theta^{14567}\nonumber\\
    &-\left(\frac{rR_{+}}{2}a+R_{+}R_{-}^{2}\frac{da}{dr}\right)dr\wedge\theta^{12467}+\left(\frac{rR_{+}}{2}a+R_{+}R_{-}^{2}\frac{da}{dr}\right)dr\wedge\theta^{13457},\nonumber
\end{align}

\begin{align}\label{SO(3)HYMRe2}
    F^{2}_{A}\wedge\mathfrak{Re}\left(\Omega\right)&=\left(R_{+}^{3}(1-b^{2})-R_{+}R_{-}^{2}(1-a^{2})\right)\theta^{123457}-2R_{+}R_{-}^{2}ab\;\theta^{124567}\\
    &+\frac{r}{2}\left(R_{+}(1-b^{2})-\frac{R_{-}^{2}}{R_{+}}(1-a^{2})\right)dr\wedge\theta^{24567}+rR_{+}ab\;dr\wedge\theta^{23457}\nonumber\\
    &-\left(\frac{rR_{+}}{2}b+R_{+}^{3}\frac{db}{dr}\right)dr\wedge\theta^{12346}+\left(\frac{rR_{-}^{2}}{2R_{+}}b+R_{+}R_{-}^{2}\frac{db}{dr}\right)dr\wedge\theta^{13567}\nonumber\\
    &+\left(\frac{rR_{+}}{2}a+R_{+}R_{-}^{2}\frac{da}{dr}\right)dr\wedge\theta^{12367}-\left(\frac{rR_{+}}{2}a+R_{+}R_{-}^{2}\frac{da}{dr}\right)dr\wedge\theta^{13456},\nonumber
\end{align}
\begin{align}\label{SO(3)HYMRe3}
    F^{3}_{A}\wedge\mathfrak{Re}\left(\Omega\right)&=\left(R_{+}^{3}(1-b^{2})-R_{+}R_{-}^{2}(1-a^{2})\right)\theta^{123467}-2R_{+}R_{-}^{2}ab\;\theta^{134567}\\
    &+\frac{r}{2}\left(R_{+}(1-b^{2})-\frac{R_{-}^{2}}{R_{+}}(1-a^{2})\right)dr\wedge\theta^{34567}+rR_{+}ab\;dr\wedge\theta^{23467}\nonumber\\
    &+\left(\frac{rR_{+}}{2}b+R_{+}^{3}\frac{db}{dr}\right)dr\wedge\theta^{12345}-\left(\frac{rR_{-}^{2}}{2R_{+}}b+R_{+}R_{-}^{2}\frac{db}{dr}\right)dr\wedge\theta^{12567}\nonumber\\
    &-\left(\frac{rR_{+}}{2}a+R_{+}R_{-}^{2}\frac{da}{dr}\right)dr\wedge\theta^{12357}+\left(\frac{rR_{+}}{2}a+R_{+}R_{-}^{2}\frac{da}{dr}\right)dr\wedge\theta^{12456},\nonumber
\end{align}

\begin{align}\label{SO(3)HYMIm1}
    F^{1}_{A}\wedge\mathfrak{Im}\left(\Omega\right)&=\left(R_{+}^{2}R_{-}(1-b^{2})-R_{-}^{3}(1-a^{2})\right)\theta^{123567}+2R^{2}_{+}R_{-}ab\;\theta^{123456}\\
    &-\frac{r}{2}\left(\frac{R^{2}_{+}}{R_{-}}(1-b^{2})-R_{-}(1-a^{2})\right)dr\wedge\theta^{23456}+rR_{-}ab\;dr\wedge\theta^{23567}\nonumber\\
    &-\left(\frac{rR_{-}}{2}b+R_{+}^{2}R_{-}\frac{db}{dr}\right)dr\wedge\theta^{12467}+\left(\frac{rR_{-}}{2}b+R_{+}^{2}R_{-}\frac{db}{dr}\right)dr\wedge\theta^{13457}\nonumber\\
    &+\left(\frac{rR_{-}}{2}a+R_{-}^{3}\frac{da}{dr}\right)dr\wedge\theta^{14567}-\left(\frac{rR_{+}^{2}}{2R_{-}}a+R_{+}^{2}R_{-}\frac{da}{dr}\right)dr\wedge\theta^{12347},\nonumber
\end{align}
\begin{align}\label{SO(3)HYMIm2}
    F^{2}_{A}\wedge\mathfrak{Im}\left(\Omega\right)&=\left(R_{+}^{2}R_{-}(1-b^{2})-R_{-}^{3}(1-a^{2})\right)\theta^{124567}+2R^{2}_{+}R_{-}ab\;\theta^{123457}\\
    &-\frac{r}{2}\left(\frac{R^{2}_{+}}{R_{-}}(1-b^{2})-R_{-}(1-a^{2})\right)dr\wedge\theta^{23457}+rR_{-}ab\;dr\wedge\theta^{24567}\nonumber\\
    &+\left(\frac{rR_{-}}{2}b+R_{+}^{2}R_{-}\frac{db}{dr}\right)dr\wedge\theta^{12367}-\left(\frac{rR_{-}}{2}b+R_{+}^{2}R_{-}\frac{db}{dr}\right)dr\wedge\theta^{13456}\nonumber\\
    &-\left(\frac{rR_{-}}{2}a+R_{-}^{3}\frac{da}{dr}\right)dr\wedge\theta^{13567}+\left(\frac{rR_{+}^{2}}{2R_{-}}a+R_{+}^{2}R_{-}\frac{da}{dr}\right)dr\wedge\theta^{13457},\nonumber
\end{align}
\begin{align}\label{SO(3)HYMIm3}
    F^{3}_{A}\wedge\mathfrak{Im}\left(\Omega\right)&=\left(R_{+}^{2}R_{-}(1-b^{2})-R_{-}^{3}(1-a^{2})\right)\theta^{134567}+2R^{2}_{+}R_{-}ab\;\theta^{123467}\\
    &-\frac{r}{2}\left(\frac{R^{2}_{+}}{R_{-}}(1-b^{2})-R_{-}(1-a^{2})\right)dr\wedge\theta^{23467}+rR_{-}ab\;dr\wedge\theta^{34567}\nonumber\\
    &-\left(\frac{rR_{-}}{2}b+R_{+}^{2}R_{-}\frac{db}{dr}\right)dr\wedge\theta^{12357}+\left(\frac{rR_{-}}{2}b+R_{+}^{2}R_{-}\frac{db}{dr}\right)dr\wedge\theta^{12456}\nonumber\\
    &+\left(\frac{rR_{-}}{2}a+R_{-}^{3}\frac{da}{dr}\right)dr\wedge\theta^{12567}-\left(\frac{rR_{+}^{2}}{2R_{-}}a+R_{+}^{2}R_{-}\frac{da}{dr}\right)dr\wedge\theta^{12345}.\nonumber
\end{align}
Observe that there are similarities among the various components. In particular the vanishing of any one of them is equivalent to the full HYM system. Setting any one of the components to be zero gives the invariant equations. They are as follows:
\begin{equation}\label{SO(3)HYMinv1}
\frac{da}{dr}=-\frac{r}{2R_{-}^{2}},
\end{equation}
\begin{equation}\label{SO(3)HYMinv2}
\frac{db}{dr}=-\frac{r}{2R_{+}^{2}},
\end{equation}
\begin{equation}\label{SO(3)HYMinv3}
R_{+}^{2}(1-b^{2})=R_{-}(1-a^{2}),
\end{equation}
\begin{equation}\label{SO(3)HYMinv4}
ab=0.
\end{equation}
Consequently, the invariant HYM connections over $P_{\text{Id}}$ obey the differential equations (\ref{SO(3)HYMinv1}), (\ref{SO(3)HYMinv2}) and satisfy the constraints (\ref{SO(3)HYMinv3}), (\ref{SO(3)HYMinv4}). Observe that (\ref{HYMEquation2}) involves only the holomorphic volume form of $X^{8}$. As a result, the coefficients of the Stenzel metric do not appear in (\ref{SO(3)HYMinv1})-(\ref{SO(3)HYMinv4}).

\subsubsection{Explicit Solution, Extension to \texorpdfstring{$S^{4}$}{text} and Decay of the Curvature Norm}\label{BundNotTriv}
The equations (\ref{SO(3)HYMinv1})-(\ref{SO(3)HYMinv4}) can be solved explicitly. We obtain precisely two solutions:
\begin{equation}
    A_{\text{HYM}_{\pi_{1}}}\defeq \frac{1}{R_{+}}\left(\theta^{5}e_{3}-\theta^{6}e_{2}+\theta^{7} e_{1}\right),
\end{equation}
\begin{equation}
    A_{\text{HYM}_{\pi_{2}}}\defeq -\frac{1}{R_{+}}\left(\theta^{5}e_{3}-\theta^{6}e_{2}+\theta^{7} e_{1}\right).
\end{equation}
The notation has been chosen in hindsight to reflect the bundle on which these connections extend. In particular, using propositions \ref{PPi1ExtendCondition} and \ref{PPi2ExtendCondition}, we find that $A_{\text{HYM}_{\pi_{1}}}$ extends to $P_{\pi_{1}}$, while $A_{\text{HYM}_{\pi_{2}}}$ extends to $P_{\pi_{2}}$. We treat  $A_{\text{HYM}_{\pi_{1}}}$ in detail. In this case:
\begin{equation}\nonumber
    a(r(t))=0,\;\;\;\;b(r(t))-1=\frac{1-\sqrt{t^{2}+1}}{\sqrt{t^{2}+1}}.
\end{equation}
The conditions on $a(r(t))$ are trivially satisfied. The function $b(r(t))-1$ is obviously smooth from the right at $t=0$. It is even, since $t$ only appears in power $2$. One easily computes that both $b(r(t))-1$ and its first derivative vanish at $t=0$. Consequently, $b(r(t))-1=O(t^{2})$.

$\newline$
Using (\ref{InvCurvSO(3)Id}) we compute the associated curvature tensors:
\begin{align}\label{InvCurvSO(3)IdPi1}
    F_{A_{\text{HYM}_{\pi_{1}}}}&=\left(\theta^{23}+\frac{R_{-}^{2}}{R_{+}^{2}}\;\theta^{56}+\frac{1}{R_{+}}\;\theta^{14}-\frac{r}{2R_{+}^{3}}dr\wedge\theta^{7}\right)\otimes e_{1}\\
     &+\left(\theta^{24}+\frac{R_{-}^{2}}{R_{+}^{2}}\;\theta^{57}-\frac{1}{R_{+}}\;\theta^{13}+\frac{r}{2R_{+}^{3}}dr\wedge\theta^{6}\right)\otimes e_{2}\nonumber\\
      &+\left(\theta^{34}+\frac{R_{-}^{2}}{R_{+}^{2}}\;\theta^{67}+\frac{1}{R_{+}}\;\theta^{12}-\frac{r}{2R_{+}^{3}}dr\wedge\theta^{5}\right)\otimes e_{3},\nonumber
\end{align}

\begin{align}\label{InvCurvSO(3)IdPi2}
    F_{A_{\text{HYM}_{\pi_{2}}}}&=\left(\theta^{23}+\frac{R_{-}^{2}}{R_{+}^{2}}\;\theta^{56}-\frac{1}{R_{+}}\;\theta^{14}+\frac{r}{2R_{+}^{3}}dr\wedge\theta^{7}\right)\otimes e_{1}\\
     &+\left(\theta^{24}+\frac{R_{-}^{2}}{R_{+}^{2}}\;\theta^{57}+\frac{1}{R_{+}}\;\theta^{13}-\frac{r}{2R_{+}^{3}}dr\wedge\theta^{6}\right)\otimes e_{2}\nonumber\\
      &+\left(\theta^{34}+\frac{R_{-}^{2}}{R_{+}^{2}}\;\theta^{67}-\frac{1}{R_{+}}\;\theta^{12}+\frac{r}{2R_{+}^{3}}dr\wedge\theta^{5}\right)\otimes e_{3}.\nonumber
\end{align}
Since both solutions smoothly extend to the singular orbit, we can study their pullbacks. The connection $A_{\text{HYM}_{\pi_{1}}}$ pulls back to the canonical invariant connection of $P_{\pi_{1}}$ over $S^{4}$:
\begin{align}
 A_{\text{SD}}&\defeq A_{\text{HYM}_{\pi_{1_{|_{S^{4}}}}}}\nonumber\\
 &=\left(\theta^{8}+\theta^{7}\right)\otimes e_{1}+\left(\theta^{9}-\theta^{6}\right)\otimes e_{2}+\left(\theta^{10}+\theta^{5}\right)\otimes e_{3}.
\end{align}
Its curvature is given by:
\begin{equation}
    F_{A_{\text{SD}}}=\left(\theta^{23}+\theta^{14}\right)\otimes e_{1}+\left(\theta^{24}-\theta^{13}\right)\otimes e_{2}+\left(\theta^{34}+\theta^{12}\right)\otimes e_{3}.
\end{equation}
An explicit calculation yields that $A_{\text{SD}}$ is a self dual instanton on $S^{4}$. This justifies our choice of notation.

$\newline$
The connection $A_{\text{HYM}_{\pi_{2}}}$ pulls back to the canonical invariant connection of $P_{\pi_{2}}$:
\begin{align}
 A_{\text{ASD}}&\defeq A_{\text{HYM}_{\pi_{2_{|_{S^{4}}}}}}\nonumber\\
 &=\left(\theta^{8}-\theta^{7}\right)\otimes e_{1}+\left(\theta^{9}+\theta^{6}\right)\otimes e_{2}+\left(\theta^{10}-\theta^{5}\right)\otimes e_{3}.
\end{align}
Its curvature is given by:
\begin{equation}\label{CurvASD}
    F_{A_{\text{ASD}}}=\left(\theta^{23}-\theta^{14}\right)\otimes e_{1}+\left(\theta^{24}+\theta^{13}\right)\otimes e_{2}+\left(\theta^{34}-\theta^{12}\right)\otimes e_{3}.
\end{equation}
An explicit calculation yields that $A_{\text{ASD}}$ is an anti-self-dual instanton on $S^{4}$. This justifies our choice of notation.

We wish to endow $\text{ad}\left(P_{\pi_{1}}\right)$ with a fiber metric. To this end, it suffices to choose an Ad-invariant inner product on $\mathfrak{so}(3)$. In general, the choice of such an inner product is free. However, we shall choose $\langle\cdot,\cdot\rangle$ so that:
\begin{equation}\label{InnProdso(3)}
    e_{i}\perp e_{j}\text{ if }i\neq j,\;\;\;\; \norm{e_{i}}^{2}=\frac{1}{2}.
\end{equation}
This is the unique inner product on $\mathfrak{so}(3)$ satisfying:
\begin{equation}\label{CrucialConditionInnProdso(3)}
\norm{\xi}^{2}=-\text{Tr}\left(\xi^{2}\right).
\end{equation}
This identity is required to relate the Yang-Mills energy of instantons to characteristic classes of the underlying bundle. Using the Stenzel metric and the fiber metric (\ref{InnProdso(3)}) we see that the curvature norms of the two connections have the same constant value on all points of $S^{4}$:
\begin{equation}\nonumber
\norm{F_{A_{\text{SD}}}}^{2}=\norm{F_{A_{\text{ASD}}}}^{2}=3.
\end{equation}
Since the restriction of the Stenzel metric on the singular orbit is round of unit radius, we have:
\begin{align}\label{ASDSDenergy}
    \mathcal{YM}\left(F_{A_{\text{SD}}}\right)&=\mathcal{YM}\left(F_{A_{\text{ASD}}}\right)\nonumber\\
    &=\int_{S^{4}}\norm{F_{A_{\text{SD}}}}^{2}dV_{g}\nonumber\\
    &=3\text{Vol}\left(S^{4}\right)\\
    &=8\pi^{2}.\nonumber
\end{align}
Owing to (\ref{CrucialConditionInnProdso(3)}), any SO(3)-connection satisfies:
\begin{equation}
\text{Tr}\left(F_{A}^{2}\right)=\left(\norm{F_{A}^{-}}^{2}-\norm{F_{A}^{+}}^{2}\right)dV_{g}.
\end{equation}
Using this result, we see that the self dual SO(3) instantons on $S^{4}$ have Yang-Mills energy equal to $-8\pi^{2}$ times the integral of the first Pontryagin class of the bundle. Similarly, ASD instantons have Yang-Mills energy equal to $8\pi^{2}$ times the integral of the first Pontryagin class. Using (\ref{ASDSDenergy}) we obtain:
\begin{align}
p_{1}\left(P_{\pi_{1}}\right)&=-1,\\
p_{1}\left(P_{\pi_{2}}\right)&=1.
\end{align}
We already knew that $P_{\pi_{1}}$ and $P_{\pi_{2}}$ are not equivariantly trivial nor equivariantly isomorphic to each other. The above calculation shows that they are genuinely nonrivial and non-isomorphic (even if we drop the requirement that the identification be equivariant).

$\newline$
Finally, we determine the radial decay rate of the curvature norm of $A_{\text{HYM}_{\pi_{1}}}$ and $A_{\text{HYM}_{\pi_{2}}}$. Since (\ref{InvCurvSO(3)IdPi1}) and (\ref{InvCurvSO(3)IdPi2}) only differ by certain signs, it suffices to treat $A_{\text{HYM}_{\pi_{1}}}$. Using (\ref{InnProdso(3)}), (\ref{X1lengthCY}), (\ref{RadiallengthCY}), (\ref{X234lengthCY}), (\ref{X567lengthCY}) and (\ref{InvCurvSO(3)IdPi1}), we find that:
\begin{align}
&\norm{F_{A_{\text{HYM}_{\pi_{1}}}}}^{2}=2\sqrt{3}\frac{3r^{4}+10r^{2}+11}{\left(r^{2}+1\right)^{3}\left(r^{2}+2\right)^{\frac{1}{2}}}.\nonumber
\end{align}
In particular, we see that as $r\to\infty$:
\begin{equation}
\norm{F_{A_{\text{HYM}_{\pi_{1}}}}}^{2}=O(r^{-3}).
\end{equation}
This decay rate is not sufficient for the Yang-Mills energy to be finite:
\begin{align}
\mathcal{YM}\left(A_{\text{HYM}_{\pi_{1}}}\right)&=\int_{X^{8}}\norm{F_{A_{\text{HYM}_{\pi_{1}}}}}^{2}dV_{g}\nonumber\\
&=\int_{\frac{\text{SO}(5)}{\text{SO}(3)}}\theta^{1234567}\int_{0}^{\infty}\norm{F_{A_{\text{HYM}_{\pi_{1}}}}}^{2}\frac{r}{2}R_{+}^{2}R_{-}^{2}dr\nonumber\\
&=\frac{\sqrt{3}}{4}\text{Vol}\left(\frac{\text{SO}(5)}{\text{SO}(3)}\right)\int_{0}^{\infty}\frac{r(r-1)(r+1)}{\left(r^{2}+1\right)^{2}\left(r^{2}+2\right)^{\frac{1}{2}}}\left(3r^{4}+10r^{2}+11\right)dr\nonumber\\
&=K\int_{0}^{\infty}O\left(r^{2}\right)dr=+\infty.\nonumber
\end{align}

\newpage
\subsection{The SO(5)-Invariant Spin(7) Instanton ODEs on \texorpdfstring{$P_{\text{Id}}$}{text}}
\subsubsection{Derivation}
We consider a general invariant connection $A\in\mathcal{A}_{\text{inv}}\left(P_{\text{Id}}\right)$ with associated curvature field:
$$F_{A}=F_{A}^{j}\otimes e_{j}.$$
We remind the reader that these take the form (\ref{InvConSO(3)Id}), (\ref{InvCurvSO(3)Id}). Using (\ref{CayleyCalib}) we compute:
\begin{align}
\Phi\wedge F_{A}^{1}&=\left(PQ(1-a^{2})+\frac{rR_{+}}{2}b+R_{+}^{3}\frac{db}{dr}\right)dr\wedge\theta^{12347}\\
&+\left(PQ(1-b^{2})-\frac{rR_{-}^{2}}{2R_{+}}b-R_{+}R_{-}^{2}\frac{db}{dr}\right)dr\wedge\theta^{14567}\nonumber\\
&+\left(PQ\;ab-\frac{rR_{+}}{2}a-R_{+}R_{-}^{2}\frac{da}{dr}\right)dr\wedge\theta^{12467}+\left(-PQ\;ab+\frac{rR_{+}}{2}a+R_{+}R_{-}^{2}\frac{da}{dr}\right)dr\wedge\theta^{13457}\nonumber\\
&+\left(\frac{rR_{+}}{2}(1-a^{2})-\frac{rR_{-}^{2}}{2R_{+}}(1-a^{2})-Q^{2}\frac{db}{dr}\right)dr\wedge\theta^{23567}+\left(rR_{+}ab-Q^{2}\frac{da}{dr}\right)dr\wedge\theta^{23456}\nonumber\\
&+\left(R_{+}^{3}(1-b^{2})-R_{+}R_{-}^{2}(1-a^{2})-Q^{2}b\right)\theta^{123456}+\left(-2R_{+}R_{-}^{2}ab+Q^{2}a\right)\theta^{123567},\nonumber
\end{align}
\begin{align}
\Phi\wedge F_{A}^{2}=&-\left(PQ(1-a^{2})+\frac{rR_{+}}{2}b+R_{+}^{3}\frac{db}{dr}\right)dr\wedge\theta^{12346}\\
&+\left(-PQ(1-b^{2})+\frac{rR_{-}^{2}}{2R_{+}}b+R_{+}R_{-}^{2}\frac{db}{dr}\right)dr\wedge\theta^{13567}\nonumber\\
&+\left(-PQ\;ab+\frac{rR_{+}}{2}a+R_{+}R_{-}^{2}\frac{da}{dr}\right)dr\wedge\theta^{12367}+\left(PQ\;ab-\frac{rR_{+}}{2}a-R_{+}R_{-}^{2}\frac{da}{dr}\right)dr\wedge\theta^{13456}\nonumber\\
&+\left(\frac{rR_{+}}{2}(1-a^{2})-\frac{rR_{-}^{2}}{2R_{+}}(1-a^{2})-Q^{2}\frac{db}{dr}\right)dr\wedge\theta^{24567}+\left(rR_{+}ab-Q^{2}\frac{da}{dr}\right)dr\wedge\theta^{23457}\nonumber\\
&+\left(R_{+}^{3}(1-b^{2})-R_{+}R_{-}^{2}(1-a^{2})-Q^{2}b\right)\theta^{123457}+\left(-2R_{+}R_{-}^{2}ab+Q^{2}a\right)\theta^{124567},\nonumber
\end{align}

\begin{align}
\Phi\wedge F_{A}^{3}&=\left(PQ(1-a^{2})+\frac{rR_{+}}{2}b+R_{+}^{3}\frac{db}{dr}\right)dr\wedge\theta^{12345}\\
&+\left(PQ(1-b^{2})-\frac{rR_{-}^{2}}{2R_{+}}b-R_{+}R_{-}^{2}\frac{db}{dr}\right)dr\wedge\theta^{12567}\nonumber\\
&+\left(PQ\;ab-\frac{rR_{+}}{2}a-R_{+}R_{-}^{2}\frac{da}{dr}\right)dr\wedge\theta^{12357}+\left(-PQ\;ab+\frac{rR_{+}}{2}a+R_{+}R_{-}^{2}\frac{da}{dr}\right)dr\wedge\theta^{12456}\nonumber\\
&+\left(\frac{rR_{+}}{2}(1-a^{2})-\frac{rR_{-}^{2}}{2R_{+}}(1-a^{2})-Q^{2}\frac{db}{dr}\right)dr\wedge\theta^{34567}+\left(rR_{+}ab-Q^{2}\frac{da}{dr}\right)dr\wedge\theta^{23467}\nonumber\\
&+\left(R_{+}^{3}(1-b^{2})-R_{+}R_{-}^{2}(1-a^{2})-Q^{2}b\right)\theta^{123467}+\left(-2R_{+}R_{-}^{2}ab+Q^{2}a\right)\theta^{134567}.\nonumber
\end{align}
We now wish to calculate the Hodge dual of the curvature. We will require the Hodge duals of all $2$-forms $\theta^{ij}$. These can be computed using (\ref{InvMetric}) and (\ref{HodgeStar}). Carrying out this calculation yields:
\begin{align}
    \star_{g}F_{A}^{1}=&-(1-a^{2})\frac{PQR_{-}^{2}}{R_{+}^{2}}dr\wedge\theta^{14567}-(1-b^{2})\frac{PQR_{+}^{2}}{R_{-}^{2}}dr\wedge\theta^{12347}\\
    &+abPQdr\wedge\theta^{12467}-abPQdr\wedge\theta^{13457}\newline\nonumber\\
    &-\frac{rQ^{2}}{2R_{+}^{2}}b\;dr\wedge\theta^{23567}-a\frac{rQ^{2}}{2R_{-}^{2}}dr\wedge\theta^{23456}\nonumber\\
    &+\frac{da}{dr}\frac{2R_{-}^{2}Q^{2}}{r}\theta^{123567}-\frac{db}{dr}\frac{2R_{+}^{2}Q^{2}}{r}\theta^{123456},\nonumber
\end{align}

\begin{align}
    \star_{g}F_{A}^{2}=&(1-a^{2})\frac{PQR_{-}^{2}}{R_{+}^{2}}dr\wedge\theta^{13567}+(1-b^{2})\frac{PQR_{+}^{2}}{R_{-}^{2}}dr\wedge\theta^{12346}\\
    &-abPQdr\wedge\theta^{12367}+abPQdr\wedge\theta^{13456}\newline\nonumber\\
    &-\frac{rQ^{2}}{2R_{+}^{2}}b\;dr\wedge\theta^{24567}-a\frac{rQ^{2}}{2R_{-}^{2}}dr\wedge\theta^{23457}\nonumber\\
    &+\frac{da}{dr}\frac{2R_{-}^{2}Q^{2}}{r}\theta^{124567}-\frac{db}{dr}\frac{2R_{+}^{2}Q^{2}}{r}\theta^{123457},\nonumber
\end{align}

\begin{align}
    \star_{g}F_{A}^{1}=&-(1-a^{2})\frac{PQR_{-}^{2}}{R_{+}^{2}}dr\wedge\theta^{12567}-(1-b^{2})\frac{PQR_{+}^{2}}{R_{-}^{2}}dr\wedge\theta^{12345}\\
    &+abPQdr\wedge\theta^{12357}-abPQdr\wedge\theta^{12456}\newline\nonumber\\
    &-\frac{rQ^{2}}{2R_{+}^{2}}b\;dr\wedge\theta^{34567}-a\frac{rQ^{2}}{2R_{-}^{2}}dr\wedge\theta^{23467}\nonumber\\
    &+\frac{da}{dr}\frac{2R_{-}^{2}Q^{2}}{r}\theta^{134567}-\frac{db}{dr}\frac{2R_{+}^{2}Q^{2}}{r}\theta^{123467}.\nonumber
\end{align}
The Spin(7) instanton equations are given by:
\begin{equation}\label{Spin(7)eqSO(3)Prelim}
    \star_{g}F^{i}_{A}=-\Phi\wedge F^{i}_{A}.
\end{equation}
The set of equations obtained by imposing (\ref{Spin(7)eqSO(3)Prelim}) is the same for each $i=1,2,3$. It is as follows:
\begin{align}
&\frac{da}{dr}=\frac{2PQ}{R_{+}R_{-}^{2}}\;ab-\frac{r}{2R_{-}^{2}}a,\\
&\frac{da}{dr}=\frac{rR_{+}}{Q^{2}}ab-\frac{r}{2R_{-}^{2}}a,\\
&\frac{db}{dr}=\frac{PQ}{R_{-}^{2}R_{+}}(1-b^{2})-\frac{PQ}{R_{+}^{3}}(1-a^{2})-\frac{r}{2R_{+}^{2}}b,\\
&\frac{db}{dr}=\frac{rR_{+}}{2Q^{2}}(1-b^{2})-\frac{R_{-}^{2}r}{2R_{+}Q^{2}}(1-a^{2})-\frac{r}{2R_{+}^{2}}b.
\end{align}
This system is overdetermined unless the metric coefficients satisfy the condition:
\begin{equation}
PQ^{3}=\frac{r}{2}R_{+}^{2}R_{-}^{2}.
\end{equation}
We recognize this as the SO(5)-invariant Monge–Ampère equation (\ref{CalabiYauEquation}) distinguishing the Stenzel metric among the K\"{a}hler metrics induced from SO(5)-invariant potentials.

$\newline$
It is useful to work in coordinates compatible with the Eschenburg-Wang analysis. We therefore switch to the variable $t=R_{-}$. An elementary calculation shows that the system takes the form:
\begin{align}\label{InvSpin(7)InstantonEqSO(3)}
    &\frac{da}{dt}=\frac{\mathcal{P}a}{t}\left(b-\frac{1}{\mathcal{P}}\right),\\
    &\frac{db}{dt}=\frac{\mathcal{P}}{2t}\left(1-b^2\right)-\frac{\mathcal{P}\mathcal{Q}}{2}\left(1-a^2\right)-\mathcal{Q}b.\nonumber
\end{align}
where we have introduced the functions $\mathcal{P},\mathcal{Q}\in C^{\infty}\left[0,\infty\right)$ defined by:
\begin{align}
&\mathcal{P}(t)\defeq\frac{\sqrt{6}\sqrt{2t^{2}+2}}{\sqrt{2t^{2}+3}},\\
&\mathcal{Q}(t)\defeq\frac{t}{t^2+1}.
\end{align}

\subsubsection{Elementary Observations}
We begin our analysis of the system (\ref{InvSpin(7)InstantonEqSO(3)}). In this section we make a few elementary observations about the dynamics. First, we have the following:
\begin{proposition}\label{aNotZero}
The dynamics \ref{InvSpin(7)InstantonEqSO(3)} preserve the vanishing of $a$ and correspondingly if $a(t)\neq 0$ for some $t> 0$, then $a(t)\neq 0$ for all $t> 0$.
\end{proposition}
\begin{proof}
The first statement is trivial. The second follows by the uniqueness part of the standard Picard theorem.
\end{proof}
Next we observe a symmetry in the solution space:
\begin{proposition}\label{SystemSymmetry}
Suppose that the pair $(a,b)$ solves the system (\ref{InvSpin(7)InstantonEqSO(3)}). Then so does $(-a,b)$.
\end{proposition}
\begin{proof}
This follows from a trivial calculation.
\end{proof}
We conclude the following: either $a=0$ for all time, or $a$ has a fixed sign throughout its lifespan. Furthermore it suffices to study the case $a>0$ as -owing to the above observation- all solutions $(a,b)$ with $a<0$ can be obtained by considering a solution where $a>0$ and inverting its sign.

$\newline$
The next proposition establishes that if one solution lies above another at some instant $t^{\star}$, the inequality persists for all time. Here, 'lying above' is interpreted componentwise.
\begin{proposition}\label{NoCross}
Suppose that $(a,b)$, $(\widetilde{a},\widetilde{b})$ are two solutions to the system \ref{InvSpin(7)InstantonEqSO(3)}. Suppose further that for some time $t^{\star}\ge 0$ we have $a(t^{\star})>\widetilde{a}(t^{\star})$ and $b(t^{\star})>\widetilde{b}(t^{\star})$. These inequalities remain true for all $t\ge t^{\star}$ for which both solutions exist.
\end{proposition}

\begin{proof}
Suppose not. Let $t_{f}$ be the first time for which the inequality fails. There are three cases:
\begin{enumerate}
    \item $\widetilde{a}(t_{f})=a(t_{f})$ and $\widetilde{b}(t_{f})=b(t_{f})$,
    \item $\widetilde{a}(t_{f})=a(t_{f})$ and $\widetilde{b}(t_{f})<b(t_{f})$,
    \item $\widetilde{a}(t_{f})<a(t_{f})$ and $\widetilde{b}(t_{f})=b(t_{f})$.
\end{enumerate}
Case $1$ contradicts the uniqueness part of the standard Picard theorem.

$\newline$
Suppose case $2$ holds. Consider the evolution of $a-\widetilde{a}$:
\begin{equation}
    \frac{d}{dt}\left(a-\widetilde{a}\right)=\frac{\mathcal{P}ab}{t}-\frac{a}{t}-\left(\frac{\mathcal{P}\widetilde{a}\widetilde{b}}{t}-\frac{\widetilde{a}}{t}\right).\nonumber
\end{equation}
At $t=t_{f}$ we have $\widetilde{a}(t_{f})=a(t_{f})=s>0$ and $\widetilde{b}(t_{f})<b(t_{f})$. Consequently:
\begin{equation}\nonumber
    \frac{d}{dt}\left(a-\widetilde{a}\right)_{|_{t_{f}}}=\frac{\mathcal{P}s}{t}\left(b(t_{f})-\widetilde{b}(t_{f})\right)>0.
\end{equation}
It follows that $\widetilde{a}(t)<a(t)$ for some time $t<t_{f}$ and the intermediate value theorem contradicts the fact that $t_{f}$ is the first time for which the inequalities fail.

$\newline$
Suppose case $3$ holds. Consider the evolution of $b-\widetilde{b}$:
\begin{equation}\nonumber
    \frac{d}{dt}\left(b-\widetilde{b}\right)=\frac{\mathcal{P}}{2t}\left(1-b^{2}\right)-\frac{\mathcal{P}\mathcal{Q}}{2}\left(1-a^{2}\right)-\mathcal{Q}b-\left(\frac{\mathcal{P}}{2t}\left(1-\widetilde{b}^{2}\right)-\frac{\mathcal{P}\mathcal{Q}}{2}\left(1-\widetilde{a}^{2}\right)-\mathcal{Q}\widetilde{b}\right).
\end{equation}
At $t=t_{f}$ we have $\widetilde{b}(t_{f})=b(t_{f})$ and $\widetilde{a}(t_{f})<a(t_{f})$. Consequently:
\begin{equation}
    \frac{d}{dt}\left(b-\widetilde{b}\right)_{|_{t_{f}}}=\frac{\mathcal{P}\mathcal{Q}}{2}\left(a(t_{f})-\widetilde{a}(t_{f})\right)>0.
\end{equation}
which leads to a contradiction as above.
\end{proof}

\begin{proposition}
Suppose that $(a,b)$ is a solution of \ref{InvSpin(7)InstantonEqSO(3)} defined in a neighbourhood of $t_{0}>0$. Take initial data at $t_{0}$ satisfying $a(t_{0})>0$, $b(t_{0})<0$ and flow backwards. Then $a\to +\infty$ as $t\to t_{\text{blowup}}\ge 0$.
\end{proposition}

\begin{proof}
We will bound $a$ from below by a function $v$ satisfying $v\to +\infty$ as $t\to 0$.

$\newline$
Consider the evolution of the product $ab$. Using the equations (\ref{InvSpin(7)InstantonEqSO(3)}), compute:
\begin{align}
    \frac{d}{dt}(ab)|_{t}&=\dt{a}b+a\dt{b}\nonumber\\
    &=\frac{P}{2t}ab^{2}+\frac{\mathcal{P}}{2t\left(t^{2}+1\right)}a+\frac{\mathcal{P}\mathcal{Q}}{2}a^3-\left(\mathcal{Q}+\frac{1}{t}\right)ab
    >-\left(\mathcal{Q}+\frac{1}{t}\right)ab,
\end{align}
where in the last line we used the fact that $a>0$ for all time. By comparison, flowing backwards in time, $ab$ stays below the solution of the I.V.P:
\begin{equation}\nonumber
    \begin{cases}
      & \dt{u}(t)=-\left(\mathcal{Q}+\frac{1}{t}\right)v,\\
      & u(t_{0})=a(t_{0})b(t_{0}).
    \end{cases}
\end{equation}
By assumption, the initial data satisfy:
\begin{equation}
    a(t_{0})b(t_{0})<0.\nonumber
\end{equation}
Consequently, $u<0$ for all $0<t<t_{0}$ and we conclude that the same is true of $ab$.

$\newline$
This allows us to estimate:
\begin{equation}\nonumber
    \dt{a}(t)=\frac{\mathcal{P}ab}{t}-\frac{a}{t}<-\frac{a}{t}.
\end{equation}
Consequently, a lies above the solution to the following I.V.P backwards of $t_{0}$:
\begin{equation}\nonumber
    \begin{cases}
      & \dt{v}(t)=-\frac{v}{t},\\
      & v(t_{0})=a(t_{0}).
    \end{cases}
\end{equation}
This is easily solved explicitly and we obtain the inequality:
\begin{equation}\nonumber
a(t)\ge\frac{a(t_{0})t_{0}}{t}\text{ for all }0<t\le t_{0}.
\end{equation}
\end{proof}
\begin{corollary}\label{bPos}
Let $T>0$ and let $(a,b)\in C^{1}[0,T]$ be a solution of \ref{InvSpin(7)InstantonEqSO(3)} satisfying $a\neq 0$. We have that $b(t)>0$ for all $t\ge 0$ for which the solution exists.
\end{corollary}
\begin{proof}
Trivially, $b(0)=\pm1$. If not, then $\dt{b}(t)$ blows up as $t\to 0$. Hence it suffices to prove the result for $t>0$. If we achieve this, the possibility that $b(0)=-1$ is excluded by continuity and thus we have that $b(0)=1$.

$\newline$
Suppose that for some $t_{0}>0$, $b(t_{0})<0$. We have that $a(t_{0})\neq 0$ by assumption. If $a(t_{0})>0$, the above proposition implies that $a$ blows up to $+\infty$ near $t=0$ contradicting the boundedness of the solution. If $a(t_{0})<0$, then $-a(t_{0})>0$. Since $(-a,b)$ is a solution, $-a$ blows up to $+\infty$ near $t=0$. Hence, $a$ blows up to $-\infty$ near $t=0$.

$\newline$
Suppose that $b(t_{0})=0$ for some $t_{0}>0$. At such a point we have:
\begin{equation}
    \dt{b}(t_{0})=\frac{\mathcal{P}(t_{0})}{2t_{0}(t_{0}^2+1)}+\frac{\mathcal{P}(t_{0})\mathcal{Q}(t_{0})}{2}a^2(t_{0})>0.\nonumber
\end{equation}
It follows that $b(t)<0$ for some $0<t<t_{0}$ and this brings us to the previous case.
\end{proof}
Putting the above together: if $(a,b)$ is a global solution of \ref{InvSpin(7)InstantonEqSO(3)}, either $a=0$ identically or the sign of $a$ is fixed and $b>0$.

\subsubsection{Solutions Extending on \texorpdfstring{$P_{\pi_{2}}$}{text}: An Explicit Family of Spin(7)-Instantons Containing a Unique HYM Connection}\label{Pi2Sols}
We classify the solutions that extend on $P_{\pi_{2}}$. In the division suggested in the final remark of the preceding subsection, this corresponds to the case $a=0$. The system (\ref{InvSpin(7)InstantonEqSO(3)}) reduces to a single nonlinear ODE that we can solve explicitly. We thus exhibit an explicit $1$-parameter family of Spin(7) instantons only one of which is HYM. This resolves (negatively) the question regarding the equivalence of the two gauge theoretic problems.

$\newline$
Owing to proposition \ref{PPi2ExtendCondition}, solutions extending to $P_{\pi_{2}}$ must satisfy $b(0)=-1$. Due to corollary \ref{bPos},the $a$-component of such a solution must vanish identically. The system (\ref{InvSpin(7)InstantonEqSO(3)}) reduces to the following ODE:
\begin{equation}\label{aZero}
\frac{db}{dt}=-\frac{\mathcal{P}\mathcal{Q}}{2}+\frac{\mathcal{P}}{2t}\left(1-b^2\right)-\mathcal{Q}b.
\end{equation}
This can be solved explicitly. We fix a positive reference time and parameterize solutions by their value at that time. We choose to work with $t_{\text{ref}}=\frac{\sqrt{6}}{2}$ (corresponding to $r_{\text{ref}}=2$). This choice is arbitrary. Note that our approach excludes solutions blowing up at $t_{\text{ref}}$. This is not an issue as we are only interested in global instantons. Writing: $\nu=b\left(t_{\text{ref}}\right)$, the associated solution to (\ref{aZero}) takes the form:
\begin{equation}\label{bPi2}
b_{\nu}(t)=\frac{\sqrt{2}}{2}\left(1+\frac{\sqrt{6}-\nu\sqrt{10t^{2}+15}}{\sqrt{30}\nu+\sqrt{6}-\left(\sqrt{5}\nu+2\right)\sqrt{2t^{2}+3}}\right)\frac{1}{\sqrt{t^2+1}}.
\end{equation}
Corresponding to $b_{\nu}$ there is a local Spin(7) instanton (\ref{Pi2InvCon}) on the restriction of $P_{\text{Id}}$ over an open submanifold of the form:
\begin{equation}
    (t_{\text{ref}}-\delta,t_{\text{ref}}+\delta)\times\frac{\text{SO}(5)}{\text{SO}(3)}\subset X^{8}.
\end{equation}
An elementary calculation yields the values of $\nu$ for which there exists a finite blowup time:
\begin{proposition}
Let $\nu \in (-\infty,-\frac{2\sqrt{5}}{5})\cup (\frac{\sqrt{10}}{5},\infty)$.
The connection $A_{\nu}$ blows up (as witnessed -for instance- by a blowup of the pointwise curvature norm) at time $t_{\text{blowup}}(\nu)$ given by:
\begin{equation}
t_{\text{blowup}}(\nu)=\frac{\sqrt{6}}{2}\frac{\sqrt{5\nu^{2}-2}}{\sqrt{5}\nu+2}.
\end{equation}
\end{proposition}
For $\nu$ outside of the range considered in the proposition, the solutions stay bounded for all time. These considerations lead to the following existence/classification result:
\begin{theorem}
Let $\nu\in [-\frac{2\sqrt{5}}{5},\frac{\sqrt{10}}{5})$. The connection $A_{\nu}$ is a smooth \emph{Spin(7)} instanton on the extended bundle $P_{\pi_{2}}$. Furthermore, these are all the invariant \emph{Spin(7)} instantons on $P_{\pi_{2}}$.
\end{theorem}
\begin{proof}
For $\nu\in [-\frac{2\sqrt{5}}{5},\frac{\sqrt{10}}{5}]$, the function $b_{\nu}$ is of class $C^{\infty}[0,\infty)$. We need to verify the extension conditions of proposition \ref{PPi2ExtendCondition}. In particular we need to prove that $b_{\nu}(t)+1$ is even and $O(t^2)$ at $t=0$. We immediately exclude $\nu=\frac{\sqrt{10}}{5}$ as the associated solution satisfies $b(0)=1$ and consequently fails the second extension condition. For $\nu\in [-\frac{2\sqrt{5}}{5},\frac{\sqrt{10}}{5})$, the first condition is clear by looking at the formula for $b_{\nu}$. The second condition is easily established by computing that:
\begin{equation}
    b_{\nu}(0)+1=\dt{b}(0)=0.
\end{equation}
For uniqueness, we note that any invariant Spin(7) instanton on $P_{\pi_{2}}$ obeys \ref{aZero} and all other solutions of this equation blow up.
\end{proof}
The HYM connection $A_{\text{HYM}_{\pi_{2}}}$ lies in the interior of this family and corresponds to the choice $\nu=-\frac{\sqrt{10}}{5}$. It is the only HYM connection in the family. The boundary point $\nu_{\partial}=-\frac{2\sqrt{5}}{5}$ corresponds to the solution:
\begin{equation}
    b_{\nu_{\partial}}(t)=-\frac{\sqrt{3}}{3}\frac{\sqrt{2t^{2}+3}}{\sqrt{t^{2}+1}}.
\end{equation}
The associated Spin(7) instanton $A_{\nu_{\partial}}$ differs from the others in that it yields a different limiting connection on $P_{\text{Id}}$ over the Stiefel manifold $\frac{\text{SO}(5)}{\text{SO}(3)}$ at infinity. In particular, for $\nu\in \left(-\frac{2\sqrt{5}}{5},\frac{\sqrt{10}}{5}\right)$ it is easily seen that:
\begin{equation}\nonumber
    \lim_{t\to\infty}b_{\nu}(t)=0.
\end{equation}
The associated limiting connection is therefore equal to $A^{\text{can}}_{P_{\text{Id}}}$. However, for $\nu=\nu_{\partial}$, we have:
\begin{equation}\nonumber
    \lim_{t\to\infty}b_{\nu_{\partial}}(t)=-\frac{\sqrt{6}}{2}.
\end{equation}
and the associated connection at infinity is given by:
\begin{equation}\nonumber
A_{\nu_{\partial}}^{\infty}=A^{\text{can}}_{P_{\text{Id}}}-\frac{\sqrt{6}}{2}\left(\theta^{5}\otimes e_{3}-\theta^{6}\otimes e_{2}+\theta^{7}\otimes  e_{1}\right).
\end{equation}
We have already computed the pointwise curvature norm of $A_{\text{HYM}_{\pi_{2}}}$. Its growth is of order $O(r^{-\frac{3}{2}})$. The rate remains the same across all elements of the family. This includes the boundary point $A_{\nu_{\partial}}$. It follows that it is possible for Spin(7) instantons not to be HYM and yet to have pointwise curvature norm decaying with the same rate as that of an HYM connection on the same bundle.

\subsubsection{Solutions Extending on \texorpdfstring{$P_{\pi_{1}}$}{text}}\label{Pi1Sols}
We now wish to classify solutions that smoothly extend over $P_{\pi_{1}}$. In the previous section we found all solutions where $a=0$. The only one satisfying $b(0)=1$ corresponds to $\nu=\frac{\sqrt{10}}{5}$. The associated instanton is $A_{\text{HYM}_{\pi_{1}}}$. Any other solution would have nonvanishing $a$-component. Consequently, we have to deal with the full system (\ref{InvSpin(7)InstantonEqSO(3)}). The first step is to obtain short time existence and uniqueness near the pole of the ODE. Subsequently, the task is to characterize which of these local solutions survive for all time to yield global Spin(7) instantons. 
\subsubsubsection{Short Time Existence and Uniqueness}
The analysis in this section relies on the method of Eschenburg and Wang (Eschenberg, Wang \cite{Esch}, section 6). We have adapted their existence result to our equation system and refined it to include continuous dependence on initial data. This does not follow from the standard Grönwall estimate as the I.V.P under consideration is singular. The continuity proof is based on the technique employed by Smoller, Wasserman, Yau and McLeod (Smoller, Wasserman, Yau, McLeod \cite{SWYM}, p.147]).
\begin{theorem}\label{ShortTime}
Let $a_{0}$ be a fixed real number. There exists a unique solution:
\begin{equation}\nonumber
\left(a,b\right)_{a_{0}}\in C^{\infty}[0,t_{max}(a_{0}))
\end{equation}
to the system \ref{InvSpin(7)InstantonEqSO(3)} such that:
\begin{align}
&a(0)=0,\label{CondAZero}\\
&\dt{a}(0)=a_{0},\label{CondDerAZero}\\
&b(0)=1.\label{CondBZero}
\end{align}
This solution satisfies the extension conditions of proposition \ref{PPi1ExtendCondition} and thus yields a \emph{Spin(7)} instanton on the restriction of $P_{\pi_{1}}$ over the open submanifold defined by $0\le t<t_{\max}(a_{0})$.

$\newline$
Furthermore, we have that for any $K>0$:
\begin{equation}
T_{K}\defeq\inf\left\{t_{\text{max}}(a_{0})\;|\;a_{0}\in[-K,K]\right\}>0
\end{equation}
and the following mapping is continuous:
\begin{align}
[-K,K]&\to C^{0}\left(\;[0,T_{K}],\mathbb{R}^{2}\right),\nonumber\\
a_{0}&\mapsto \left(a,b\right)_{a_{0}}.
\end{align}
\end{theorem}
We will prove this result in four stages. The first step is to study the formal Taylor series of smooth solutions at $t=0$. The second step is to derive and analyze ODEs governing perturbations of high order polynomial truncations of the series. The idea is to show that, if the order is high enough, the resulting ODEs are uniquely soluble for sufficiently short time in suitable Banach spaces. The third step is to argue that the solutions so obtained are smooth and have the correct formal series at $t=0$. The final step is to understand how this existence/ uniqueness argument behaves under change of initial data. This involves proving that the estimates can be made to be uniform in $a_{0}$ for $a_{0}$ in compact sets and establishing the desired continuity result.

\begin{proposition}\label{FormalSeries}
Fix $a_{0}\in\mathbb{R}$. There exists a unique $(a,b)_{a_{0}}\in\mathbb{R}[[t]]^{2}$ solving the system (\ref{InvSpin(7)InstantonEqSO(3)}) and satisfying the conditions (\ref{CondAZero}), (\ref{CondDerAZero}), (\ref{CondBZero}). Here, differentiation is understood in the formal sense (as a derivation of the formal power series ring).
\end{proposition}

\begin{proof}
Considering the ODEs governing $a(-t)$ and $b(-t)$ and invoking local uniqueness, we find that $b$ is even and $a$ is odd. This allows us to write:
\begin{equation}
    a=\sum_{k=0}^{\infty}\frac{a_{k}}{(2k+1)!}t^{2k+1},\;\;\;b=\sum_{k=0}^{\infty}\frac{b_{k}}{(2k)!}t^{2k},\;\text{where}\;\;b_{0}=1.
\end{equation}
Using the parity of $a,b$ and the coefficient functions, we introduce the series:
\begin{align}
    a\left(\mathcal{P}b-1\right)&=\sum_{k=0}^{\infty}\frac{c_{k}}{(2k+1)!}t^{2k+1},\;\;\;\;\;\;\;\;\frac{\mathcal{P}\mathcal{Q}}{2}\left(1-a^2\right)=\sum_{k=0}^{\infty}\frac{e_{k}}{(2k+1)!}t^{2k+1},\nonumber\\
    \frac{\mathcal{P}}{2}\left(1-b^{2}\right)&=\sum_{k=0}^{\infty}\frac{d_{k}}{(2k)!}t^{2k}\;d_{0}=0,\;\;\;\;\;\;\;\;\;\;\;\;\;\;\;\;\;\;\;\;\;\;\mathcal{Q}b=\sum_{k=0}^{\infty}\frac{f_{k}}{(2k+1)!}t^{2k+1}.\nonumber
\end{align}
The ODE for $a$ translates to the condition:
\begin{equation}\label{FormalTaylora}
    a_{k}=\frac{c_{k}}{2k+1}\text{ for all }k\ge 0.
\end{equation}
We compute $c_{k}$ in terms of $a_{0},...,a_{k},b_{0},...,b_{k}$. This yields:
\begin{align}
    c_{k}&=\frac{d^{2k+1}}{dt^{2k+1}}_{|_{t=0}}\left(a (\mathcal{P}b-1)\right),\nonumber\\
    &=a_{k}+\mathcal{G}(a_{0},...,a_{k-1},b_{0},...,b_{k}).\nonumber
\end{align}
Here $\mathcal{G}$ denotes some function of coefficients of lower order. We will slightly abuse notation and maintain use of the symbol $\mathcal{G}$ in subsequent calculations -even though the particular function may not be the same. Using (\ref{FormalTaylora}) we obtain:
\begin{equation}
\frac{2k}{2k+1}a_{k}=\mathcal{G}(a_{0},...,a_{k-1},b_{0},...,b_{k}).
\end{equation}
This determines $a_{k}$ in terms of coefficients of lower order provided that $k\neq 0$. We conclude that we are allowed to choose $a_{0}$ freely.

$\newline$
We perform a similar calculation for $b$. The second equation in (\ref{InvSpin(7)InstantonEqSO(3)}) translates to the relation:

\begin{equation}\label{FormalTaylorb}
    b_{k+1}=\frac{d_{k+1}}{2k+2}-e_{k}-f_{k}\text{ for all }k\ge 0.
\end{equation}
We note that $e_{k}$ and $f_{k}$ only involve terms depending on $a_{0},...,a_{k},b_{0},...,b_{k}$ and it is thus unnecessary to compute them. We compute $d_{k+1}$ in terms of $a_{0},...,a_{k},b_{0},...,b_{k+1}$:
\begin{align}
    d_{k+1}&=\frac{d^{2k+2}}{dt^{2k+2}}_{|_{t=0}}\left(\frac{\mathcal{P}}{2}(1-b^{2})\right),\nonumber\\
    &=-2b_{k+1}+\mathcal{G}(b_{0},...,b_{k}).\nonumber
\end{align}
Using (\ref{FormalTaylorb}), we obtain:
\begin{equation}\nonumber
    \frac{k+2}{k+1}b_{k+1}=\mathcal{G}(a_{0},...,a_{k},b_{0},...,b_{k}).
\end{equation}
It follows that $b_{k+1}$ is determined by lower order coefficients for each $k\ge 0$.

$\newline$
The above calculations demonstrate that the formal Taylor series at $0$ is uniquely determined by induction given a choice of $a_{0}\in\mathbb{R}$.
\end{proof}
Although the content of the preceding proposition is enough for the purposes of our existence theorem, continuity requires more refined knowledge of the formal Taylor series. In particular, we are interested in the dependence of its coefficients on $a_{0}$. We explicitly calculate the first few terms of the series associated to some fixed $a_{0}$:
\begin{align}
    a(t)&=a_{0}t-\frac{a_{0}}{3}t^{3}+O(t^{5}),\label{anear0}\\
    b(t)&=1-\frac{t^2}{2}+\left(\frac{3}{8}+\frac{a_{0}^{2}}{6}\right)t^{4}+O(t^{6}).\label{bnear0}
\end{align}
In fact, we are able to obtain the following:
\begin{proposition}\label{CoeffCts}
 The coefficients of the formal Taylor series $(a,b)_{a_{0}}$ are polynomials (possibly of order $0$) in $a_{0}$. 
\end{proposition}
\begin{proof}
This is certainly true for $a_{0},b_{0}$ and $b_{1}$. Repeating the calculations of the preceding proposition, but keeping track of the lower order terms yields the following recurrence relations for the coefficients:
\begin{align}
    a_{k}&=\frac{1}{2k}\sum_{m=1}^{k}\sum_{l=0}^{m}\binom{2k+1}{2m}\binom{2m} {2l}\mathcal{P}^{\left(2(m-l)\right)}_{|_{t=0}}a_{k-m}\;b_{l},\nonumber\\
    b_{k+1}&=-\frac{1}{4k+8}\sum_{m=1}^{k}\sum_{l=0}^{m}\binom{2k+2}{2m}\binom{2m}{2l}\mathcal{P}^{\left(2(k-m)+2\right)}_{|_{t=0}}\;b_{m-l}\;b_{l}-\frac{1}{2k+4}\sum_{l=1}^{k}\binom{2k+2}{2l}b_{k+1-l}\;b_{l}\nonumber\\
    &-\frac{k+1}{2k+4}\left(\mathcal{P}\mathcal{Q}\right)^{(2k+1)}|_{|_{t=0}}+\frac{k+1}{2k+4}\sum_{m=1}^{k}\sum_{l=0}^{m-1}\binom{2k+1}{2m}\binom{2m} {2l+1}\left(\mathcal{P}\mathcal{Q}\right)^{\left(2(k-m)+1\right)}|_{|_{t=0}}\;a_{m-l-1}\; a_{l}\nonumber\\
    &-\frac{k+1}{k+2}\sum_{m=0}^{k}\binom{2k+1} {2m}\mathcal{Q}^{\left(2(k-m)+1\right)}|_{|_{t=0}}b_{m}.\nonumber
\end{align}

$\newline$
The result follows by induction.
\end{proof}
We now discuss how to use this formal series in order to obtain an honest solution of the system (\ref{InvSpin(7)InstantonEqSO(3)}). For ease of exposition, we introduce the following functions:
\begin{align}
F_{1}(t,u,v)&\defeq u\left(\mathcal{P}(t)\; v-1\right),\nonumber\\
F_{2}(t,v)&\defeq\frac{\mathcal{P}(t)}{2}(1-v^{2}),\nonumber\\
F_{3}(t,u,v)&\defeq -\frac{\mathcal{P}(t)\; \mathcal{Q}(t)}{2}(1-u^{2})-\mathcal{Q}(t)\; v.\nonumber
\end{align}
We rewrite the ODE system (\ref{InvSpin(7)InstantonEqSO(3)}) as:
\begin{align}
\frac{da}{dt}&=\frac{F_{1}\left(t,a,b\right)}{t},\nonumber\\
\frac{db}{dt}&=\frac{F_{2}\left(t,b\right)}{t}+F_{3}\left(t,a,b\right).\nonumber
\end{align}
Further, we let $p^{a}_{m}(t,a_{0})$, $p^{b}_{m}(t,a_{0})$ denote the order $m$ Taylor polynomials corresponding to the initial data $a_{0}$. These are obtained by truncating the respective series. We also introduce the following error functions capturing the failure of the Taylor polynomials to solve (\ref{InvSpin(7)InstantonEqSO(3)}):
\begin{align}
&E_{m}^{a}(t,a_{0})\defeq\frac{d}{dt}p_{a}^{m}(t,a_{0})-\frac{F_{1}\left(t,p_{m}^{a}(t,a_{0}),p_{m}^{b}(t,a_{0})\right)}{t},\nonumber\\
&E_{m}^{b}(t,a_{0})\defeq\frac{d}{dt}p_{b}^{m}(t,a_{0})-\frac{F_{2}\left(t,p_{m}^{b}(t,a_{0})\right)}{t}-F_{3}(t,p_{m}^{a}(t,a_{0}),p_{m}^{b}(t,a_{0})).\nonumber
\end{align}
They are smooth and $O(t^{m})$ at $t=0$. To see this, recall that the full formal series was constructed by matching derivatives at the origin. Consequently, the first $m-1$ derivatives of the error functions vanish at $t=0$.

We now introduce the Banach spaces we will be working with. We define:
\begin{equation}
    \mathcal{O}_{T}(m)\defeq\left\{f\in C^{0}[0,T]\;s.t.\;\sup_{t\in [0,T]}\frac{|f(t)|}{t^{m}}<\infty\right\},
\end{equation}
\begin{equation}\nonumber
    \norm{f}_{\mathcal{O}_{T}(m)}\defeq\sup_{t\in [0,T]}\frac{|f(t)|}{t^{m}}.
\end{equation}
We immediately observe that the error functions $E^{a}_{m}, E^{b}_{m}$ lie in this space (they are $O(t^{m})$). Furthermore, the functions $p^{a}_{m},\;p^{b}_{m}-1$ lie in $\mathcal{O}(1)$. In fact -in light of proposition (\ref{CoeffCts})- we have:
\begin{corollary}\label{CtyIntoOm}
$E^{a}_{m}(t,\cdot),\; E^{b}_{m}(t,\cdot)$ define continuous mappings from the space of initial data into $\mathcal{O}(m)$. Similarly, $p^{a}_{m}(t,\cdot),\;p^{b}_{m}(t,\cdot)-1$ define continuous mappings from the space of initial data into $\mathcal{O}(1)$.
\end{corollary}
We finally recast the problem as an integral equation for a perturbation of the polynomials $(p_{m}^{a},p_{m}^{b})$. Given a pair of functions $(u,v)\in\mathcal{O}^{\oplus 2}_{T}(m)$ we define:
\begin{align}
    \Theta^{1}_{m,a_{0}}\left(u,v\right)(s)&\defeq\int_{0}^{s}\left(\frac{F_{1}\left(t,p^{a}_{m}(a_{0},t)+u(t),p^{b}_{m}(a_{0},t)+v(t)\right)}{t}-\dt{p}^{a}_{m}(a_{0},t)\right)dt,\nonumber\\
    \Theta^{2}_{m,a_{0}}\left(u,v\right)(s)&\defeq\int_{0}^{s}\left(\frac{F_{2}\left(t,p_{m}^{b}(a_{0},t)+v(t)\right)}{t}+F_{3}(t,p^{a}_{m}(a_{0},t)+u(t),p^{b}_{m}(a_{0},t)+v(t))-\dt{p}^{b}_{m}(a_{0},t)\right)dt.\nonumber
\end{align}
It can be easily checked (by expanding out the integrands, counting order of vanishing and noting that integration raises this by one) that we obtain a nonlinear integral operator:
\begin{equation}
    \Theta_{m,a_{0}}\defeq \Theta_{m,a_{0}}^{1}\times\Theta_{m,a_{0}}^{2}:\mathcal{O}^{\oplus 2}_{T}(m)\to\mathcal{O}^{\oplus 2}_{T}(m).
\end{equation}
The following proposition is the heart of the matter:
\begin{proposition}\label{FixedPoint}
Let $a_{0}$ be fixed. Fix $R>0$. For sufficiently large $m$ (depending on $F_{1},F_{2},F_{3},R$) and sufficiently small $T$ (depending on $m$ and $a_{0}$), the operator $\Theta_{m,a_{0}}$ has a unique fixed point $(u,v)$ in $\overline{B}_{R}(0)\subset\mathcal{O}^{\oplus 2}_{T}(m)$. Furthermore, this fixed point is smooth in $[0,T]$ and the associated solution
\begin{equation}\nonumber
(a,b)\defeq(p^{a}_{m}+u,p_{m}^{b}+v)
\end{equation}
to the system (\ref{Spin(7)InstantonEquation}) satisfies (\ref{CondAZero}), (\ref{CondDerAZero}), (\ref{CondBZero}).
\end{proposition}

\begin{proof}
In what follows, our notation suppresses dependence on $a_{0}$. Fix $R>0$. We will select $m$ and $T$ such that $\Theta_{m}$ is a contraction on $\overline{B}_{R}(0)\subset\mathcal{O}^{\oplus 2}_{T}(m)$.

$\newline$
Consider the domain:
\begin{equation}
D_{R}\defeq[0,1]\times\overline{B}_{2R}(0,1)\subset\mathbb{R}^{3}.
\end{equation}
Let $L>0$ be a Lipschitz constant in the $(u,v)$ variables for the restrictions of $F_{1},F_{2},F_{3}$ on $D_{R}$. L is controlled by $L^{\infty}$ bounds on the restrictions of the derivatives of the $F_{i}$ on $D_{R}$. Choose:
\begin{equation}\nonumber
m>\max\left\{2L,1\right\}.
\end{equation}
Pick $T$ such that:
\begin{equation}\nonumber
    T<\min\left\{1,\frac{(m+1)R}{\norm{E^{a}_{m}}_{\mathcal{O}_{1}(m)}+\norm{E^{b}_{m}}_{\mathcal{O}_{1}(m)}},\frac{R}{\norm{p^{a}_{m}}_{\mathcal{O}_{1}(m)}},\frac{R}{\norm{p^{b}_{m}-1}_{\mathcal{O}_{1}(m)}}\right\}.
\end{equation}
Clearly, for $t\in [0,T]$ we have:
\begin{equation}\label{smallpab}
|p^{a}_{m}(t)|\le R,\;\;\;|p^{b}_{m}(t)-1|\le R.
\end{equation}
Claim that we also have:
\begin{equation}\label{ThetaSmallAt0}
\norm{\Theta_{m}(0,0)}_{\mathcal{O}_{T}(m)}\le \frac{R}{2}.
\end{equation}
To see this, we estimate as follows:
\begin{align}
    \norm{\Theta_{m}(0,0)}_{\mathcal{O}_{T}(m)}&=\norm{\Theta^{1}_{m}(0,0)}_{\mathcal{O}_{T}(m)}+\norm{\Theta^{2}_{m}(0,0)}_{\mathcal{O}_{T}(m)}\nonumber\\
    &=\norm{\int_{0}^{r}E^{a}_{m}(t)dt}_{\mathcal{O}_{T}(m)}+\norm{\int_{0}^{r}E^{b}_{m}(t)dt}_{\mathcal{O}_{T}(m)}\nonumber\\
    &\le\sup_{r\in[0,T]}\frac{1}{r^{m}}\int_{0}^{r}|E^{a}_{m}(t)|dt+\sup_{r\in[0,T]}\frac{1}{r^{m}}\int_{0}^{r}|E^{b}_{m}(t)|dt\nonumber\\
     &\le\sup_{r\in[0,T]}\frac{\norm{E^{a}_{m}}_{\mathcal{O}_{T}(m)}}{r^{m}}\int_{0}^{r}t^{m}dt+\sup_{r\in[0,T]}\frac{\norm{E^{b}_{m}}_{\mathcal{O}_{T}(m)}}{r^{m}}\int_{0}^{r}t^{m}dt\nonumber\\
     &\le\sup_{r\in[0,T]}\frac{\norm{E^{a}_{m}}_{\mathcal{O}_{T}(m)}}{m+1}r+\sup_{r\in[0,T]}\frac{\norm{E^{b}_{m}}_{\mathcal{O}_{T}(m)}}{m+1}r\nonumber\\
     &\le\frac{\norm{E^{a}_{m}}_{\mathcal{O}_{T}(m)}+\norm{E^{b}_{m}}_{\mathcal{O}_{T}(m)}}{(m+1)}T\le \frac{R}{2}.
\end{align}
We now prove contraction estimates for $\Theta^{1}_{m}$ and $\Theta^{2}_{m}$. Fix $0\le r\le T$ and compute:
\begin{align}
    \left|\Theta^{1}_{m}(u,v)(r)-\Theta^{1}_{m}(\widetilde{u},\widetilde{v})(r)\right|&\le\int_{0}^{r}\frac{1}{t}\left|F_{1}\left(t,p^{a}_{m}+u,p^{b}_{m}+v\right)-F_{1}\left(t,p^{a}_{m}+\widetilde{u},p^{b}_{m}+\widetilde{v}\right)\right|dt\nonumber\\
    &\le\int_{0}^{r}\frac{L}{t}\left(|u-\widetilde{u}|+|v-\widetilde{v}|\right)dt\nonumber\\
    &\le L\left(\norm{u-\widetilde{u}}_{\mathcal{O}_{T}(m)}+\norm{v-\widetilde{v}}_{\mathcal{O}_{T}(m)}\right)\int_{0}^{r}t^{m-1}dt\nonumber\\
    &\le\frac{Lr^{m}}{m}\left(\norm{u-\widetilde{u}}_{\mathcal{O}_{T}(m)}+\norm{v-\widetilde{v}}_{\mathcal{O}_{T}(m)}\right).\nonumber
\end{align}
In this calculation, the $L$-Lipschitz estimate is valid due to (\ref{smallpab}) and the fact that the uniform norm is controlled by the $\mathcal{O}_{T}(m)$ norm when $0<T<1$. We conclude that:
\begin{align}\label{Theta1Contract}
\norm{\Theta^{1}_{m}(u,v)-\Theta^{1}_{m}(\widetilde{u},\widetilde{v})}_{\mathcal{O}_{T}(m)}&\le\frac{L}{m}\left(\norm{u-\widetilde{u}}_{\mathcal{O}_{T}(m)}+\norm{v-\widetilde{v}}_{\mathcal{O}_{T}(m)}\right).
\end{align}
A similar calculation yields:
\begin{align}\label{Theta2Contract}
\norm{\Theta^{1}_{m}(u,v)-\Theta^{1}_{m}(\widetilde{u},\widetilde{v})}_{\mathcal{O}_{T}(m)}&\le\left(\frac{L}{m}+\frac{LT}{m+1}\right)\norm{v-\widetilde{v}}_{\mathcal{O}_{T}(m)}+\frac{LT}{m+1}\norm{u-\widetilde{u}}_{\mathcal{O}_{T}(m)}\nonumber\\
&\le\frac{1}{2}\left(\norm{u-\widetilde{u}}_{\mathcal{O}_{T}(m)}+\norm{v-\widetilde{v}}_{\mathcal{O}_{T}(m)}\right).
\end{align}
Due to (\ref{ThetaSmallAt0}), (\ref{Theta1Contract}) and (\ref{Theta2Contract}), the closed $R$-ball in $\mathcal{O}^{\oplus 2}_{T}(m)$ is stable under $\Theta$. The contraction mapping theorem (CMT) yields a unique fixed point $(u,v)$ in this ball.

$\newline$
This fixed point is necessarily of class $C^{1}[0,T]$ (by the fundamental theorem of calculus). Consequently $(a,b)$ is $C^{1}$ and it therefore constitutes an honest solution of (\ref{InvSpin(7)InstantonEqSO(3)}) on $[0,T]$. Considering the order of vanishing of $u$ at $0$ and looking at the equations, we observe that $\dt{u}(t)=O(t^{m-1})$. Conditions (\ref{CondAZero}), (\ref{CondDerAZero}), (\ref{CondBZero}) follow.

$\newline$
Full regularity follows by a simple bootstrap procedure. Since flows of smooth (non-autonomous) vector fields are smooth, $(u,v)$ is smooth in $(0,T]$. The task is to establish smoothness at $0$. Smoothness in $(0,T]$ legitimizes differentiation of the equations for $t>0$. This gives an expression for the second derivatives of $u$ and $v$ involving terms in $\frac{u}{t^{2}}$, $\frac{v}{t^{2}}$, $\frac{\dt{u}}{t}$ and $\frac{\dt{v}}{t}$. It is thus clear that $u^{(2)}(t),\;v^{(2)}(t)\to 0$ as $t\to 0$. Hence $u,v$ are of class $C^{2}[0,T]$ with vanishing second derivative at $0$. We can iterate this argument to conclude that $u,v$ are of class $C^{m-1}[0,T]$ with vanishing derivatives at $0$ up to order $m-1$. The only constraint on $m$ required for the contraction argument to run is $m>\max\left\{2L,1\right\}$. It follows that the operator $\Theta_{l}$ is a contraction for arbitrarily large $l>m$ (perhaps for shorter time $T$). Fixing $l>m$, we let $(u_{l},v_{l})$ be the associated fixed point. Repeating the argument above, it lies in $C^{l-1}[0,T]$ with vanishing derivatives up to order $l-1$. It is thus $O(t^{m+1})$. It follows that $(u_{l}+p_{l}^{a}-p_{m}^{a},v_{l}+p_{l}^{b}-p_{m}^{b})$ is also $O(t^{m+1})$. Consequently -by further decreasing $T$ as necessary- we can arrange that the latter has as small $\mathcal{O}^{\oplus 2}_{T}(m)$ norm as we like. In particular, we take this to be less than $R$. Furthermore, $(u_{l}+p_{l}^{a}-p_{m}^{a},v_{l}+p_{l}^{b}-p_{m}^{b})$ is a fixed point of $\Theta_{m}$. But $\Theta_{m}$ has a unique fixed point in the closed $R$-ball. It follows that:
\begin{equation}
    (u,v)=(u_{l}+p_{l}^{a}-p_{m}^{a},v_{l}+p_{l}^{b}-p_{m}^{b})
\end{equation}
and hence that $u,v$ lie in $C^{l-1}$. Since $l$ was arbitrary, the proof is complete.
\end{proof}
We now have enough for the first part of theorem \ref{ShortTime}. The preceding proposition guarantees the existence of a smooth solution $(a,b)$ satisfying (\ref{CondAZero}), (\ref{CondDerAZero}), (\ref{CondBZero}). The algebraic calculation in the start of this section uniquely specifies its full formal Taylor series at $t=0$ so that it passes the extension criterion in proposition \ref{PPi1ExtendCondition}. Finally, suppose that there is another smooth solution $(\widetilde{a},\widetilde{b})$ satisfying (\ref{CondAZero}), (\ref{CondDerAZero}), (\ref{CondBZero}). Arguing as above, we find that the two solutions share the same formal Taylor series at $0$ (the series discovered in proposition \ref{FormalSeries}). Let $m$ be as in proposition \ref{FixedPoint}. We have that $(a-p^{a}_{m},b-p^{b}_{m})$, $(\widetilde{a}-p^{a}_{m},\widetilde{b}-p^{b}_{m})$ are $O(t^{m+1})$. For short enough time $T$, the $\mathcal{O}^{\oplus 2}_{T}(m)$ norms of these functions are less than $R$. Since both functions are fixed points of $\Theta_{m}$ and lie in the closed $R$-ball, they are equal. Hence $(a,b)=(\widetilde{a},\widetilde{b})$.

$\newline$
It remains to study the dependence of solutions on variations of the initial data $a_{0}$. We immediately obtain:
\begin{proposition}
Fix $K>0$. We have:
\begin{equation}\nonumber
T_{K}=\inf\left\{t_{\text{max}}(a_{0})\;|\;a_{0}\in[-K,K]\right\}>0.
\end{equation}
\end{proposition}

\begin{proof}
In our existence proof, once, $R,L,m$ are fixed, $T$ needs to be controlled from above by quantities decreasing with the $\mathcal{O}_{1}(m)$ norms of the error functions and the $\mathcal{O}_{1}(1)$ norms of $p^{a}_{m}, \;p^{b}_{m}-1$. By corollary \ref{CtyIntoOm}, these norms depend continuously on $a_{0}$ and are hence bounded for $a_{0}$ in a compact set. It follows that we can choose $T$ small enough so that the contraction argument works for all $a_{0}\in [-K,K]$.
\end{proof}
Note that the contraction constant can be taken to be the same across all $a_{0}\in [-K,K]$. This is vital for the continuity proof, which we now discuss.
\begin{proposition}
The mapping defined by:
\begin{align}
[-K,K]&\to C^{0}\left(\;[0,T_{K}],\mathbb{R}^{2}\right)\nonumber\\
a_{0}&\mapsto \left(a,b\right)_{a_{0}}
\end{align}
is continuous.
\end{proposition}
\begin{proof}
Consider the trivial (infinite-rank) vector bundle over [-K,K]:
\begin{equation}\nonumber
    E\defeq[-K,K]\times\mathcal{O}^{\oplus 2}_{T_{K}}(m).
\end{equation}
The following map is fiber-preserving and continuous:
\begin{align}
    S:E&\to E,\nonumber\\
    \left(a_{0},(u,v)\right)&\mapsto\left(a_{0},\Theta_{m,a_{0}}(u,v)\right).\nonumber
\end{align}
There is a unique section $s$ of $E$ that is fixed by $S$ (the one assigning to each choice of initial data the associated fixed point of $\Theta_{a_{0},m}$). The task is to prove that $s$ is continuous. To this end, we fix $x\in [-K,K]$ and prove that $s$ is continuous at $x$. Fix $\epsilon>0$ and define the following (continuous) section of $E$:
\begin{equation}\nonumber
    u_{x}(a_{0})\defeq (a_{0},s(x)).
\end{equation}
We will run the CMT iteration on each fiber with initial condition determined by $u_{x}$. Letting $0<C<1$ be the contraction constant of $\Theta_{a_{0},m}$ and using the convergence rate estimate of the CMT we have:
\begin{align}
  \norm{\Theta^{N}_{m,a_{0}}(u_{x}(a_{0}))-s(a_{0})}_{\mathcal{O}_{T_{K}}(m)^{\oplus 2}}&\le\frac{\norm{\Theta_{m,a_{0}}(u_{x}(a_{0}))-u_{x}(a_{0})}_{\mathcal{O}_{T_{K}}^{\oplus 2}(m)}}{1-C}  C^{N}\nonumber\\
  &\le\frac{2R}{1-C}C^{N}.\label{CtyEst1}
\end{align}
Fix $N$ large enough so that this quantity is controlled by $\frac{\epsilon}{2}$. Since $S$ and $u$ are continuous, we have:
\begin{equation}\nonumber
\lim_{a_{0}\to x}S^{N}u_{x}(a_{0})=S^{N}u_{x}(x)=(x,s(x)).
\end{equation}
Consequently, for $a_{0}$ sufficiently close to $x$, we can achieve:
\begin{equation}\nonumber
\norm{\Theta^{N}_{m,a_{0}}\left(u_{x}(a_{0})\right)-s(x)}_{\mathcal{O}^{\oplus 2}_{T_{K}}(m)}<\frac{\epsilon}{2}.\label{CtyEst2}
\end{equation}
Using (\ref{CtyEst1}), (\ref{CtyEst2}) and the triangle inequality completes the proof.
\end{proof}
Uniqueness implies that the solution associated to $a_{0}=0$ corresponds to $A_{\text{HYM}_{\pi_{1}}}$. This instanton will play a central role in the analysis of the global properties of the system.
\subsubsubsection{Global Existence for Small Initial Data}
The previous section yields a characterization of short-time solutions near the pole. We are now tasked with understanding which of these solutions are global. In this section we establish that:
\begin{theorem}\label{GlobalExistence}
There exists an $\epsilon>0$ such that for $|a_{0}|<\epsilon$, the short time solutions of theorem \ref{ShortTime} are global.
\end{theorem}
The heart of the argument lies in the following proposition. Its conditions are subsequently easily verified (for small initial data) by a continuity argument.
\begin{proposition}\label{TrappedRegion}
Suppose that $a_{0}>0$ and let $(a,b)_{a_{0}}$ be a solution to the system (\ref{InvSpin(7)InstantonEqSO(3)}) such that $a$ attains a critical point in the spacetime region:
\begin{equation}\label{SpacetimeRegion}
t>\frac{\sqrt{6}}{2\sqrt{1-2a^{2}}}.
\end{equation}
Then $t_{\text{max}}(a_{0})=+\infty$.
\end{proposition}

\begin{proof}
By proposition \ref{aNotZero}, $a(t)>0$ for all $0<t<t_{\text{max}}(a_{0})$. Looking at the ODE for $a$, we conclude that the critical points of $a$ are precisely the points where $b=\mathcal{P}^{-1}$. We seek an expression for the second derivative of $a$ at a critical point occurring at time $t=t_{\text{crit}}>0$. Differentiating the ODE for $a$ and setting $b=\mathcal{P}^{-1}$, we obtain:
\begin{equation}
    \frac{d^{2}a}{dt^{2}}_{|_{t=t_{\text{crit}}}}=\frac{3\; a\left(t_{\text{crit}}\right)}{2t_{\text{crit}}^{2}\left(2t_{\text{crit}}^{2}+3\right)}\left[\left(4\; a\left(t_{\text{crit}}\right)^{2}-2\right)t_{\text{crit}}^{2}+3\right].
\end{equation}
The first factor is strictly positive. Consequently, the nature of the critical point depends on the sign of:
\begin{equation}
F(t,a)\defeq\left(4a^{2}-2\right)t^{2}+3
\end{equation}
at $\left(t_{\text{crit}},a\left(t_{\text{crit}}\right)\right)$. For $(t,a)$ in the spacetime region \ref{SpacetimeRegion}, we have $F(t,a)<0$. Hence, any critical point occurring in the region is a maximum.

$\newline$
Suppose that a maximum does occur inside the region (\ref{SpacetimeRegion}). For a short amount of time thereafter $a$ is decreasing. The only way that $a$ can ever increase again is if it reaches a minimum. A minimum can only occur if $(t,a(t))$ exits the spacetime region (\ref{SpacetimeRegion}). For this to occur, $a$ has to increase. It follows that $a$ decreases for as long as the solution survives. Consequently $a$ is bounded from above. Since $a>0$, it follows that $a$ is also bounded from below. Since $a$ consistently decreases after the maximum point, we have that $b(t)<\mathcal{P}(t)^{-1}$ for $t>t_{\text{crit}}$. By corollary \ref{bPos}, $b>0$ for all time. Hence both $a$ and $b$ are bounded and thus survive for all time $t\ge 0$.
\end{proof}
Proposition \ref{TrappedRegion} applies provided that the initial data is small enough:
\begin{proposition}\label{CriticalAttained}
There exists $\epsilon>0$ such that if $0\le a_{0}<\epsilon$, then $a_{a_{0}}$ attains a critical point in the spacetime region (\ref{SpacetimeRegion}).
\end{proposition}

\begin{proof}
The idea is to use a continuity argument and compare with the solution corresponding to $a_{0}=0$:
\begin{equation}
a_{\text{HYM}}(t)=0,\;\;\;b_{\text{HYM}}(t)=\frac{1}{\sqrt{t^{2}+1}}.\nonumber\\
\end{equation}
Note that $\mathcal{P}(0)=2$ and $b(0)=1$ (independently of the choice of $a_{0}$). Hence $b$ always starts above $\mathcal{P}^{-1}$. For $a_{0}=0$, the solution $b_{\text{HYM}}$ crosses $\mathcal{P}^{-1}$ at the time: $t=\frac{3\sqrt{2}}{2}$. For $a_{0}>0$, formulae (\ref{anear0}) and (\ref{bnear0}) show that -at least for a very short time- to the right of $t=0$ we have
\begin{equation}
\left(a,b\right)>\left(a_{\text{HYM}},b_{\text{HYM}}\right)
\end{equation}
where the inequality is understood componentwise. By proposition \ref{NoCross}, this inequality persists for as long as the solutions exist. Consequently, if $a_{0}>0$, $b_{a_{0}}$ can only cross $\mathcal{P}^{-1}$ strictly after $t=\frac{3\sqrt{2}}{2}$.

$\newline$
Consider only $|a_{0}|\le1$. By the second assertion of theorem \ref{ShortTime}, the maximal existence time of the resulting solutions is bounded below by a positive number $T_{1}$. Furthermore, these solutions depend continuously on $a_{0}$ (in the $C^{0}[0,T_{1}]$ norm). Composing with the local flow associated to taking initial conditions at $t=T_{1}$, we see that the maximal existence time is lower semicontinuous in $a_{0}$. Furthermore, we see that if a particular choice of $a_{0}$ yields a solution surviving past some time $t=T$, the mapping sending initial conditions to their associated solutions is continuous from an open neighbourhood of $a_{0}$ into $C^{0}[0,T]$.

$\newline$
Since $(a_{\text{HYM}},b_{\text{HYM}})$ (associated to $a_{0}=0$) is global, initial data close to $0$ lead to solutions that survive arbitrarily long. In particular, we can choose $\epsilon >0$ to be small enough so that solutions associated to $0<a_{0}<\epsilon$ survive past $t=4$. Furthermore -at the expense of taking $\epsilon$ to be even smaller- we can appeal to continuity to arrange that:
\begin{align}
\sup_{t\in[0,4]}\left|a_{a_{0}}(t)\right|&<\frac{1}{2},\label{Controlona}\\
\sup_{t\in[0,4]}\left|b_{a_{0}}(t)-\frac{1}{\sqrt{t^{2}+1}}\right|&<\frac{1}{2}\inf_{t\in [3,4]}\left|\frac{1}{\sqrt{t^{2}+1}}-\frac{1}{\mathcal{P}(t)}\right|.\label{Controlonb}
\end{align}
Condition (\ref{Controlona}) implies that for any $\sqrt{3}< t\le 4$ the point $(t,a(t))$ lies in the spacetime region (\ref{SpacetimeRegion}). Condition (\ref{Controlonb}) implies that for any $3\le t\le 4$ we have:
\begin{equation}\nonumber
b_{a_{0}}(t)<\frac{1}{\mathcal{P}(t)}.
\end{equation}
By the intermediate value theorem, there exists a $0<t_{\text{crit}}<3$ where $b_{a_{0}}$ crosses $\mathcal{P}^{-1}$. But we have seen that this time must be after $t=\frac{3\sqrt{2}}{2}$ and consequently after $t=\sqrt{3}$. Hence, the critical point at $t=t_{\text{crit}}$ occurs in the spacetime region (\ref{SpacetimeRegion}). 
\end{proof}
Theorem \ref{GlobalExistence} easily follows from the preceding two propositions and the symmetry of the system (\ref{InvSpin(7)InstantonEqSO(3)}) -as formulated in proposition \ref{SystemSymmetry}-.

\begin{proof}\textbf{[of Theorem \ref{GlobalExistence}]}
Proposition \ref{CriticalAttained} yields a threshold $\epsilon>0$ such that for any $0\le a_{0}<\epsilon$, the $a$ component of the associated solution attains a critical point in the region (\ref{SpacetimeRegion}). Proposition \ref{TrappedRegion} then implies that $(a,b)_{a_{0}}$ is global. Finally, proposition \ref{SystemSymmetry} proves that solutions associated to $-\epsilon<a_{0}\le 0$ are global too.
\end{proof}

\subsubsubsection{Finite Time Blowup for Large Initial Data}
We now wish to study the development of large initial data. We will obtain the following:
\begin{theorem}\label{Blowup}
Suppose that:
\begin{equation}\nonumber
    |a_{0}|>\frac{1}{2\arctanh\left(\frac{1}{2}\right)}.
\end{equation}
Then $(a,b)_{a_{0}}$ blows up in finite time at most equal to:
\begin{equation}\label{ApproxBlowupTime}
    t_{\text{blowup}}\left(a_{0}\right)\defeq \frac{3\sqrt{2}}{2}\;\frac{\left(1-\tanh^{2}\left(\frac{1}{2|a_{0}|}\right)\right)^{\frac{1}{2}}}{1-2\tanh\left(\frac{1}{2|a_{0}|}\right)}.
\end{equation}
Furthermore, the blowup set:
\begin{equation}
    \mathcal{S}_{\text{blowup}}\defeq\left\{a_{0}\in\mathbb{R}\;\text{s.t.}\;(a,b)_{a_{0}}\;\text{blows up in finite time}\right\}
\end{equation}
is of the form:
\begin{equation}
    \mathcal{S}_{\text{blowup}}=(-\infty,-x)\cup(x,\infty)
\end{equation}
for some $0<x<\frac{1}{2\arctanh\left(\frac{1}{2}\right)}$.
\end{theorem}
Our analysis relies on an apriori bound on $\frac{b}{a}$:
\begin{proposition}\label{BlowupBoundProp}
For any $a_{0}\ge 0$, the solution $(a,b)_{a_{0}}$ satisfies the following inequality for all $0\le t<t_{\text{max}}(a_{0})$:
\begin{equation}\label{BlowupBound}
    b(t)>\frac{t}{2\sqrt{t^{2}+1}}a(t).
\end{equation}
\end{proposition}
\begin{proof}
\ref{BlowupBound} is clearly satisfied at $t=0$. To show that it persists for as long as solutions survive, we let $t_{\star}$ be any time such that:
\begin{equation}\nonumber
b(t_{\star})=\frac{t_{\star}}{2\sqrt{t_{\star}^{2}+1}}a(t_{\star})
\end{equation}
and we compute:
\begin{equation}
\frac{d}{dt}_{|_{t=t_{\star}}}\left(b(t)-\frac{t}{2\sqrt{t^{2}+1}}a(t)\right)=\sqrt{6}\; \frac{b(t_{\star})^{2}\;t_{\star}^{2}+b(t_{\star})^{2}+1}{t_{\star}\sqrt{2t_{\star}^{2}+2}\sqrt{2t_{\star}^{2}+3}}>0.
\end{equation}
\end{proof}
Proposition \ref{BlowupBoundProp} allows us to estimate:
\begin{equation}\label{Comparison1BlowupSection}
    \dt{a}=\frac{\mathcal{P}a}{t}\left(b-\frac{1}{\mathcal{P}}\right)>\frac{\mathcal{P}a}{t}\left(\frac{ta}{2\sqrt{t^{2}+1}}-\frac{1}{\mathcal{P}}\right).
\end{equation}
Fix a reference time $t_{0}$. Estimate (\ref{Comparison1BlowupSection}) implies that -past $t_{0}$- $a$ is bounded below by the solution of the following I.V.P. of Riccati type:
\begin{equation}\label{RiccatiEquation}
    \begin{cases}
      & \dt{u}(t)=u(t)\left(\frac{\sqrt{3}}{\sqrt{2t^{2}+3}}u(t)-\frac{1}{t}\right),\\
      & u(t_{0})=a(t_{0}).\\
    \end{cases}
\end{equation}
Setting $x_{0}\defeq a(t_{0})$, equation (\ref{RiccatiEquation}) can be solved explicitly to give:
\begin{equation}\label{RiccatiSolution}
    u_{t_{0},x_{0}}(t)=\frac{t_{0}x_{0}}{t\left(t_{0}x_{0}\arctanh\left(\frac{\sqrt{3}}{\sqrt{2t^{2}+3}}\right)-t_{0}x_{0}\arctanh\left(\frac{\sqrt{3}}{\sqrt{2t_{0}^{2}+3}}\right)-1\right)}.
\end{equation}
The task is now to determine conditions on $t_{0},x_{0}$ such that the function $u_{t_{0},x_{0}}$ blows up in finite time. An elementary calculation demonstrates that the denominator of (\ref{RiccatiSolution}) vanishes at time:
\begin{equation}\label{ApproxBlowupTimeGeneralt}
\mathcal{T}(t_{0},x_{0})\defeq\frac{\sqrt{6}}{2}\left(1-\frac{3}{2t_{0}^{2}+3}\right)^{\frac{1}{2}}\frac{\left(1-\tanh^{2}\left(\frac{1}{t_{0}x_{0}}\right)\right)^{\frac{1}{2}}}{\frac{\sqrt{3}}{\sqrt{2t_{0}^{2}+3}}-\tanh\left(\frac{1}{t_{0}x_{0}}\right)}.
\end{equation}
We introduce the function:
\begin{equation}\label{ThresholdFunction}
\mathcal{R}\left(t\right)\defeq\frac{1}{t\arctanh\left(\frac{\sqrt{3}}{\sqrt{2t^{2}+3}}\right)}.
\end{equation}
\begin{proposition}\label{ThresholdProp}
Fix $t_{0}>0$. We have:
\begin{equation}\nonumber
    \mathcal{T}\left(t_{0},\mathcal{R}(t_{0})\right)=+\infty
\end{equation}
and $\mathcal{T}(t_{0},x)$ decreases monotonically to $t_{0}$ (as a function of $x$) for $x>\mathcal{R}(t_{0})$. In particular, we have the following pointwise limit:
\begin{equation}\nonumber
\lim_{x\to\infty}\mathcal{T}(t_{0},x)=t_{0}.
\end{equation}
\end{proposition}
\begin{proof}
The proof is an elementary explicit calculation which we omit.
\end{proof}
We provide a short interpretation of proposition \ref{ThresholdProp}. For each time $t_{0}>0$, the function $\mathcal{R}(t_{0})$ provides a threshold, such that if $u$ solves (\ref{RiccatiEquation}) and satisfies 
\begin{equation}\nonumber
u(t_{0})>\mathcal{R}(t_{0})
\end{equation}
then $u$ blows up in finite time equal to $\mathcal{T}(t_{0},u(t_{0}))>t_{0}$. Fix $t_{0}>0$. For $u(t_{0})$ close to (but above) the threshold, the blowup time can be arbitrarily large. As $u(t_{0})\to\infty$, the blowup time approaches $t_{0}$ from above. Consequently, for very large initial data, the solution survives for arbitrarily short time past $t_{0}$.

$\newline$
Since $u_{t_{0},x_{0}}$ bounds $a$ from below, we obtain:
\begin{proposition}\label{BlowupCriterionGeneralTime}
Suppose that $(a,b)_{a_{0}}$ is a solution of the system (\ref{InvSpin(7)InstantonEqSO(3)}) satisfying:
\begin{equation}\nonumber
a(t_{0})>\mathcal{R}(t_{0})\;\text{for some}\;t_{0}>0.
\end{equation}
Then $a$ blows up to $+\infty$ in finite time at most equal to $\mathcal{T}(t_{0},a(t_{0}))$.
\end{proposition}
The task is to verify that for large initial data $a_{0}$, the $a$-component of the solution eventually crosses the threshold $\mathcal{R}$, depicted below:
\begin{center}
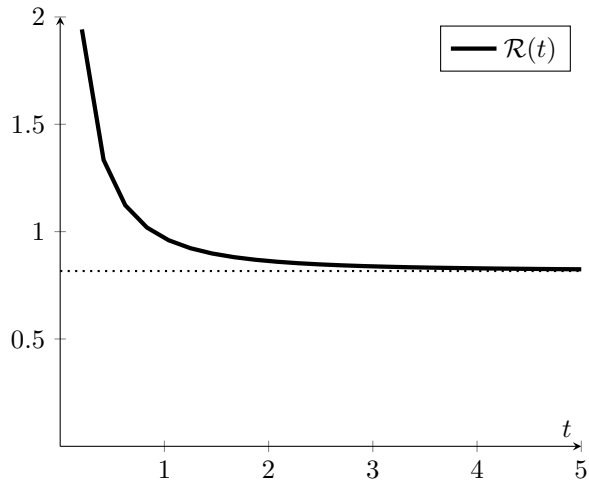

    \begin{tikzpicture}
        \begin{axis}[
            xmin=0, xmax=5,
            ymin=0, ymax=2,
            axis lines=center,
            axis on top=true,
            domain=0:5,
            xlabel=$t$,
            ]

            \addplot [mark=none,draw=black,ultra thick] {Threshold(\x)};
            \addlegendentry{$\mathcal{R}(t)$}
            \draw [black, dotted, thick] (axis cs:0,0.81649658091)--(axis cs:5,0.81649658091);
        \end{axis}
    \end{tikzpicture}
    \captionof{figure}{Graph of the threshold function $\mathcal{R}$.}
\label{plot:Threshold Function Plot}
\end{center}
We will use the reference time $t_{0}=\frac{3\sqrt{2}}{2}$. This is the time where $b_{\text{HYM}}$ crosses $\mathcal{P}^{-1}$. We are able to obtain the following bound:
\begin{proposition}\label{aLargeAtChosenReference}
Fix $a_{0}>0$. Let $(a,b)_{a_{0}}$ be the development of the initial data $a_{0}$. We have:
\begin{equation}\label{aLargeAtChosenReferenceEstimate}
a\left(\frac{3\sqrt{2}}{2}\right)>\frac{2\sqrt{2}}{3}a_{0}.
\end{equation}
\end{proposition}
\begin{proof}
Fix $a_{0}>0$. Formula (\ref{bnear0}) demonstrates that -at least for a short time-, $b_{a_{0}}$ exceeds $b_{\text{HYM}}$ to the right of $t=0$. Proposition \ref{NoCross} establishes that $b_{a_{0}}>b_{\text{HYM}}$ until $t=t_{\text{max}}(a_{0})$. Incorporating this bound with the ODE governing $a$, we estimate:
\begin{align}
    \dt{a}&=\frac{\mathcal{P}a}{t}\left(b-\frac{1}{\mathcal{P}}\right)\nonumber\\
    &>\frac{\mathcal{P}a}{t}\left(\frac{1}{\sqrt{t^{2}+1}}-\frac{1}{\mathcal{P}}\right).\nonumber
\end{align}
Simplifying, we conclude that $a$ is bounded below by solutions of the following (singular) I.V.P:
\begin{equation}\label{ODEaExceedsThreshold}
    \begin{cases}
      & \dt{v}(t)=\frac{v(t)}{t}\left(\frac{2\sqrt{3}}{\sqrt{2t^{2}+3}}-1\right),\\
      & v(0)=0,\nonumber\\
      & \dt{v}(0)=a_{0}.
    \end{cases}
\end{equation}
The problem is well-posed (solutions exist and are uniquely determined by the prescribed initial data) and $v$ takes the form:
\begin{equation}\label{ODEaExceedsThresholdSol}
v_{a_{0}}(t)=\frac{36\;a_{0}\;t}{\left(3+\sqrt{6t^{2}+9}\right)^{2}}.
\end{equation}
The function $v_{a_{0}}$ has a global maximum at time $t=t_{0}=\frac{3\sqrt{2}}{2}$ with value $\frac{2\sqrt{2}}{3}a_{0}$.
\end{proof}
The upshot is that by choosing $a_{0}$ to be sufficiently large, we can arrange that $a(t_{0})$ exceeds any number we like. In particular, we can arrange that $a(t_{0})$ exceeds the threshold $\mathcal{R}(t_{0})$.

$\newline$
We now have enough to complete the proof of theorem \ref{Blowup}.
\begin{proof}
Evaluating (\ref{ApproxBlowupTimeGeneralt}) and (\ref{ThresholdFunction}) at the reference time $t_{0}=\frac{3\sqrt{2}}{2}$ we obtain:
\begin{equation}\label{BlowupAtChosenReference}
\mathcal{T}\left(t_{0},x\right)=\frac{3\sqrt{2}}{2}\frac{\left(1-\tanh^{2}\left(\frac{\sqrt{2}}{3x}\right)\right)^{\frac{1}{2}}}{1-2\tanh\left(\frac{\sqrt{2}}{3x}\right)},\;\;\;\mathcal{R}\left(t_{0}\right)=\frac{\sqrt{2}}{3\arctanh\left(\frac{1}{2}\right)}.
\end{equation}
Let $(a,b)$ be a solution of the system (\ref{InvSpin(7)InstantonEqSO(3)}) satisfying:
\begin{equation}\label{BlowupCriterion}
a\left(t_{0}\right)>\frac{\sqrt{2}}{3\arctanh\left(\frac{1}{2}\right)}.
\end{equation}
Using (\ref{BlowupAtChosenReference}) and proposition \ref{BlowupCriterionGeneralTime}, we conclude that the solution blows up to $+\infty$ in finite time at most equal to $\mathcal{T}\left(t_{0},a\left(t_{0}\right)\right)$. Proposition \ref{aLargeAtChosenReference} guarantees that (\ref{BlowupCriterion}) is satisfied provided that we take:
\begin{equation}\label{AugmentedBlowupCriterion}
a_{0}>\frac{1}{2\arctanh\left(\frac{1}{2}\right)}.
\end{equation}
By proposition \ref{ThresholdFunction}, the function $\mathcal{T}(t_{0},x)$ is monotonic in $x$ provided that $x>\mathcal{R}(t_{0})$. Condition (\ref{AugmentedBlowupCriterion}) guarantees that the right hand side of (\ref{aLargeAtChosenReferenceEstimate}) exceeds $\mathcal{R}\left(t_{0}\right)$ and is thus large enough for the monotonicity statement to apply. We obtain:
\begin{equation}\nonumber
    \mathcal{T}\left(t_{0},a\left(t_{0}\right)\right)<\mathcal{T}\left(t_{0},\frac{2\sqrt{2}}{3}a_{0}\right)=\frac{3\sqrt{2}}{2}\;\frac{\left(1-\tanh^{2}\left(\frac{1}{2|a_{0}|}\right)\right)^{\frac{1}{2}}}{1-2\tanh\left(\frac{1}{2|a_{0}|}\right)}.
\end{equation}
Defining $t_{\text{blowup}}(a_{0})$ to be equal to the right hand side of this inequality, we have established the first assertion of theorem \ref{Blowup}.

$\newline$
Define the positive and negative blowup sets as:
\begin{align}\nonumber
    \mathcal{S}^{+}_{\text{blowup}}&\defeq\left\{a_{0}\in\mathbb{R}\;\text{s.t.}\;a_{a_{0}}\;\text{blows up to }+\infty\text{ in finite time}\right\}\nonumber\\
    \mathcal{S}^{-}_{\text{blowup}}&\defeq\left\{a_{0}\in\mathbb{R}\;\text{s.t.}\;a_{a_{0}}\;\text{blows up to }-\infty\text{  in finite time}\right\}\nonumber
\end{align}
so that:
\begin{equation}\nonumber
   \mathcal{S}_{\text{blowup}}=  \mathcal{S}_{\text{blowup}}^{+}\cup  \mathcal{S}_{\text{blowup}}^{-}.
\end{equation}
The last assertion of theorem \ref{Blowup} will follow from proposition \ref{SystemSymmetry} if we establish the existence of $x>0$ such that:
\begin{equation}\label{PosBlowupSetChar}
    \mathcal{S}_{\text{blowup}}^{+}=(x,\infty).
\end{equation}
We first prove that the positive blowup set is open. Let $a_{0}\in\mathcal{S}^{+}_{\text{blowup}}$ and let $t_{\star}$ be the blowup time of the associated solution. By definition:
\begin{equation}\nonumber
\lim_{t\to t_{\star}}a(t)=+\infty.
\end{equation}
Consequently, there is a time $T\in [\frac{t_{\star}}{2},t_{\star})$ such that:
\begin{equation}\nonumber
    a\left(T\right)>2\sup_{t\in[\frac{t_{\star}}{2},t_{\star}]}\mathcal{R}(t).
\end{equation}
By continuity with respect to variation of the initial data, we obtain a $\delta>0$ such that for $\widetilde{a_{0}}\in(a_{0}-\delta,a_{0}+\delta)$:
\begin{equation}\nonumber
    \frac{a(T)}{2}<\widetilde{a}(T)<\frac{3}{2}a(T).
\end{equation}
Consequently:
\begin{equation}
    \widetilde{a}(T)>\sup_{t\in[\frac{t_{\star}}{2},t_{\star}]}\mathcal{R}(t)\ge \mathcal{R}(T).
\end{equation}
By proposition \ref{ThresholdProp}, the choice $\widetilde{a_{0}}$ leads to finite-time blowup and $\mathcal{S}^{+}_{\text{blowup}}$ is indeed open.

$\newline$
Finally, by proposition \ref{NoCross}, if a certain choice of $a_{0}>0$ leads to finite-time blowup, so do all $\widetilde{a_{0}}>a_{0}$. Together with openness, this property yields (\ref{PosBlowupSetChar}) for some $x\ge 0$. Theorem \ref{GlobalExistence} implies that $x>0$.
\end{proof}

\subsubsection{The Moduli Space}
The results of the preceding sections are sufficient to obtain a complete description of the moduli space of SO(5) invariant Spin(7) instantons with structure group SO(3) on the Stenzel manifold. We denote this object as $\mathcal{M}^{\text{Spin}(7)}_{\text{inv}}\left(X^{8}\right)$. The trivial bundle $P_{1}$ doesn't contribute to this moduli space. This is due to the nonexistence theorem \ref{NonexistenceTrivialBundle}. At the risk of being pedantic, we are ignoring the trivial solution $A=0$.

$\newline$
Let $P$ be a $G$-homogeneous (or cohomogeneity one) principal $S$-bundle. There are two natural ways to set up a moduli space of $G$-invariant solutions to a gauge-theoretic problem on $P$. One is to quotient the set of invariant solutions by the group of equivariant gauge transformations. The other is carried out in two steps. Initially one quotients the set of all (not necessarily invariant) solutions by the set of all (not necessarily equivariant) gauge transformations. The action of $G$ on the total space $P$ induces an action on the set of all connections. This action restricts to the set of solutions and passes to the quotient. The moduli space is then defined to be the $G$-invariant locus. There is an obvious map from the first construction to the second construction. If the fiber $S$ is semisimple and we restrict attention to irreducible connections, this map is a homeomorphism.

$\newline$
In our setting, the structure group is SO(3) (which is indeed semisimple) and furthermore, all solutions are irreducible. It follows that the two constructions coincide. We will follow the first. Recall that each invariant connection constitutes its own equivariant gauge equivalence class. Consequently, the moduli spaces on the individual bundles are:
\begin{align}
\mathcal{M}\left(P_{\pi_{1}}\right)&\defeq\left\{A\in\mathcal{A}_{\text{inv}}\left(P_{\pi_{1}}\right)\;\text{s.t.}\;\star_{g}F_{A}=-\Phi\wedge F_{A}\right\},\nonumber\\
\mathcal{M}\left(P_{\pi_{2}}\right)&\defeq\left\{A\in\mathcal{A}_{\text{inv}}\left(P_{\pi_{2}}\right)\;\text{s.t.}\;\star_{g}F_{A}=-\Phi\wedge F_{A}\right\}.\nonumber
\end{align}
Due to the results of section \ref{Pi1Sols}, we have that $\mathcal{M}\left(P_{\pi_{1}}\right)$ is a compact interval. It can be parameterized by initial conditions $a_{0}=\dt{a}(0)$ leading to global solutions. Using this parameterization, theorem \ref{Blowup} gives us a number $x>0$ such that:
\begin{equation}\label{ModuliPi1}
   \mathcal{M}\left(P_{\pi_{1}}\right)\cong[-x,x] .
\end{equation}
$\mathcal{M}\left(P_{\pi_{1}}\right)$ contains a unique HYM connection $A_{\text{HYM}_{\pi_{1}}}$ corresponding to $a_{0}=0$. It is represented by the red dot in the following diagram. The black dots represent the boundary points $\pm x$.
\begin{center}
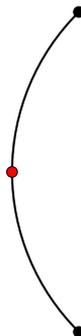

    \begin{tikzpicture}
        \draw[thick,black] (4,0) arc (180:135:3cm);
        \draw[thick,black] (4,0) arc (180:225:3cm);
        \draw[fill=black] (4.87868,2.12132) circle[radius=2pt];
        \draw[fill=black] (4.87868,-2.12132) circle[radius=2pt];
        \draw[fill=red] (4,0) circle[radius=2pt];
    \end{tikzpicture}
\captionof{figure}{The Moduli Space $\mathcal{M}\left(P_{\pi_{1}}\right)$}
\label{plot:Moduli Space 1}
\end{center}
Due to the results of section \ref{Pi2Sols}, we have that $\mathcal{M}\left(P_{\pi_{2}}\right)$ is a half-open half-closed interval. We can parameterize it by the value $\nu=b(t_{0})$ at time $t_{0}=\frac{\sqrt{6}}{2}$. Using this parameterization and setting:
\begin{equation}
    \nu_{1}\defeq-\frac{2\sqrt{5}}{5},\;\;\;\nu_{2}\defeq\frac{\sqrt{10}}{5},\nonumber\\
\end{equation}
we have that:
\begin{equation}\label{ModuliPi2}
   \mathcal{M}\left(P_{\pi_{2}}\right)=\left\{A_{\nu}\;\text{s.t.}\;\nu\in[\nu_{1},\nu_{2})\right\}\cong[\nu_{1},\nu_{2}).
\end{equation}
$\mathcal{M}\left(P_{\pi_{2}}\right)$ contains a unique HYM connection $A_{\text{HYM}_{\pi_{2}}}$ corresponding to $\nu=-\frac{\sqrt{10}}{5}$. It is represented by the green dot in the following diagram. The black dot represents the boundary point $\nu_{1}$.
\begin{center}
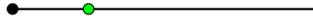

    \begin{tikzpicture}
        \draw[thick,black] (0,0) -- (4,0);
        \draw[fill=green] (1,0) circle[radius=2pt];
        \draw[fill=black] (0,0) circle[radius=2pt];
    \end{tikzpicture}
\captionof{figure}{The Moduli Space $\mathcal{M}\left(P_{\pi_{2}}\right)$}
\label{plot:Moduli Space 2}
\end{center}

$\newline$
We observe that $\mathcal{M}\left(P_{\pi_{2}}\right)$ is not compact. Interestingly, it admits a natural compactification. To understand the noncompactness phenomenon, we study the (missing) limit $\nu\to\frac{\sqrt{10}}{5}$. To identify what the limit should be we work on $X^{8}-S^{4}$. Using the explicit formula (\ref{bPi2}) with $\nu=\frac{\sqrt{10}}{5}$ yields the HYM connection $A_{\text{HYM}_{\pi_{1}}}$. We conclude that (over $X^{8}-S^{4}$):
\begin{equation}
\lim_{\nu\to\nu_{2}}A_{\nu}=A_{\text{HYM}_{\pi_{1}}}.
\end{equation}
This can be understood pointwise -with a choice of some background reference connection- or in a suitable weighted norm. 

$\newline$
We conclude that the Spin(7) instantons  in $\mathcal{M}\left(P_{\pi_{2}}\right)$ are trying to converge to the (unique) HYM connection of $\mathcal{M}\left(P_{\pi_{1}}\right)$, but fail to do so as this connection does not smoothly extend to the bundle on which they live. Notably the singularity happens around a codimension $4$ Cayley submanifold (the singular orbit $S^{4}$). This reasoning motivates us to glue in $\mathcal{M}\left(P_{\pi_{1}}\right)$, by forcing the point $a_{0}=0$ to be the missing endpoint of $\mathcal{M}\left(P_{\pi_{2}}\right)$. This leads to the following picture of $\mathcal{M}^{\text{Spin}(7)}_{\text{inv}}\left(X^{8}\right)$:
\begin{center}
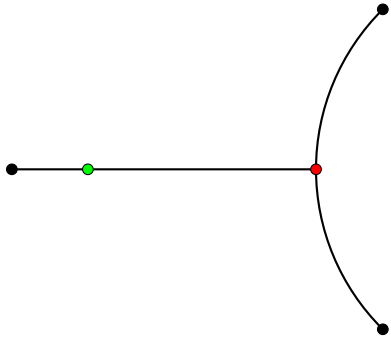

    \begin{tikzpicture}
        \draw[thick,black] (0,0) -- (4,0);
        \draw[thick,black] (4,0) arc (180:135:3cm);
        \draw[thick,black] (4,0) arc (180:225:3cm);
        \draw[fill=black] (0,0) circle[radius=2pt];
        \draw[fill=black] (4.87868,2.12132) circle[radius=2pt];
        \draw[fill=black] (4.87868,-2.12132) circle[radius=2pt];
        \draw[fill=green] (1,0) circle[radius=2pt];
        \draw[fill=red] (4,0) circle[radius=2pt];
    \end{tikzpicture}
\captionof{figure}{The Moduli Space $\mathcal{M}^{\text{Spin}(7)}_{\text{inv}}\left(X^{8}\right)$}
\label{plot:Moduli Space Full}
\end{center}

$\newline$
The notation in this diagram is consistent with figures \ref{plot:Moduli Space 1} and \ref{plot:Moduli Space 2}. Crucially, the space $\mathcal{M}^{\text{Spin}(7)}_{\text{inv}}\left(X^{8}\right)$ is compact.

$\newline$
This suggests a relationship between the Spin(7) instantons and the HYM connections. Indeed, they are not equivalent in general; but furthermore, the structure of $\mathcal{M}^{\text{Spin}(7)}_{\text{inv}}\left(X^{8}\right)$ hints that the latter might play a role in the compactification of Spin(7) instanton moduli spaces (over noncompact Calabi-Yau 4-folds). It may be a general phenomenon that certain sequences of Spin(7) instantons fail to converge because the limit lives on a different bundle. Furthermore this new bundle should agree with the original one outside of a (codimension 4) Cayley submanifold. The missing limit should be HYM.

\newpage
\section*{Acknowledgements}
The author would like to thank his PhD advisors Andrew Dancer and Jason Lotay for suggesting the problem discussed in this article and providing numerous helpful comments and recommendations. He would also like to thank Gonçalo Oliveira for an elucidating discussion regarding the issue of gauging in the equivariant setting. This work was supported by the EPSRC Centre for Doctoral Training in Partial Differential Equations: Analysis and Applications (grant number EP/L015811/1) . 
\bibliographystyle{plain} 
\bibliography{bibliography} 

\begin{thebibliography}{10}

\bibitem{Esch}
J.~H. Eschenburg and McKenzie~Y. Wang.
\newblock {T}he {I}nitial {V}alue {P}roblem for {C}ohomogeneity {O}ne
  {E}instein {M}etrics.
\newblock {\em The Journal of Geometric Analysis}, 10(1):109--137, Mar 2000.

\bibitem{SWYM}
S.~T. Yau \& J. B.~McLeod Joel A.~Smoller, Arthur G.~Wasserman.
\newblock {S}mooth {S}tatic {S}olutions of the {E}instein/{Y}ang-{M}ills
  {E}quations.
\newblock {\em Communications in Mathematical Physics}, 143(S1):115--147, 1991.

\bibitem{Joyce}
Dominic Joyce.
\newblock {\em {R}iemannian {H}olonomy {G}roups and {C}alibrated {G}eometry}.
\newblock Oxford University Press, Oxford, 2007.

\bibitem{Kob}
Kobayashi and Nomizu.
\newblock {\em {F}oundations of {D}ifferential {G}eometry Vol. 1}.
\newblock Wiley, 1963.

\bibitem{Lewis}
C.~Lewis.
\newblock {\em \emph{{S}pin(7)} {I}nstantons}.
\newblock PhD thesis, The University of Oxford, 1998.

\bibitem{Lot}
Jason~D. Lotay and Goncalo Oliveira.
\newblock $\emph{{SU}}(2)^2$-{I}nvariant ${G}_2$-{I}nstantons.
\newblock {\em Mathematische Annalen}, 371:961–1011, 2018.

\bibitem{Oli}
Goncalo Oliveira.
\newblock {C}alabi-{Y}au {M}onopoles for the {S}tenzel {M}etric.
\newblock {\em Communications in Mathematical Physics}, 341:699–728, 2016.

\bibitem{Pat}
Giorgio Patrizio and Pit-Mann Wong.
\newblock {S}tein {M}anifolds with {C}ompact {S}ymmetric {C}enter.
\newblock {\em Mathematische Annalen}, 289(1):355--382, Mar 1991.

\bibitem{Sal}
Dietmar~A. Salamon and Thomas Walpuski.
\newblock {N}otes on the {O}ctonions.
\newblock 2010.

\bibitem{Stenz}
Matthew~B. Stenzel.
\newblock {R}icci-{F}lat {M}etrics on the {C}omplexification of a {C}ompact
  {R}ank {O}ne {S}ymmetric {S}pace.
\newblock {\em {M}anuscripta {M}athematica}, 80(1):151--163, Dec 1993.

\bibitem{Uh}
K.~Uhlenbeck and S.~T. Yau.
\newblock {O}n the {E}xistence of {H}ermitian-{Y}ang-{M}ills {C}onnections in
  {S}table {V}ector {B}undles.
\newblock {\em Communications on Pure and Applied Mathematics},
  39(S1):S257--S293, 1986.

\bibitem{wang}
Hsien-chung Wang.
\newblock {O}n {I}nvariant {C}onnections {O}ver a {P}rincipal {F}ibre {B}undle.
\newblock {\em Nagoya Math. J.}, 13:1--19, 1958.

\end{thebibliography}

\end{document}